\documentclass[10pt,twoside,leqno]{amsart}
\pdfoutput=1
\usepackage{amsfonts}
\usepackage{amsmath}
\usepackage{mathtools}
\usepackage{amsthm}
\usepackage{amssymb}
\usepackage{pifont}
\newcommand{\cmark}{\ding{51}}

\usepackage{mathrsfs}
\usepackage[numbers]{natbib}
\usepackage[fit]{truncate}
\usepackage{geometry}
\usepackage{hyperref}
\usepackage{bm}
\usepackage{tikz-cd}
\usepackage[makeroom]{cancel}
\usepackage{xcolor,colortbl}
\usepackage{hhline}
\usepackage{float}
\usepackage{enumitem} 
\usepackage{stackengine}
\newcommand\xrowht[2][0]{\addstackgap[.5\dimexpr#2\relax]{\vphantom{#1}}}
\usepackage{array, multirow}
\newcommand{\R}{\mathbb{R}}
\newcommand{\C}{\mathbb{C}}

\newcommand{\cm}[1]{\ignorespaces}
\newcommand{\eps}{\varepsilon}
\newcommand{\nperp}{\mathfrak{n}^{\perp_g}}
\newcommand{\frs}[1]{\mathfrak{s}_{#1}}

\theoremstyle{plain}
\newtheorem{theorem}{Theorem}[section]
\newtheorem*{theorem*}{Theorem}
\newtheorem{corollary}[theorem]{Corollary}
\newtheorem{lemma}[theorem]{Lemma}

\newtheorem{proposition}[theorem]{Proposition}
\newtheorem*{proposition*}{Proposition}

\theoremstyle{definition}

\newtheorem{example}[theorem]{Example}

\newtheorem{remark}[theorem]{Remark}
\numberwithin{equation}{section}

\begin{document}

\title[Hermitian structures on a class of almost nilpotent Lie algebras]{Hermitian structures on a class \\ of almost nilpotent solvmanifolds}

\author{Anna Fino}
\address[A. Fino]{Dipartimento di Matematica ``G. Peano''\\
	Universit\`a di Torino\\
	Via Carlo Alberto 10\\
	10123 Torino, Italy} \email{annamaria.fino@unito.it}

\author{Fabio Paradiso}
\address[F. Paradiso]{Dipartimento di Matematica ``G. Peano''\\
	Universit\`a di Torino\\
	Via Carlo Alberto 10\\
	10123 Torino, Italy} \email{fabio.paradiso@unito.it}


\begin{abstract} 
In this paper we  investigate  the existence of   invariant SKT, balanced  and  generalized K\"ahler structures  on compact quotients $\Gamma \backslash G$,  where $G$ is an almost  nilpotent Lie group whose nilradical  has one-dimensional commutator and $\Gamma$ is a lattice of $G$. 
We   first obtain a characterization of Hermitian   almost nilpotent Lie algebras $\mathfrak{g}$   whose nilradical $\mathfrak{n}$  has one-dimensional commutator   and    a classification result in real dimension six. Then, we study the ones admitting SKT and  balanced structures  and we examine the behaviour of such structures under flows. In particular, we construct new examples of  compact   SKT manifolds. 

Finally,  we prove  some  non-existence results for generalized K\"ahler structures  in  real dimension six.     In higher dimension  we construct     the  first examples  of non-split generalized K\"ahler structures  (i.e., such that the associated complex structures do not commute) on  almost abelian   Lie algebras. This leads  to  new   compact  (non-K\"ahler) manifolds  admitting    non-split generalized K\"ahler structures.

\end{abstract}

\maketitle


\section{Introduction}
Let $(M,J)$ be a complex manifold of complex dimension $n$. A  $J$-Hermitian metric $g$ on $(M, J)$ 
is called \emph{strong K\"ahler with torsion} (SKT for short) or \emph{pluriclosed} if its fundamental $2$-form $\omega \coloneqq g(J\cdot, \cdot)$ is $dd^c$-closed, where 
$d^c\coloneqq i(\overline\partial - \partial)$ is the real Dolbeault operator.
  The SKT condition, first introduced in \cite{Bis}, is one of the most studied generalizations of the K\"ahler condition (see for example \cite{FT1, FV1, FV2, GGP, ST, ST1} and the references therein).  SKT metrics also crop up in the setup of generalized geometry: following \cite{Gua}, a \emph{generalized K\"ahler structure} on an even-dimensional smooth manifold $M$ is defined as a pair $(\mathcal{J}_1,\mathcal{J}_2)$ of complex structures on the vector bundle $TM \oplus T^*M$ which are orthogonal with respect to the natural split-signature inner product $\left<\cdot,\cdot\right>$ on $TM \oplus T^*M$,
\[
\left< X+ \xi, Y+ \eta \right> =\frac{1}{2} (\xi(Y)+\eta(X)),\quad X+\xi,Y+\eta \in TM \oplus T^*M,
\] 
integrable with respect to the ($H$-twisted)  Courant bracket
\[
[X+\xi,Y+\eta]_H=[X,Y]+\mathcal{L}_X \eta - \iota_X \xi + \iota_Y \iota_X H,\quad X+\xi,Y+\eta \in \Gamma(TM \oplus T^*M),
\] 
where $H$ is a closed $3$-form on $M$,  and such that the inner product $\left< \mathcal{J}_1 \cdot, \mathcal{J}_2 \cdot \right>$ is positive-definite. Generalized K\"ahler structures have been extensively studied in literature, see for example \cite{AGG, AG, DM,  FT, Hit}

It was proven in \cite{Gua} that a generalized K\"ahler structure is actually equivalent to a triple $(J_+,J_-,g)$, where $J_\pm$ are  complex structures on $M$ and $g$ is a $J_\pm$-Hermitian SKT metric such that
$d^c_+\omega_+ + d^c_- \omega_-=0,$
where $\omega_\pm \coloneqq g(J_\pm \cdot, \cdot)$ and $d^c_\pm$ is the real Dolbeault operator for $J_\pm$.

A generalized K\"ahler structure $(J_+,J_-,g)$ is called \emph{split} if $[J_+,J_-]=0$.  In the  non-split case,  by  \cite{Hit}  one has that the $(2,0)$-part (with respect to $J_\pm$) of $[J_+,J_-]g^{-1} \in \Gamma(\Lambda^2TM)$ always defines a \emph{holomorphic Poisson structure} $\sigma \in  \Gamma(\Lambda^2 T^{1,0}M)$, namely  $\sigma$  is holomorphic with respect to $J_\pm$, i.e., $\overline\partial_{\pm}  \sigma=0$, and \emph{Poisson}, that is, $[\sigma,\sigma]=0$, where $\overline\partial_{\pm}$ is the Cauchy-Riemann operator  associated to $J_{\pm}$ and $[ \cdot, \cdot]$ is the Schouten bracket.

Another special class of Hermitian structures is provided by \emph{balanced} structures, namely Hermitian structures $(J,g)$ whose  fundamental form  $\omega$ is coclosed or equivalently satisfying $d\omega^{n-1}=0$. 

A natural  setting for the study of  SKT  and balanced structures is provided   by compact quotients $\Gamma \backslash G$  of simply connected Lie groups  $G$ by lattices $\Gamma$  endowed with an invariant complex structure, i.e., with a complex structure induced by a  complex structure on  the Lie algebra $\mathfrak{g}$ of $G$.
This allows to simplify the problem, by reducing to  structures at the level of   Lie algebras.

In the nilpotent setting,  SKT structures were studied in \cite{FPS}, where $6$-dimensional nilpotent Lie algebras admitting SKT structures were classified, and in \cite{EFV}.

SKT structures on special classes of solvable Lie groups were instead studied in \cite{FOU,AL,FP1,FS}, with \cite{FT, FP1} focusing on generalized K\"ahler structures on \emph{almost abelian} Lie algebras, namely solvable Lie algebras admitting a codimension one abelian ideal.
Up to now, in literature, these are the  only examples of non-K\"ahler solvable Lie algebras admitting generalized K\"ahler structures, with the first example having been found in \cite{FT} and a full classification in  real dimension six obtained in \cite{FP1}. Moreover, all the known  solvable examples are split.

In \cite{FOU}, it was proven that, up to isomorphism, only two $6$-dimensional  SKT unimodular non-nilpotent  Lie algebras admit  a closed $(3,0)$-form: one  is almost abelian, while the other is \emph{almost nilpotent}, namely its  nilradical $\mathfrak{n}$  has codimension one.  In particular, $\mathfrak{n} \cong \mathfrak{h}_3 \oplus \R^2$, with $\mathfrak{h}_3$ denoting the $3$-dimensional Heisenberg Lie algebra, hence its commutator $\mathfrak{n}^1 \coloneqq [\mathfrak{n},\mathfrak{n}]$ is one-dimensional.

Motivated by this, in this paper we concentrate our attention on almost nilpotent Lie algebras whose nilradical $\mathfrak{n}$  has one-dimensional commutator $\mathfrak{n}^1$. Such Lie algebras can be represented as the semidirect product $\mathfrak{n} \rtimes_D \R$, for a nilpotent Lie algebra $\mathfrak{n}$ with $\dim \mathfrak{n}^1=1$ and some $D \in \text{Der}(\mathfrak{n})$. Special focus is placed on \emph{strongly unimodular} Lie algebras, since this is a necessary condition in order for the corresponding simply connected Lie group to admit compact quotients by lattices (\cite{Gar}). 

Exploiting the fact that, on such Lie algebras, the dimension of $\mathfrak{n}^1$ and the codimension of $\mathfrak{n}$ are both equal to one,  in Section \ref{sec_Hermitian}  we study Hermitian structures $(J,g)$ on them, depending on the relative positions of the two lines $J \mathfrak{n}^1$ and $\mathfrak{n}^{\perp_g}$, where $\mathfrak{n}^{\perp_g}$ denotes the orthogonal complement of $\mathfrak{n}$ in $\mathfrak{g}$.  In Section \ref{casssixsection} we apply  the previous results to classify, up to isomorphisms,   $6$-dimensional  strongly unimodular almost nilpotent Lie algebras whose nilradical $\mathfrak{n}$  has one-dimensional commutator $\mathfrak{n}^1$ and admitting complex structures.

For the SKT structures in Section \ref{sectionSKTstruct} we find characterizations and classification results in  real dimension six for the two extremal cases, namely the one where the two lines coincide ($J \mathfrak{n}^1 =\mathfrak{n}^{\perp_g}$) and the one where they are orthogonal ($J \mathfrak{n}^1 \subset \mathfrak{n}$).

The SKT condition is preserved by  the \emph{pluriclosed flow}, which  is   a parabolic flow, introduced  in   \cite{Str, ST, ST1},  evolving SKT metrics on a complex manifold $(M,J)$ and  reducing to the K\"ahler-Ricci flow for K\"ahler initial data.
In particular, this flow was studied on nilpotent Lie groups in \cite{EFV1,AL} and on almost abelian Lie groups in \cite{AL}. In  Section \ref{sec_plflow}  we analyze its behaviour on almost nilpotent Lie groups with nilradical having one-dimensional commutator.

In Section \ref{balancedstructsection} we investigate  the existence of balanced structures and  in Section \ref{balancedflowsection}, we also study their behaviour under the parabolic flow introduced in \cite{BV} as a generalization of the Calabi flow (see \cite{Cal}).

Finally, in Section  \ref{sectionGKsection} we study the existence of generalized K\"ahler structures and  we construct  the first examples of    (non-K\"ahler)   compact    solvmanifolds admitting  non-split generalized K\"ahler structures.

\smallskip

\emph{Acknowledgements}. The authors would like to thank Beatrice Brienza for useful remarks regarding Theorem \ref{gkexamples} and an anonymous referee for meaningful comments and remarks. The paper is supported by Project PRIN 2017 \lq \lq Real and complex manifolds: Topology, Geometry and Holomorphic Dynamics” and by GNSAGA of INdAM.

\section{Characterization of Hermitian Lie algebras}  \label{sec_Hermitian}

Recall that an almost complex structure $J$ on $2n$-dimensional real Lie algebra  $\mathfrak{g}$ is  an endomorphism of $\mathfrak{g}$ such that  $J^2 =- \text{Id}_{\mathfrak{g}}$. 

An  almost-Hermitian metric $g$  on $(\mathfrak{g}, J )$  is a (positive definite) inner product which is $J$-orthogonal, i.e., such that $g(JX,JY)=g(X,Y),$ for every $X,Y \in  \mathfrak{g}$.

If  $J$ is integrable, or equivalently if    its  Nijenhuis tensor 
$$
N (X, Y ) = [J X, J Y ]-J [J X, Y ] - J [X, J Y ] - [X, Y ],  \, \quad  X, Y \in \mathfrak{g}, 
$$
vanishes,  then $(J,g)$ is a Hermitian structure on $\mathfrak{g}$ and $(\mathfrak{g}, J, g)$ is said to be Hermitian.

In this section we shall study  Hermitian  almost nilpotent Lie algebras  $(\mathfrak{g}, J, g)$ whose nilradical  $\mathfrak{n}$ has one-dimensional commutator.
Recall that the nilradical  $\mathfrak{n}$ of a solvable Lie algebra $\mathfrak{g}$ is the maximal nilpotent  ideal of $\mathfrak{g}$ and  a solvable  Lie algebra $\mathfrak{g}$  is   called {\em almost nilpotent}  if its  nilradical  $\mathfrak{n}$ has codimension one,   or equivalently   if   $\mathfrak{g}$    is isomorphic to the  semidirect product $\mathfrak{n} \rtimes_D \R$, for  a  derivation $D$ of $\mathfrak{n}$.

 If we denote by  
\[
\mathfrak{n}^0\coloneqq \mathfrak{n}, \,  \mathfrak n^1  \coloneqq [\mathfrak{n}, \mathfrak{n}], \ldots, \mathfrak{n}^l \coloneqq [\mathfrak{n},\mathfrak{n}^{l-1}], \,  l \geq 2,
\]
the terms of the descending central series of $\mathfrak{n}$,  we recall that a solvable  Lie algebra $\mathfrak{g}$ is {\em strongly unimodular}  if $\operatorname{tr}  \left ( \text{ad}_X \rvert_{\mathfrak{n}^l / \mathfrak{n}^{l+1}} \right )=0$,  for every $X \in \mathfrak{g}$ and every $l \in \mathbb{N}$.  

Observe that, if the nilradical $\mathfrak{n}$ of $\mathfrak{g}$ is $r$-step nilpotent, then one has 
\[
\operatorname{tr}( \text{ad}_X) = \sum_{l=0}^{r-1} \operatorname{tr}  \left (\text{ad}_X \rvert_{\mathfrak{n}^l / \mathfrak{n}^{l+1}}\right ), \quad X \in \mathfrak{g}.
\]
In particular, a strongly unimodular solvable Lie algebra is always unimodular, i.e., $\operatorname{tr} \text{ad}_X =0$, for every $X \in \mathfrak g$.

Since the  Lie algebra  $\mathfrak{g}$ of  a simply connected solvable Lie group  admitting  compact quotients by lattices  must be strongly unimodular  (see \cite{Gar}),  we shall often focus on strongly unimodular almost nilpotent   Lie algebras.

Let  $\mathfrak{g}$  be an  almost nilpotent   Lie algebra and  assume  that the commutator   $\mathfrak{n}^1$ of the nilradical $\mathfrak{n}$  has dimension one. Then,  by \cite{BD, DDI},   $\mathfrak{n}$ has to be a central extension of a Heisenberg Lie algebra  $\mathfrak{h}_{2l+1}$,  i.e., a direct sum  $\mathfrak{h}_{2l+1} \oplus \R^h$   with  $l,h$  positive integers.  Recall that  $\mathfrak{h}_{2l+1} = {\mathbb R}\left<x_1, \ldots,x_l,y_1, \ldots, y_l, z\right>$   is the $2$-step nilpotent Lie algebra with  structure equations  $[x_k,y_k] = z$,  for  every $k = 1,\ldots,l$.
 
 Moreover, $\mathfrak{g}$ is strongly unimodular if and only if $\operatorname{tr} \text{ad}_X \rvert_{\mathfrak{n}^l / \mathfrak{n}^{l+1}}=0$, for $l =0,1$.
 
Now, let $(J,g)$ be an almost-Hermitian structure on $\mathfrak{g}$ and denote by $\hat{\theta} \in [0,\frac{\pi}{2}]$ the (smaller) angle formed by the two lines $J \mathfrak{n}^1$ and $\mathfrak{n}^{\perp_g}$ inside $\mathfrak{g}$, where by $\mathfrak{n}^{\perp_g}$ we denote the orthogonal complement of $\mathfrak{n}$ in $\mathfrak{g}$.  We can  then distinguish the following three cases:
\begin{enumerate}
\item $\hat\theta = 0$, namely $J \mathfrak{n}^1= \mathfrak{n}^{\perp_g}$,
\item $\hat\theta = \frac{\pi}{2}$, namely $J \mathfrak{n}^1  \subset  \mathfrak{n}$,
\item $\hat\theta \in (0,\frac{\pi}{2})$, corresponding to the cases where $J \mathfrak{n}^1  \cap   \mathfrak{n}^{\perp_g}= \{ 0 \}$ and  $J \mathfrak{n}^1  \cap  \mathfrak{n}  = \{ 0 \}$.
\end{enumerate}

We  will now focus on the first two cases, where we can find a characterization  if $J$ is integrable.

\subsection{Case (1)} \label{Hermitian_perp_sec} Let us assume $J \mathfrak{n}^1 = \mathfrak{n}^{\perp_g}$ and denote by $\mathfrak{k}_1 \coloneqq \mathfrak{n} \cap J\mathfrak{n}$ the maximal $J$-invariant subspace of $\mathfrak{n}$.
One can find a unitary basis $\{e_1,\ldots,e_{2n}\}$ of $(\mathfrak{g}, J, g)$ such that $J \nperp =\mathfrak{n}^1= \R \left< e_1 \right>$, $\mathfrak{k}_1 = \mathbb R\left<e_2,\ldots,e_{2n-1}\right>$, $\nperp=\R \left< e_{2n} \right >$ and
$Je_1=e_{2n}$, $Je_{2l}=e_{2l+1}$, $l=1,\ldots,n-1$. We will call this basis an adapted  one.

In this adapted  basis, the Lie bracket  on $\mathfrak{g}$ is  given by
\[
[e_{2n},X]=B(X), \qquad	[Y,Z]=-\eta(Y,Z)e_1,  \qquad   X \in \mathfrak{n},  \, \,  Y,Z \in \mathfrak{k}_1,
\]
where  $B$ is a derivation of $\mathfrak{n}$  and $\eta$ is a non-zero $2$-form on $\mathfrak{k}_1$. 
One can observe that $B$ must preserve $\mathfrak{n}^1=\R \left< e_1 \right>$, since the fact that $B$ is a derivation implies $B \mathfrak{n}^1\subseteq[B\mathfrak{n},\mathfrak{n}]+[\mathfrak{n},B\mathfrak{n}] \subseteq \mathfrak{n}^1$. Then, with respect to the basis $\{ e_l \}$, the derivation $B$ may be written as
\begin{equation} \label{B}
B=\begin{pmatrix} a & \beta \\ 0 & A \end{pmatrix},\quad a \in \R, \quad \beta \in \mathfrak{k}_1^*,\quad A \in \mathfrak{gl}(\mathfrak{k}_1).
\end{equation}

Therefore   the almost-Hermitian Lie algebra  $(\mathfrak{g},J,g)$  is completely  determined by the algebraic  data 
\[
(a,\beta,A,\eta) \in \R \times \mathfrak{k}_1^* \times \mathfrak{gl}(\mathfrak{k}_1) \times \Lambda^2 \mathfrak{k}_1^*.
\]

Indeed, we can consider the vector space $\R^{2n}= \mathbb R \left<e_1,\ldots,e_{2n}\right>$ endowed with the Lie algebra structure provided by a Lie bracket operation
\[
\mu \in  \Lambda^2 (\R^{2n})^* \otimes \R^{2n}
\]
making it isomorphic, as a Lie algebra, to the semidirect product $\mathfrak{n} \rtimes_B \R \left < e_{2n} \right >$, with $\mathfrak{n} = \mathbb R \left< e_1,\ldots,e_{2n-1} \right>$ having Lie bracket
\[
[Y,Z]=-\eta(Y,Z)e_1,  \quad  Y, Z \in \mathfrak{n},
\]
and $B$ as in \eqref{B}. We denote the corresponding Lie bracket operation on $\R^{2n}$ by $\mu(a,\beta,A,\eta)$ to keep track of the algebraic data. Then, we can consider the almost-Hermitian structure given by $Je_1=e_{2n}$, $Je_{2l}=e_{2l+1}$, $l=1,\ldots,n-1$ and $g = \sum_{l=1}^{2n} (e^l)^2$. The resulting almost-Hermitian Lie algebra, which we shall denote by $(\mu(a,\beta,A,\eta),J,g)$, omitting the ambient vector space $\R^{2n}$ (often also denoted by $\mathfrak{g}$ when no confusion arises), is equivalent to the original almost-Hermitian Lie algebra $(\mathfrak{g},J,g)$.

\begin{lemma} \label{Jacobi_perp}
The data $(a,\beta,A,\eta)$ represent a Lie algebra $\mathfrak{n} \rtimes_B \R \left <e_{2n} \right >$   if and only if 
\begin{equation} \label{lie_perp}
A^*\eta=a\eta,
\end{equation}
where  $A^* \eta$ is  the $2$-form  on $\mathfrak{k}_1$ defined by 
$$
(A^*\eta)(X,Y)  = \eta(A(X),Y)+ \eta (X, A(Y)), \, X , Y\in \mathfrak{k}_1.
$$

\end{lemma}
\begin{proof}
To prove this, we have to determine the conditions under which $B=\text{ad}_{e_{2n}}\rvert_{\mathfrak{n}}$ is actually a derivation of $\mathfrak{n}$. Let $X,Y \in \mathfrak{k}_1$ be arbitrary. Then,
\[
B([e_1,X])=[Be_1,X]+[e_1,BX]
\]
already holds since $[e_1,\mathfrak{n}]=0$ and $Be_1=a \, e_1$. Instead
\[
B([X,Y])=B(-\eta(X,Y)e_1)=-a\eta(X,Y) e_1,
\]
while
\[
[BX,Y]+[X,BY]=[AX,Y]+[X,AY]=-(\eta(AX,Y)+\eta(X,AY))e_1,
\]
from which it follows that $B$ is a derivation of $\mathfrak{n}$ if and only if \eqref{lie_perp} holds.
\end{proof}

\begin{remark}
We note that the algebraic datum \lq \lq$a \in \R$'' is redundant, since it is implicitly determined by $A$ and $\eta$. Anyway, we shall not omit it for major clarity and completeness. Moreover, in what follows, the algebraic data $(a,\beta,A,\eta)$ are assumed to satisfy \eqref{lie_perp}.
\end{remark}

\begin{remark} \label{strunim_perp}
One can see that the data $(a,\beta,A,\eta)$ (satisfying \eqref{lie_perp}) correspond to a strongly unimodular Lie algebra if and only if $a=\operatorname{tr}A=0$. This follows by observing that $\mathfrak{n}^0/\mathfrak{n}^1 \cong \mathfrak{k}_1 = \mathbb R \left<e_2,\ldots,e_{2n-1}\right>$, $\mathfrak{n}^1/\mathfrak{n}^2=\mathfrak{n}^1=\R \left < e_1 \right >$ and that $[e_{2n},e_1]=a  e_1$ and $\text{ad}_{e_{2n}}\rvert_{\mathfrak{k}_1}$ is represented by the matrix $A$.
\end{remark}

\begin{proposition} \label{complex}  
$(\mu(a,\beta,A,\eta),J,g)$ is Hermitian, i.e., $J$ is integrable, if and only if
\begin{equation} \label{integrable_perp}
[A,J \rvert_{\mathfrak{k}_1}]=0,\quad \beta=0, \quad \eta \in \Lambda^{1,1}\mathfrak{k}_1^*,
\end{equation}
where   $\Lambda^{1,1}\mathfrak{k}_1^*$ denotes the space of $(1,1)$-forms on $\mathfrak{k}_1$.
\end{proposition}
\begin{proof}
Exploiting that $N(J\cdot,J\cdot)=- N (\cdot, \cdot)$, one can see that it is enough to check the vanishing of $N(e_1,X)$ and $N(X,Y)$, for $X,Y \in \mathfrak{k}_1$: one computes
\begin{align*}
	N(e_1,X)&=J[e_{2n},X]-[e_{2n},JX]\\
	        &=J(\beta(X)e_1+AX)-\beta(JX)e_1-A(JX)\\
	        &=\beta(JX)e_1+[A,J \rvert_{\mathfrak{k}_1}](X)+\beta(X)e_{2n},
\end{align*}
which vanishes for all $X \in \mathfrak{k}_1$ if and only if $[A,J \rvert_{\mathfrak{k}_1}]=0$ and $\beta=0$. On the other hand,
\begin{align*}
	N(X,Y)&=(\eta(JX,JY)-\eta(X,Y))e_1-(\eta(JX,Y)+\eta(X,JY))e_{2n}, \quad X,Y \in \mathfrak{k}_1,
\end{align*}
from which the condition $\eta \in \Lambda^{1,1}\mathfrak{k}_1^*$ follows.
\end{proof}

\subsection{Case (2)} \label{Hermitian_sub_sec}
Suppose $J \mathfrak{n}^1 \subset \mathfrak{n}$.
In this case, let us define
\[
\mathfrak{k}_1 \coloneqq \mathfrak{n} \cap J\mathfrak{n},\quad \mathfrak{k}_2 \coloneqq (\mathfrak{n}^1)^{\perp_g} \cap \mathfrak{n},\quad \mathfrak{k}_3 \coloneqq \mathfrak{k}_2 \cap J\mathfrak{k}_2.
\]
Then, one can find an adapted unitary basis $\{e_1,\ldots,e_{2n}\}$ of $\mathfrak{g}$ such that $\mathfrak{n}= \mathbb R \left<e_1,\ldots,e_{2n-1}\right>$, $\mathfrak{n}^1=\R \left <e_1 \right >$ and $Je_{2j-1}=e_{2j}$, $j=1,\ldots,n$. Then, we have $\mathfrak{k}_1=\ \mathbb R \left<e_1,\ldots,e_{2n-2}\right>$, $\mathfrak{k}_2= \mathbb R \left<e_2,\ldots,e_{2n-1}\right>$, $\mathfrak{k}_3= \mathbb R \left<e_3,\ldots,e_{2n-2}\right>$ (with $\mathfrak{k}_3=\{0\}$ when $n=2$). 

Similarly to the previous case, the Lie brackets are given by
\[
[e_{2n},X]=B(X), \qquad	[Y,Z]=-\eta(Y,Z)e_1,  \qquad   X \in \mathfrak{n},  \, \,  Y,Z \in \mathfrak{k}_2,
\]
where  $B$ is a derivation of $\mathfrak{n}$  and $\eta$ is a  non-zero $2$-form on $\mathfrak{k}_2$. 

Again, $B$ must preserve $\mathfrak{n}^1=\R \left < e_1 \right >$. Then, with respect to the decomposition $$\mathfrak{n}=\R \left <e_1 \right> \oplus \R \left < e_2 \right > \oplus \mathfrak{k}_3 \oplus \R  \left < e_{2n-1}\right >,$$ $B$ may be written in matrix form as
\begin{equation} \label{B_sub}
B=\left( \begin{array}{c|c|c|c}
	     a_1 & p_1 & \phantom{\;\;\;\;} \alpha \phantom{\;\;\;\;} & v_1 \\ \hline
	     0 & a_2 & \gamma & v_2 \\ \hline
	     0 \vphantom{\begin{pmatrix} 0 \\ 0 \\ 0 \end{pmatrix}} & w & A & v \\ \hline
	     0 & p_2 & \beta & a
	     \end{array}
  \right), \quad a,a_1,a_2,p_1,p_2,v_1,v_2 \in \R,\; v,w \in \mathfrak{k}_3,\; \alpha,\beta,\gamma \in \mathfrak{k}_3^*,\; A \in \mathfrak{gl}(\mathfrak{k}_3),
\end{equation}
while
\begin{equation} \label{eta_sub}
\eta=\xi+\delta \wedge e^2 + \nu \wedge e^{2n-1} + \lambda\, e^2 \wedge e^{2n-1},\quad \xi \in \Lambda^2 \mathfrak{k}_3^*, \; \delta,\nu \in \mathfrak{k}_3^*,\; \lambda \in \R.
\end{equation}
We denote the whole set of algebraic data with
\begin{equation} \label{Xi}
\Xi=(a,a_1,a_2,p_1,p_2,v_1,v_2,v,w,\alpha,\beta,\gamma,A,\xi,\delta,\nu,\lambda).
\end{equation}

Analogously to case (1), we can consider the vector space $\R^{2n}$ with a bracket operation $\mu(\Xi)$ determined by the above algebraic data. The Lie algebra $(\R^{2n},\mu(\Xi))$ is isomorphic to the semidirect product $\mathfrak{n} \rtimes_B \R \left <e_{2n} \right >$, $\mathfrak{n}= \mathbb R \left< e_1,\ldots,e_{2n-1} \right>$. Endowing this Lie algebra with the almost-Hermitian structure $(J,g)$ defined by $Je_{2l -1}=e_{2l}$, $l=1,\ldots,n$, $g=\sum_{l=1}^{2n} (e^l)^2$, yields an almost-Hermitian Lie algebra, denoted by $(\mu(\Xi),J,g)$, isomorphic to the original almost-Hermitian Lie algebra having algebraic data $\Xi$. We can then reduce to studying model almost-Hermitian Lie algebras $(\mu(\Xi),J,g)$.

\begin{proposition} \label{prop_Lie}
The algebraic data $\Xi$ in \eqref{Xi} represent a Lie algebra $\mathfrak{n} \rtimes_B \R \left <e_{2n} \right>$ if and only if the following conditions hold:
\begin{align}{}
\label{Lie1}	&(a_2-a_1)\delta+A^*\delta + p_2\nu -\lambda \beta - \iota_w\xi=0,\\
\label{Lie2}	&\lambda(a+a_2-a_1)+\nu(w)+\delta(v)=0,\\
\label{Lie3}	&a_1\xi-A^*\xi+\gamma \wedge \delta + \beta \wedge \nu=0,\\
\label{Lie4}	&(a-a_1)\nu + A^*\nu + \lambda \gamma + v_2 \delta - \iota_v \xi =0.
\end{align}
\end{proposition}
\begin{proof}
We only need to ensure that $B$ is a derivation of $\mathfrak{n}$, namely that $$B[X,Y]=[BX,Y]+[X,BY]$$ for all $X,Y \in \mathfrak{n}$. It is enough to impose this condition for the following pairs of vectors:

i) $X=e_2$, $Y \in \mathfrak{k}_3$: we have
\begin{align*}
B[e_2,Y]&=B(\delta(Y)e_1)=a_1\delta(Y)e_1, \\
[Be_2,Y]&=[p_1e_1+a_2 e_2+w+p_2e_{2n-1},Y]=(a_2\delta(Y)-\xi(w,Y)+p_2\nu(Y))e_1,\\
[e_2,BY]&=[e_2,\alpha(Y)e_1+\gamma(Y)e_2+AY+\beta(Y)e_{2n-1}]=(\delta(AY)-\lambda\beta(Y))e_1,
\end{align*}
from which \eqref{Lie1} follows,

ii) $X=e_2$, $Y=e_{2n-1}$: we have
\begin{align*}
&B[e_2,e_{2n-1}]=B(-\lambda e_1)=-\lambda a_1 e_1, \\
&[Be_2,e_{2n-1}]=[p_1e_1+a_2e_2+w+p_2e_{2n-1},e_{2n-1}]=-(\lambda a_2 + \nu(w))e_1, \\
&[e_2,Be_{2n-1}]=[e_2,v_1e_1+v_2e_2+v+ae_{2n-1}]=(\delta(v)-\lambda a)e_1,
\end{align*}
which gives us \eqref{Lie2},

iii) $X,Y \in \mathfrak{k}_3$: we have
\begin{align*}
B[X,Y]=&B(-\xi(X,Y)e_1)=-\xi(X,Y)a_1e_1,\\
[BX,Y]=&[\alpha(X)e_1+\gamma(X)e_2+AX+\beta(X)e_{2n-1},Y]\\=&(\alpha(X)\delta(Y)-\xi(AX,Y)+\beta(X)\nu(Y))e_1,\\
[X,BY]=&-[BY,X]=-(\alpha(Y)\delta(X)-\xi(AY,X)+\beta(Y)\nu(X))e_1,
\end{align*}
so that \eqref{Lie3} follows.

iv) $X \in \mathfrak{k}_3$, $Y=e_{2n-1}$: we have
\begin{align*}
&B[X,e_{2n-1}]=B(-\nu(X)e_1)=-a_1\nu(X)e_1,\\
&[BX,e_{2n-1}]=[\alpha(X)e_1+\gamma(X)e_2+AX+\beta(X)e_{2n-1},e_{2n-1}]=-(\lambda \gamma(X) + \nu(AX))e_1,\\
&[X,Be_{2n-1}]=[X,v_1e_1+v_2e_2+v+ae_{2n-1}]=-(v_2\delta(X)-\xi(v,X)+a\nu(X))e_1,
\end{align*}
which yields \eqref{Lie4}.
\end{proof}

In what follows, we implicitly assume that the algebraic data $\Xi$ in \eqref{Xi} satisfy \eqref{Lie1}--\eqref{Lie4}. 

\begin{remark}
Analogously to Remark \ref{strunim_perp}, we have that the algebraic data $\Xi$ define a strongly unimodular Lie algebra if and only if $a_1=a+a_2+\operatorname{tr}A=0$.
\end{remark}

\begin{proposition}
$(\mu(\Xi),J,g)$ is Hermitian, i.e., $J$ is integrable, if and only if
\begin{equation} \label{sub_integrable}
\begin{alignedat}{3}
&p_1=p_2=0, \qquad& &\beta=\delta=0, \qquad& &\lambda=a_2-a_1, \\
&[A,J\rvert_{\mathfrak{k}_3}]=0, \qquad& &w=0, \qquad& &\nu=\gamma-J\alpha.
\end{alignedat}
\end{equation}
\end{proposition}
\begin{proof}
One can see that it is enough to check the vanishing of $N(e_1,X)$, $N(e_1,e_{2n-1})$ and $N(X,e_{2n-1})$, for a generic $X \in \mathfrak{k}_3$, where $N$ denotes the Nijenhuis tensor associated with $J$. An explicit computation yields
\begin{align*}
	N(e_1,X)=&\delta(X)e_2-\delta(JX)e_1,\\
	N(e_1,e_{2n-1})=&p_1e_1-(\lambda-a_2+a_1)e_2+w+p_2e_{2n-1},\\
	N(X,e_{2n-1})=&(-\nu(X)+\gamma(X)+\alpha(JX))e_1+(-\nu(JX)-\alpha(X)+\gamma(JX))e_2 \\&+ (AJ-JA)(X)+\beta(JX)e_{2n-1} - \beta(X)e_{2n},
\end{align*}
and the claim follows.
\end{proof}

\section{Classification in real  dimension six} \label{casssixsection}

To classify   almost nilpotent Lie algebras $\mathfrak{g}$ with nilradical $\mathfrak{n}$ satisfying $\dim \mathfrak{n}^1 = 1$  and admitting complex structures it is sufficient to consider only the cases (1) and (2). Indeed,  if  $\mathfrak{g}$ has  a complex structure  $J$, then     either  $J \mathfrak{n}^1 \subset \mathfrak{n}$ or  $J \mathfrak{n}^1 \not \subset \mathfrak{n}$.   If $J \mathfrak{n}^1 \subset \mathfrak{n}$, then we reduce to case (2).  If  $J \mathfrak{n}^1 \not \subset \mathfrak{n}$, one considers a $J$-Hermitian  metric   such that  $J \mathfrak{n}^1$ is orthogonal to  $\mathfrak{n}$ and we reduce to the case (1).

Given an almost nilpotent Lie algebra $\mathfrak{g}$ with nilradical $\mathfrak{n}$ satisfying $\dim \mathfrak{n}^1 = 1$ and  an endomorphism $F \in \mathfrak{gl}(\mathfrak{\mathfrak{n}})$ such that $F(\mathfrak{n}^1)\subseteq \mathfrak{n}^1$, one can define  a natural map  $\pi (F) \in  \mathfrak{gl}(\mathfrak{n}/\mathfrak{n}^1)$, as
 $$\pi(F)(X+[\mathfrak{n},\mathfrak{n}])\coloneqq F(X)+[\mathfrak{n},\mathfrak{n}],  \, X \in \mathfrak{n}.$$

Hence, the adjoint action $\text{ad}\colon \mathfrak{g} \to \mathfrak{gl}(\mathfrak{g})$ induces two well-defined linear maps
\begin{alignat*}{2}
	\varphi_1 \colon \mathfrak{g}/\mathfrak{n} &\to \mathfrak{gl}(\mathfrak{n})/\text{ad}(\mathfrak{n}), &\qquad \qquad \varphi_2 \colon \mathfrak{g}/\mathfrak{n} &\to \mathfrak{gl}(\mathfrak{n}/\mathfrak{n}^1), \\
	X+\mathfrak{n} &\mapsto \text{ad}_X\rvert_{\mathfrak{n}} + \text{ad}(\mathfrak{n}), &\qquad \qquad X+\mathfrak{n} &\mapsto \pi(\text{ad}_X\rvert_{\mathfrak{n}}),
\end{alignat*}
where $\text{ad}(\mathfrak{n})$ is the Lie algebra of inner derivations of $\mathfrak{n}$.
 As $\mathfrak{g}/\mathfrak{n} \cong \R$,  $\varphi_1$ induces a well-defined endomorphism $B_\mathfrak{g}$ of $\mathfrak{n}$ up to inner derivations and non-zero rescalings, while $\varphi_2$ induces a well-defined endomorphism $A_\mathfrak{g}$ of $\mathfrak{n}/\mathfrak{n}^1$ up to non-zero rescalings.

We also observe that the Lie bracket of $\mathfrak{n}$ is completely determined by the following $\mathfrak{n}^1$-valued $2$-form:
\[
\eta_\mathfrak{g} \in \Lambda^2(\mathfrak{n}/\mathfrak{n}^1)^* \otimes \mathfrak{n}^1,\qquad \eta_\mathfrak{g}(X+\mathfrak{n}^1,Y+\mathfrak{n}^1)\coloneqq -[X,Y],\quad X,Y \in \mathfrak{n}.
\]
We shall study the existence of complex structures on $\mathfrak{g}$ assuming  that $\dim \mathfrak{g} =6$  and studying  separately the cases  $J \mathfrak{n}^1 \not\subset \mathfrak{n}$  and $J \mathfrak{n}^1 \subset \mathfrak{n}$.

In what follows, Lie algebras are denoted via structure equations: for example, the notation
\[
\mathfrak{s}_{4.7} \oplus \R^2=(f^{23},f^{36},-f^{26},0,0,0)
\]
means that the Lie algebra $\mathfrak{s}_{4.7} \oplus \R^2$ admits a basis $\{f_1,\ldots,f_6\}$ such that
\[
df^1=f^{23},\quad df^2=f^{36},\quad df^3=-f^{26},\quad df^4=df^5=df^6=0,
\]
where $f^{kl \ldots}$ is a short-hand for the wedge product $f^k \wedge f^l \wedge \ldots$ of $1$-forms.

We refer to Tables \ref{table-h3} and \ref{table-h5} for the full list of $6$-dimensional strongly unimodular almost nilpotent Lie algebras with one-dimensional commutator, written in the above notation.
 
\begin{theorem} \label{class_cpx_perp}
Let $\mathfrak{g}$ be a $6$-dimensional strongly unimodular almost nilpotent Lie algebra with nilradical $\mathfrak{n}$ satisfying $\dim \mathfrak{n}^1=1$. Then, $\mathfrak{g}$ admits complex structures $J$ satisfying $J \mathfrak{n}^1 \not\subset \mathfrak{n}$ if and only if it is isomorphic to one of the following: 
\begin{itemize}[leftmargin=0.6cm]
\item[]$\mathfrak{h}_3 \oplus \mathfrak{s}_{3.3}^0=(f^{23},0,0,f^{56},-f^{46},0)$, \smallskip
\item[]$\mathfrak{s}_{4.7}\oplus \R^2=(f^{23},f^{36},-f^{26},0,0,0)$,\smallskip
\item[]$\mathfrak{s}_{6.44}=(f^{23},f^{36},-f^{26},f^{26}+f^{56},f^{36}-f^{46},0)$, \smallskip
\item[]$\mathfrak{s}_{6.52}^{0,b}=(f^{23},f^{36},-f^{26},bf^{56},-bf^{46},0)$, $b>0$, \smallskip
\item[]$\mathfrak{s}_{6.159}=(f^{24}+f^{35},0,-f^{56},0,f^{36},0)$, \smallskip
\item[]$\mathfrak{s}_{6.162}^{1}=(f^{24}+f^{35},f^{26},f^{36},-f^{46},-f^{56},0)$, \smallskip
\item[]$\mathfrak{s}_{6.165}^a=(f^{24}+f^{35},af^{26}+f^{36},-f^{26}+af^{36},-af^{46}+f^{56},-f^{46}-af^{56},0)$, $a>0$,\smallskip
\item[]$\mathfrak{s}_{6.166}^a=(f^{24}+f^{35},-f^{46},-af^{56},f^{26},af^{36},0)$, $0<|a|\leq 1$, \smallskip
\item[]$\mathfrak{s}_{6.167}=(f^{24}+f^{35},-f^{36},-f^{26},f^{26}+f^{56},f^{36}-f^{46},0)$.
\end{itemize}
An explicit  example of  complex structure  is given by
$
 Jf_1=f_6,\,Jf_2=f_3,\, Jf_4=f_5, 
 $
for  all the Lie algebras in the previous list, with the exception of   $\mathfrak{s}_{6.159}$   and  $\mathfrak{s}_{6.166}^a$,  for which  we have 
\[
Jf_1=f_6,\,Jf_2=-f_4,\, Jf_3=f_5.
\]
\end{theorem}
\begin{proof}
Assume $\mathfrak{g}$ satisfies the requirements of the statement and let $J$ be a complex structure on $\mathfrak{g}$ satisfying $J \mathfrak{n}^1 \not\subset\mathfrak{n}$. Choose an adapted basis $\{e_l \}$ for $(\mathfrak{g},J)$ such that $Je_1=e_6$, $Je_2=e_3$, $Je_4=e_5$. Then, one is free to consider the Hermitian metric $g$ according to which the previously defined basis is orthonormal. The resulting Hermitian structure $(J,g)$ is such that  $J\mathfrak{n}^1=\nperp$, so that we may consider an equivalent model Hermitian Lie algebra $(\mu(a,\beta,A,\eta),J,g)$, satisfying $a=\operatorname{tr}A=0$, $A^*\eta=0$ and \eqref{integrable_perp} by Remark \ref{strunim_perp}, Lemma \ref{Jacobi_perp} and Proposition \ref{complex}, respectively.

We now observe that the matrices $B$ and $A$ in \eqref{B} are representatives of $B_{\mathfrak{g}}$ and $A_{\mathfrak{g}}$, respectively, while $\eta$ can be identified with $\eta_{\mathfrak{g}}$.
Independently of the adapted basis and Hermitian metric one chooses, conditions \eqref{integrable_perp} now impose the following conditions on $(B_\mathfrak{g},A_{\mathfrak{g}},\eta_{\mathfrak{g}})$:
\begin{enumerate}
	\item $[A,J \rvert_{\mathfrak{k}_1}]=0$, which  implies that there must exist a linear complex structure $\tilde{J}$ on $\mathfrak{n} / [\mathfrak{n},\mathfrak{n}]$ which is preserved by $\varphi_2$, i.e., which commutes with $A_{\mathfrak{g}}$,
	\item $\eta \in \Lambda^{1,1}\mathfrak{k}_1^*$, which  implies that $\eta_\mathfrak{g}$ must by $\tilde{J}$-invariant,
	\item  $\beta=0$, which  implies that there must exist a complement of $[\mathfrak{n},\mathfrak{n}]$ inside $\mathfrak{n}$ which is invariant under some representative of $B_{\mathfrak{g}}$.
\end{enumerate}
We can now look at the list of $6$-dimensional strongly unimodular almost nilpotent Lie algebras with nilradical $\mathfrak{n}$ satisfying $\dim \mathfrak{n}^1=1$, up to isomorphism, and test the previous three conditions on each one of them. Such a list was obtained in \cite{SW}, and we shall borrow its notations for non-nilpotent Lie algebras. All these Lie algebras are also listed in Tables \ref{table-h3} and \ref{table-h5} for easier reference.

We start by noticing that, in each of these Lie algebras, the fixed basis $\{f_l \}$ is such that $\mathfrak{n}= \mathbb R \left<f_1,\ldots,f_5\right>$, with $\mathfrak{n}^1 =\R \left < f_1 \right >$. This implies that $\mathfrak{g}/\mathfrak{n}$ and $\mathfrak{n}/\mathfrak{n}^1$ can be identified with $\R \left < f_6 \right >$ and $ \mathbb R  \left<f_2,\ldots,f_5\right>$, respectively, so that $B_\mathfrak{g}$ and $A_{\mathfrak{g}}$ may be represented by the $5 \times 5$-matrix associated with $\text{ad}_{f_6}\rvert_{\mathbb R \left<f_1,\ldots,f_5\right>}$ and the $4 \times 4$-matrix associated with the projection of $\text{ad}_{f_6}\rvert_{\mathbb R \left<f_2,\ldots,f_5\right>}$ onto $\mathbb R \left<f_2,\ldots,f_5\right>$, respectively. Instead, $\eta_\mathfrak{g}$ can be identified with (minus) the Lie bracket operation restricted to $\mathbb R \left<f_2,\ldots,f_5\right> \times \mathbb R \left<f_2,\ldots,f_5\right>$. For simplicity, we keep denoting the two matrices and the restricted Lie bracket by $B_\mathfrak{g}$, $A_{\mathfrak{g}}$ and $\eta_{\mathfrak{g}}$.

We can then use condition (3) in the above list to rule out some of the Lie algebras. In particular, we need to discard all those Lie algebras $\mathfrak{g}$ in which the matrix $B_{\mathfrak{g}}$ features a Jordan block comprising $f_1$ and some other element of $\{f_2,\ldots,f_5\}$: this is the case for the following Lie algebras:
\[
\frs{5.15} \oplus \R,\quad \frs{5.16} \oplus \R, \quad \frs{6.30}, \quad \frs{6.31}.
\]
Now, in order to apply conditions (1) and (2), we can write the generic linear complex structure on $\R\left<f_2,f_3,f_4,f_5\right> \cong \mathfrak{n} / \mathfrak{n}^1$ as a $4 \times 4$-matrix $\tilde{J}$ and impose
\begin{itemize}
	\item $\tilde{J}^2=-\text{Id}$,
	\item $[A_{\mathfrak{g}},\tilde{J}]=0$ to impose condition (1),
	\item $\tilde{J}\eta_{\mathfrak{g}}=\eta_{\mathfrak{g}}$ to impose condition (2).
\end{itemize}
Performing explicit computations for each of the remaining Lie algebras, we always get to contradictions, except for the Lie algebras listed in the statement. Then, for each of these Lie algebras we can determine an explicit complex structure $J$ satisfying $J \mathfrak{n}^1\not\subset \mathfrak{n}$, concluding the proof. 
\end{proof}

For the other case $J \mathfrak{n}^1\subset \mathfrak{n}$ we can prove the following

\begin{theorem} \label{class_cpx_sub}
Let $\mathfrak{g}$ be a $6$-dimensional strongly unimodular almost nilpotent Lie algebra with nilradical $\mathfrak{n}$ satisfying $\dim \mathfrak{n}^1 =1$. Then, $\mathfrak{g}$ admits complex structures $J$ satisfying $J \mathfrak{n}^1\subset \mathfrak{n}$ if and only if it is isomorphic to one of the following:
\begin{itemize}[leftmargin=0.6cm]
\item[]$\frs{4.6} \oplus \R^2=(f^{23},f^{26},-f^{36},0,0,0),\quad Jf_1=f_2,\,Jf_3=f_6,\,Jf_4=f_5$, \smallskip
\item[]$\frs{4.7} \oplus \R^2=(f^{23},f^{36},-f^{26},0,0,0),\quad Jf_1=f_4,\,Jf_2=f_3,\,Jf_5=f_6$, \smallskip
\item[]$\frs{5.16} \oplus \R=(f^{23}+f^{46},f^{36},-f^{26},0,0,0),\quad Jf_1=-f_5,\,Jf_2=f_3,\,Jf_4=f_6$, \smallskip
\item[]$\frs{6.25}=(f^{23},f^{36},-f^{26},0,f^{46},0),\quad Jf_1=-f_5,\,Jf_2=f_3,\,Jf_4=f_6$, \smallskip
\item[]$\frs{6.51}^{a,0}=(f^{23},af^{26},-af^{36},f^{56},-f^{46},0)$, $a>0,\quad Jf_1=af_2,\,Jf_3=f_6,\,Jf_4=f_5$, \smallskip
\item[]$\frs{6.158}=(f^{24}+f^{35},0,f^{36},0,-f^{56},0),\quad Jf_1=f_3,\,Jf_2=f_4,\,Jf_5=f_6$, \smallskip
\item[]$\frs{6.164}^a=(f^{24}+f^{35},af^{26},f^{56},-af^{46},-f^{36},0)$, $a >0,\quad Jf_1=af_2,\,Jf_3=f_5,\,Jf_4=f_6$.
\end{itemize}
For each of these Lie algebras, an explicit example of complex structure  $J$ satisfying $J \mathfrak{n}^1 \subset \mathfrak{n}$ is indicated in the previous list.
\end{theorem}
\begin{proof}
Let $\mathfrak{g}$ be a Lie algebra satisfying the requirements of the statement. Without loss of generality, we may equip $(\mathfrak{g},J)$ with a Hermitian metric, so that the resulting $(\mathfrak{g},J,g)$ is equivalent to a $6$-dimensional model Hermitian Lie algebra $(\mu(\Xi),J,g)$. By Section \ref{Hermitian_sub_sec}, we deduce that the matrix $B$ representing $\text{ad}_{e_{6}}\rvert_{\mathfrak{n}}$ and the $2$-form $\eta=de^1\rvert_{\Lambda^2 \mathfrak{n}}$, with respect to an adapted basis $\{ e_l \}$, are of the form
\begin{equation} \label{B_sub1}
	B=\left( \begin{array}{c|c|c|c}
		0 & 0 & \alpha & v_1 \\ \hline
		0 & a_2 & J\alpha + \nu & v_2 \\ \hline
		0 & 0 & \begin{matrix} -\tfrac{1}{2}(a+a_2) & q \\ -q & -\tfrac{1}{2}(a+a_2) \end{matrix} & v \\ \hline
		0 & 0 & 0 & a
	\end{array} \right), \quad \eta=c\,e^{34} + \nu \wedge e^{5} + a_2 \, e^{25},
\end{equation}
with the data
\[
a,a_2,v_1,v_2,q,c \in \R, \quad v \in \mathfrak{k}_3, \quad \alpha,\nu \in \mathfrak{k}_3^*,
\]
grouped in the notation
\begin{equation} \label{Xinew}
\Xi=(a,a_2,v_1,v_2,q,c,v,\alpha,\nu),
\end{equation}
satisfying
\begin{align}
\label{Lie1new}	&a_2(a+a_2)=0, \\
\label{Lie2new}	&c\,A^*e^{34}=0, \\
\label{Lie3new}	&(a+a_2)\nu+a_2 J\alpha + A^* \nu - c\,Jv^\flat =0,
\end{align}
by rewriting conditions \eqref{Lie1}--\eqref{Lie4} and noticing that $\iota_ve^{34}=Jv^\flat$. Moreover, we note that $\mathfrak{n} \cong \mathfrak{h}_3 \oplus \R^2$ (or, equivalently, $\eta^2=0$) if and only if
\begin{equation} \label{etasq}
	c\, a_2=0.
\end{equation}

Now, for every Lie algebra from Tables \ref{table-h3} and \ref{table-h5} in the Appendix, except the ones indicated in the statement, we shall operate by contradiction, assuming each to be isomorphic to the Lie algebra determined by $\mu(\Xi)$ and showing how one gets to a contradiction:
\begin{itemize}[leftmargin=0.6cm]
	\item $\mathfrak{h}_3 \oplus \frs{3.1}^{-1}$: $\text{Spec}(A_{\mathfrak{g}})=\{0,1,-1\}$ forces $a_2=-a \neq 0$, $q=0$. Now, $\eta_\mathfrak{g}(A_{\mathfrak{g}} \cdot, \cdot)=0$ would force $0=\eta(Be_2,e_5)=a^2$, a contradiction,
	\item $\mathfrak{h}_3 \oplus \frs{3.3}^0$: $\text{Spec}(A)=\{0,\pm i\}$ forces $a=a_2=0$, $q \neq 0$. Again, we have $\eta_\mathfrak{g}(A_{\mathfrak{g}} \cdot, \cdot)=0$, while 
	\[
	\eta(Be_3,e_3)=cq, \quad \eta(Be_3,e_5)=-\nu(e_4)q, \quad \eta(Be_4,e_5)=\nu(e_3)q,
	\]
	so that one should impose $c=0$, $\nu=0$, which would imply $\eta=0$, a contradiction,
	\item $\frs{5.15} \oplus \R$, $\frs{6.24}$, $\frs{6.30}$, $\frs{6.32}^{-1}$: $\text{Spec}(A_{\mathfrak{g}})=\{0,\pm 1\}$ forces $a_2=-a \neq 0$, $q=0$. Now, \eqref{etasq} forces $c=0$, so that \eqref{Lie3new} implies $\alpha=0$. Now, an explicit isomorphism of $\mu(\Xi)$ with $\frs{4.6} \oplus \R^2$ is provided by
	\begin{gather*}
	f_1=e_1,\quad f_2=\tfrac{1}{2a^2}(2av_1e_1+(av_2+\nu(v))e_2+2v+2a^2e_6),\\
	f_3=e_2,\quad f_4=\tfrac{1}{a}\nu(e_4)e_2+e_4,\quad f_5=\tfrac{1}{a}\nu(e_3)e_2+e_3, \quad f_6=\tfrac{1}{a} e_6,
	\end{gather*}
	\item $\frs{6.31}$: since $\text{Spec}(A_{\mathfrak{g}})=\{0,\pm i\}$ we must have $a=a_2=0$, $q \neq 0$. One can see that, no matter the chosen representative, $\dim(\ker B_{\mathfrak{g}})=1$, while $\R\left<e_1,e_2\right> \subseteq \ker B$, a contradiction.
	\item $\frs{6.34}^0$: we have $\text{Spec}(A_{\mathfrak{g}})=\{0,\pm i\}$, so $a=a_2=0$, $q \neq 0$. We should have $\eta_{\mathfrak{g}}(A_{\mathfrak{g}} \cdot, A_{\mathfrak{g}} \cdot)=0$, but if we set
	\[
	e_l^\prime=e_l + \tfrac{1}{q}(-J\alpha(e_l)e_1 + (\alpha-J\eta)(e_l)e_2),\quad l=3,4,
	\]
	we obtain
	\[
	\eta(Be_3^\prime,Be_4^\prime)=q^2c,
	\]
	forcing $c=0$. Equation \eqref{Lie3new} now reads $A^*\nu=0$, so that $\nu=0$, since $A$ is non-degenerate. This implies $\eta=0$,
	\item $\frs{6.43}$, $\frs{6.45}^{-1}$, $\frs{6.163}$: we have $\text{Spec}(A_{\mathfrak{g}})=\{\pm 1\}$, which forces $a_2=a \neq 0$, $q=0$, but now \eqref{Lie1new} leads to a contradiction,
	\item $\frs{6.44}$, $\frs{6.52}^{0,b}$, $\frs{6.165}$, $\frs{6.166}$, $\frs{6.167}$: $\text{Spec}(A_{\mathfrak{g}}) \cap \R = \emptyset$ immediately yields a contradiction,
	\item $\frs{6.46}^{a,-a}$: $\text{Spec}(A_{\mathfrak{g}})=\{a,-a,\pm i\}$ imposes $a_2=-a$, $q \neq 0$. Now we note that $\iota_Y \eta_\mathfrak{g}=0$ for all real eigenvectors of $A_{\mathfrak{g}}$, while, in our model Hermitian Lie algebra, $e_2$ is an eigenvector of $B$ and we have $\iota_{e_2}\eta=-ae^5 \neq 0$, a contradiction,
	\item $\frs{6.47}^{-1}$: $\text{Spec}(A_{\mathfrak{g}})=\{\pm 1\}$, so that we must have $a_2=a \neq 0$. Condition \eqref{Lie1new} now provides a contradiction,
	\item $\frs{6.159}$, $\frs{6.161}^{\varepsilon}$: we have $\text{Spec}(A_{\mathfrak{g}})=\{0,\pm i\}$, which forces $a=a_2=0$, implying $\eta^2=0$ by \eqref{etasq}, a contradiction,
	\item $\frs{6.160}$: by $\text{Spec}(A_{\mathfrak{g}})=\{0,\pm 1\}$ we deduce $a_2=-a \neq 0$, $q =0$. We then deduce that the kernel of $A$ is even-dimensional, which contradicts the fact that $\ker A_{\mathfrak{g}}$ is one-dimensional,
	\item $\frs{6.162}^a$: $\text{Spec}(A_{\mathfrak{g}})=\{\pm 1, \pm a\}$ imposes $a_2=a \neq 0$, which is incompatible with \eqref{Lie1new}.
\end{itemize}
For the remaining Lie algebras, an explicit complex structure $J$ satisfying $J \mathfrak{n}^1 \subset \mathfrak{n}$ can be found.
\end{proof}

As a consequence of Theorems  \ref{class_cpx_perp} and \ref{class_cpx_sub} we get  a   classification, up to isomorphisms, of $6$-dimensional strongly unimodular   almost nilpotent Lie algebras $\mathfrak{g}$ with nilradical $\mathfrak{n}$ satisfying $\dim \mathfrak{n}^1 = 1$  and admitting complex structures.

\section{SKT structures}   \label{sectionSKTstruct}

In this section we shall study the existence of SKT structures separately in cases (1),  (2) and (3).
We first obtain characterizations in terms of the algebraic data introduced in Section \ref{sec_Hermitian}, focusing on the strongly unimodular case. We then use these results to prove classification results, with particular attention to the $6$-dimensional case.

\subsection{Case (1)}
Before moving on, we show the following
\begin{lemma} \label{domg}
Let $\omega$ be the  fundamental form of the $2n$-dimensional  almost-Hermitian Lie algebra $(\mu(a,\beta,A,\eta),J,g)$. Then
 \[
d\omega=(\eta-A^*\omega)\wedge e^{2n},
\]
in the fixed adapted unitary basis $\{ e_1,\ldots,e_{2n} \}$, where $A^* \omega$ is a shorthand for $A^*(\omega \rvert_{\Lambda^2 \mathfrak{k}_1})$.
\end{lemma}
\begin{proof}
One has
\begin{equation} \label{domega}
d\omega(X,Y,Z)=g([X,Y],JZ)+g([Z,X],JY)+g([Y,Z],JX),\quad X,Y,Z \in \mathfrak{g}.
\end{equation}
It is easy to see that the above expression vanishes when $X,Y,Z \in \mathfrak{k}_1$ or one of the three vectors lies in $\mathfrak{n}^1$, so that one only needs to compute $d\omega(e_{2n},Y,Z)$, for $Y,Z \in \mathfrak{k}_1$:
\begin{align*}
d\omega(e_{2n},Y,Z)&=g([e_{2n},Y],JZ)-g([e_{2n},Z],JY)+g([Y,Z],Je_{2n})\\
&=-\omega(AY,Z)-\omega(Y,AZ)+\eta(Y,Z),
\end{align*}
from which the claim follows.
\end{proof}

If $J$ is integrable, we can prove the following
\begin{theorem}
A Hermitian almost nilpotent Lie algebra $(\mu(a,0,A,\eta),J,g)$  is SKT if and only if, in addition to \eqref{integrable_perp},
\begin{equation} \label{SKTeq}
	\eta^2=\eta\wedge A^*\omega,\quad (aA+A^2+A^tA)^*\omega=2a\eta.
\end{equation}
In particular, it is K\"ahler if and only if, in addition to \eqref{integrable_perp}, one has
\begin{equation} \label{Kahlereq}
	\eta=A^*\omega.
\end{equation}
\end{theorem}
\begin{proof}
The part of the claim regarding the K\"ahler condition follows directly from Lemma \ref{domg}. For the SKT condition, we have to impose $0=dd^c \omega = dJd\omega$. Observing that $\eta-A^*\omega \in \Lambda^{1,1}\mathfrak{k}_1^*$, since $\eta \in \Lambda^{1,1}\mathfrak{k}_1^*$ and $[A,J\rvert_{\mathfrak{k}_1}]=0$, we use Lemma \ref{domg} to obtain
\[
Jd\omega=-e^1 \wedge (\eta - A^*\omega).
\]
Before proceeding, we observe that
\[
d\eta=-a \eta \wedge e^{2n},
\]
which follows from the fact that $0=d^2e^1=d(\eta+ae^1 \wedge e^{2n})=d\eta+a \eta \wedge e^{2n}$. Then
\begin{align*}
	dd^c \omega&=-de^1 \wedge (\eta-A^*\omega)+e^1 \wedge (d\eta -d(A^*\omega))\\
	&=-\eta \wedge (\eta -A^*\omega) - a e^1 \wedge (\eta-A^*\omega) \wedge e^{2n}-ae^1 \wedge \eta \wedge e^{2n} -e^1 \wedge d(A^*\omega) \\
	&=-\eta^2+\eta \wedge A^*\omega - 2a e^1 \wedge \eta \wedge e^{2n} + ae^1 \wedge A^*\omega \wedge e^{2n} -e^1 \wedge d(A^*\omega).
\end{align*}
Now, we know that $d(A^*\omega) \subset \Lambda^2\mathfrak{k}_1^* \wedge e^{2n}$, and, for $X,Y \in \mathfrak{k}_1$, one computes
\begin{align*}
	d(A^*\omega)(e_{2n},X,Y)&=-A^*\omega([e_{2n},X],Y)-A^*\omega(X,[e_{2n},Y]) \\
	&=-A^*\omega(AX,Y)-A^*\omega(X,AY)\\
	&=-\omega(A^2X,Y)-2\omega(AX,AY)-\omega(X,A^2Y) \\
	&=-\omega((A^2+A^tA)X,Y)-\omega(X,(A^2+A^tA)Y)\\
	&=-(A^2+A^tA)^*\omega(X,Y),
\end{align*}
so that
\[
d(A^*\omega)=-(A^2+A^tA)^*\omega \wedge e^{2n}.
\]
The claim now follows by substituting this expression in the computation for $dd^c\omega$ and by separately imposing the vanishing of the  components of $dd^c\omega$ in $\Lambda^4\mathfrak{k}_1^*$ and $e^1 \wedge \Lambda^2\mathfrak{k}_1^* \wedge e^{2n}$.
\end{proof}

We can then study the SKT and K\"ahler condition in the strongly unimodular case.
\begin{proposition} \label{perp_SKTprop}
A Hermitian strongly unimodular almost nilpotent Lie algebra $(\mu(0,0,A,\eta),J,g)$ is never K\"ahler and it is SKT if and only if $A \in \mathfrak{u}(\mathfrak{k}_1,J\rvert_{\mathfrak{k}_1},g)$ and $\eta^2=0$ (or equivalently $\mathfrak{n} \cong \mathfrak{h}_3 \oplus \R^{2(n-2)}$). 
\end{proposition}
\begin{proof}
By Remark \ref{strunim_perp}, the strongly unimodular condition imposes in particular $a=0$. Then the second condition in \eqref{SKTeq} reads
\[
(A^2+A^tA)^*\omega=0.
\]	
By \cite{AL}, this is equivalent to $A+A^t=0$, meaning that $A \in \mathfrak{u}(\mathfrak{k}_1,J\rvert_{\mathfrak{k}_1},g)$, which is compatible with the condition $\operatorname{tr}A=0$ for the strong unimodularity. We then have $A^*\omega=0$, so that the first condition in \eqref{SKTeq} reads $\eta^2=0$, while the K\"ahler condition \eqref{Kahlereq} reduces to $\eta=0$, concluding the proof.
\end{proof}

\begin{theorem} \label{class_SKT_perp}
Let $\mathfrak{g}$ be a strongly unimodular $2n$-dimensional almost nilpotent Lie algebra with nilradical $\mathfrak{n}$ satisfying $\dim \mathfrak{n}^1=1$. Then $\mathfrak{g}$ admits an SKT structure $(J,g)$ satisfying $J \mathfrak{n}^1 = \mathfrak{n}^\perp$ if and only if 
it is isomorphic to the Lie algebra $ \R \left< f_1,\ldots,f_{2n} \right>$ with structure equations
\begin{equation} \label{streq_SKT}
	df^1=f^2 \wedge f^3,\quad df^{2l}=b_l f^{2l+1} \wedge f^{2n},\quad df^{2l+1}=-b_l f^{2l} \wedge f^{2n},\, l=1,\ldots,n-1,\quad df^{2n}=0,
\end{equation}
for some (not all vanishing) $b_l \in \R$, $l=1,\ldots,n-1$. In particular $\mathfrak{g}$ is of type I, namely the eigenvalues of its adjoint maps $\operatorname{ad}_X$, $X \in \mathfrak{g}$, are all purely imaginary. An example of SKT structure $(J,g)$ satisfying $J \mathfrak{n}^1 = \mathfrak{n}^\perp$ on such Lie algebra is given by
\begin{gather*}
Jf_1=f_{2n},\quad Jf_{2l}=f_{2l+1},\,l=1,\ldots,n-1,\\
g=\sum_{l=1}^{2n} (f^l)^2.
\end{gather*}

In real dimension six, $\mathfrak{g}$ is isomorphic to one among:
\begin{itemize}[leftmargin=0.6cm]
	\item[]$\mathfrak{h}_3 \oplus \mathfrak{s}_{3.3}^0=(f^{23},0,0,f^{56},-f^{46},0)$, \smallskip
	\item[]$\mathfrak{s}_{4.7}\oplus \R^2=(f^{23},f^{36},-f^{26},0,0,0)$,\smallskip
	\item[]$\mathfrak{s}_{6.52}^{0,b}=(f^{23},f^{36},-f^{26},bf^{56},-bf^{46},0)$, $b>0$.
\end{itemize}
\end{theorem}
\begin{proof}
Let $(J, g)$ be an SKT structure on $\mathfrak{g}$ and   $\{e_l \}$ be an adapted unitary  basis. This provides an equivalence between this Hermitian Lie algebra and a model Hermitian Lie algebra $(\mu(0,0,A,\eta),J,g)$, with $A \in \mathfrak{u}(\mathfrak{k}_1,J\rvert_{\mathfrak{k}_1},g)$, $\eta^2=0$. Let us denote by $H_{\eta}$ the skew-symmetric endomorphism of $(\mathfrak{k}_1,g)$ defined by  $\eta=g(H_{\eta} \cdot,\cdot)$. Then, \eqref{lie_perp} and the third condition in \eqref{integrable_perp} may be rewritten as
\[
A^t H_\eta + H_{\eta} A = aH_{\eta},\quad [H_\eta,J\rvert_{\mathfrak{k}_1}]=0.
\]
Using that $a=0$ and $A+A^t=0$, we obtain in particular that the skew-symmetric endomorphisms $A$ and $H_{\mathfrak{\eta}}$ commute, so that they may be diagonalized simultaneously. Exploiting $\eta^2=0$, it follows that there exists a unitary  basis $\{e_2^\prime,\ldots,e_{2n-1}^\prime\}$ of $\mathfrak{k}_1$, satisfying $Je_{2l}^\prime=e_{2l+1}^\prime$, $l=1,\ldots,n-1$, such that in the new basis we have
\[
A=\text{diag}(C_{b_1},\ldots,C_{b_{n-1}}),\quad H_{\eta}=\text{diag}(C_{-c},0,\ldots,0),
\]
for some (not all vanishing) $b_l \in \R$, $l=1,\ldots,n-1$, $c \in \R-\{0\}$.
where $C_{b}$, $b \in \R$, denotes the $2 \times 2$ matrix
\begin{equation} \label{Cb}
C_{b}=\begin{pmatrix} 0 & b \\ -b & 0 \end{pmatrix}.
\end{equation}
Up to rescaling $e_1$, the first part of the claim follows. The second part follows easily by looking at the classification of $6$-dimensional almost nilpotent Lie algebras with nilradical isomorphic to $\mathfrak{h}_3 \oplus \R^2$ (\cite{SW}, see the Appendix).
\end{proof}

\begin{remark} \label{datum_c}
By the previous proof, up to changing the adapted unitary basis $\{e_1,\ldots,e_{2n}\}$, we can therefore assume that in the SKT case the algebraic datum $\eta \in \Lambda^2 \mathfrak{k}_1^*$ is of the form $\eta= c e^{23}$. We can therefore replace $\eta$ with the simpler datum $c \in \R-\{0\}$ and denote the whole algebraic data simply with the pair $(A,c)$.
\end{remark}

\begin{remark} \label{vanishing_b}
	We note that a solvable Lie  algebra with structure equations \eqref{streq_SKT} is $2$-step solvable if and only if $b_1=0$, while it is decomposable if and only if any of the $b_l$'s vanish. In particular, if $\mathfrak{g}$ is $2$-step solvable, then it is decomposable. The vanishing of $b_1$ singles out the Lie algebra $\R\left<f_1,f_2,f_3\right>$, isomorphic to $\mathfrak{h}_3$, while each vanishing $b_l$, $l \neq 1$, determines a $\R^2$-block in the decomposition of $\mathfrak{g}$. 
\end{remark}

\begin{remark} The decomposable Lie algebra
$\mathfrak{h}_3 \oplus \mathfrak{s}_{3.3}^0$ is a particular case   of \cite[Theorem 7.5]{FS}.
The existence of SKT structures on the Lie algebra $\mathfrak{s}_{4.7} \oplus \R^2$ follows from \cite{MS}, since the four-dimensional $\mathfrak{s}_{4.7}$ admits SKT structures.
Instead, an SKT structure on $\mathfrak{s}_{6.52}^{0,b}$ was determined in \cite{FOU}.
\end{remark}

\begin{remark}
It is easy to see that the Lie group with Lie algebra $\mathfrak{h}_3 \oplus \mathfrak{s}_{3.3}^0$ admits compact quotients by lattices, for example the ones diffeomorphic to the product between the Iwasawa manifold and the solvmanifold $(2\pi\mathbb{Z} \times \mathbb{Z}^2) \backslash E(2)$. The same is true for the Lie group with Lie algebra $\mathfrak{s}_{4.7} \oplus \R^2$, which admits compact quotients diffeomorphic to the product of a secondary Kodaira surface and a $2$-torus (see \cite{Has}). For some values of the parameter $b$ (for example $b=1$, see \cite{FOU}), also the Lie groups with Lie algebra $\mathfrak{s}_{6.52}^{0,b}$ admit compact quotients by lattices.
\end{remark}

\subsection{Case (2)}
Like in the previous case, we start with a lemma regarding an expression for $d \omega$.
\begin{lemma} \label{sub_domega}
Let  $\omega$  be the fundamental form of the almost-Hermitian Lie algebra $(\mu(\Xi),J,g)$.  Then
\begin{equation} \label{sub_domega_eq}
	\begin{aligned}
d\omega=&-(a_1+a_2)e^{1} \wedge e^{2} \wedge e^{2n} + \gamma \wedge e^{1} \wedge e^{2n} + e^2 \wedge \xi - \nu \wedge e^{2} \wedge e^{2n-1} \\ &- \alpha \wedge e^{2} \wedge e^{2n}
        + v_1 e^{2}\wedge e^{2n-1} \wedge e^{2n} - v_2 e^{1} \wedge e^{2n-1} \wedge e^{2n} \\
         &+ (Jv)^\flat \wedge e^{2n-1} \wedge e^{2n} - A^*\omega \wedge e^{2n}.
    \end{aligned}
\end{equation}
\end{lemma}
\begin{proof}
 Using \eqref{B_sub}, \eqref{eta_sub} and \eqref{domega}, we have
\begingroup
\allowdisplaybreaks
\begin{align*}
d\omega(e_1,e_2,e_{2n})&=g([e_{2n},e_1],-e_1)+g([e_2,e_{2n}],e_2)=-a_1-a_2, \\
d\omega(e_1,X,e_{2n})&=-g([e_{2n},X],e_2)=-\gamma(X), \\
d\omega(e_2,X,Y)&=g([X,Y],-e_1)=\xi(X,Y), \\
d\omega(e_2,X,e_{2n-1})&=g([X,e_{2n-1}],-e_1)=\nu(X),\\
d\omega(e_1,e_{2n-1},e_{2n})&=-g([e_{2n},e_{2n-1}],e_2)=-v_2, \\
d\omega(e_2,e_{2n-1},e_{2n})&=g([e_{2n},e_{2n-1}],e_1)=v_1, \\
d\omega(e_2,X,e_{2n})&=-g([e_{2n},X],-e_1)=\alpha(X), \\
d\omega(X,e_{2n-1},e_{2n})&=g([e_{2n-1},e_{2n}],JX)=g(Jv,X), \\
d\omega(X,Y,e_{2n})&=g([e_{2n},X],JY)-g([e_{2n},Y],JX)=-\omega(AX,Y)-\omega(X,AY), 
\end{align*}
\endgroup
for every $X,Y \in \mathfrak{k}_3$, from which the claim follows.
\end{proof}

\begin{proposition} \label{sub_Kahlerprop}
	$(\mu(\Xi),J,g)$ is K\"ahler if and only if, in addition to \eqref{sub_integrable}, the following conditions are satisfied:
	\begin{equation}
		\label{sub_Kahler}
		\begin{gathered}
			a_1=\tfrac{a}{2}, \qquad a_2=-\tfrac{a}{2}, \qquad v_1=v_2=0, \qquad v=0, \\
			\alpha=\gamma=0, \qquad \xi=0, \qquad A \in \mathfrak{u}(\mathfrak{k}_3,J\rvert_{\mathfrak{k}_3},g),
		\end{gathered}
	\end{equation}
	with $a \neq 0$.
	Together, conditions \eqref{sub_integrable} and \eqref{sub_Kahler} also automatically imply \eqref{Lie1}--\eqref{Lie4}. The remaining independent data are $a\in \R-\{0\}$ and $A \in \mathfrak{u}(\mathfrak{k}_3,J\rvert_{\mathfrak{k}_3},g)$.
\end{proposition}
\begin{proof}
	All conditions apart from $a_1=\frac{a}{2}$ and $a_2=-\frac{a}{2}$ follow directly from Lemma \ref{sub_domega}. Imposing these conditions, together with $a_2=-a_1$, equations \eqref{Lie1}, \eqref{Lie3} and \eqref{Lie4} are satisfied, while \eqref{Lie2} reads
	\[
	0=2a_2(a+2a_2).
	\]
	Observing that now we have
	$\eta=2a_2 \, e^{2} \wedge e^{2n-1}$,
	which we want not to vanish, we have to impose $a_2=-\frac{a}{2}$, with $a \neq 0$, as well.
\end{proof}
\begin{remark}
	Let $\mathfrak{g}$ be a $2n$-dimensional almost nilpotent Lie algebra with nilradical $\mathfrak{n}$ satisfying $\dim \mathfrak{n}^1=1$.  
	Using the previous characterization, one can easily obtain that $\mathfrak{g}$ admits K\"ahler structures $(J,g)$ satisfying $J \mathfrak{n}^1\subset \mathfrak{n}$ if and only if it is isomorphic to the Lie algebra $\R \left< f_1,\ldots,f_{2n}\right>$ with structure equations
	\begin{equation}
		\begin{gathered}
	df^1=\frac{1}{2}f^{1} \wedge f^{2n} + f^{2} \wedge f^{3}, \qquad df^2=-\frac{1}{2}f^{2} \wedge f^{2n}, \qquad df^3=f^{3} \wedge f^{2n},\\ df^{2l+2}=b_lf^{2l+3} \wedge f^{2n},\qquad df^{2l+3}=-b_lf^{2l+2} \wedge f^{2n},\quad 1 \leq l \leq n-2, \qquad df^{2n}=0,
	    \end{gathered}
	\end{equation}
    for some $b_l \in \R$, $l=1,\ldots,n-2$.  
	In particular, $\mathfrak{g}$ is $3$-step solvable, non-unimodular and $\mathfrak{n} \cong \mathfrak{h}_3 \oplus \R^{2(n-2)}$.

An explicit example of such K\"ahler structure on these Lie algebras is provided by
\begin{gather*}
 Jf_1=-f_2,\quad Jf_3=f_{2n},\quad Jf_{2l+2}=f_{2l+3},\quad 1 \leq l \leq n-2,\\
 g=\sum_{l=1}^{2n} (f^l)^2.
\end{gather*}
	In  real dimension four  the  structure  equations reduce to  $( f^{23}+\frac{1}{2}f^{14}, -\frac{1}{2}f^{24},f^{34}, 0)$  and  in real  dimension six, $\mathfrak{g}$ is isomorphic to one among
	\begin{itemize}[leftmargin=0.6cm]
		\item[]$\frs{4.8}^{-\frac{1}{2}} \oplus \R^2=(f^{23}+\frac{1}{2}f^{16},f^{26},-\frac{1}{2}f^{36},0,0,0)$, \smallskip
		\item[]$\frs{6.78}^{a,-\frac{a}{2},0}=(f^{23}+\frac{a}{2}f^{16},af^{26},-\frac{a}{2}f^{36},f^{56},-f^{46},0)$, $a \neq 0$.
	\end{itemize}
\end{remark}

\begin{lemma}
Let us consider the $6$-dimensional strongly unimodular Hermitian almost nilpotent Lie algebra $(\mu(\Xi),J,g)$, with $\Xi$ in \eqref{Xinew}. One has
	\begin{equation} \label{dJdomg}
		\begin{aligned}
			dd^c\omega=&a_2(a+a_2) e^{1256} + c(a+a_2) e^{1346} \\
			&+e^1 \wedge \left( (A+a\,{\rm Id}_{\mathfrak{k}_3})^*(J\alpha) + cJv^\flat + a_2(J\alpha+\nu)) \right) \wedge e^{56} \\
			&+e^2 \wedge (A+(a+2a_2){\rm Id}_{\mathfrak{k}_3})^*(J\nu-\alpha) \wedge e^{56} \\
			&+\left(\lVert \alpha \rVert^2 + \lVert \nu \rVert^2 + \lVert J\alpha + \nu \rVert^2 + a_2(a+a_2) -2cv_1 \right) e^{3456}
		\end{aligned}
	\end{equation}
\end{lemma}
\begin{proof}
	The claim follows from an explicit computation. Using the fact that $e_1$ is central and $\mathfrak{n}^1 =\R\left < e_1 \right > $, one first notices that, for every $3$-form $\sigma$, we must have
\[
d\sigma(e_1,X,Y,Z)=0,\quad X,Y,Z \in \mathfrak{n},
\]
while, letting $X \in \mathfrak{k}_3$ be arbitrary,
\begingroup
\allowdisplaybreaks
	\begin{align*}
		dd^c\omega(e_1,e_2,X,e_6)=&-a_2d\omega(e_1,e_2,JX) - d\omega(e_1,e_2,AJX)=0,\\[1ex]
		dd^c\omega(e_1,e_2,e_5,e_6)=&-a_2 d \omega(e_1,e_2,e_6) - d\omega(e_1,e_2,Jv) -a d\omega (e_1,e_2,e_6) \\
		=& a_2(a+a_2), \\[1ex]
		dd^c\omega(e_1,e_3,e_4,e_6)=&-d\omega(e_2,e_3,Ae_4) + d\omega(e_2,e_4,Ae_3) \\
		=&(a+a_2)d\omega(e_2,e_3,e_4)\\
		=&c(a+a_2), \\[1ex]
		dd^c\omega(e_1,X,e_5,e_6)=&-d\omega(e_1,(J\alpha+\nu)(X)e_2,e_6)-d\omega(e_2,AJX,e_6)\\
		&-d\omega(-v_2e_1+Jv+ae_6,e_2,JX) \\
		=&a_2(J\alpha+\nu)(X)-\alpha(AJX)+c\omega(Jv,JX)-a \alpha(JX) \\
		=&a_2(J\alpha+\nu)(X)-(A+a\,{\rm Id}_{\mathfrak{k}_3})^*\alpha(JX)+cg(Jv,X),\\[1ex]
		dd^c\omega(e_2,e_3,e_4,e_6)=&cd\omega(e_1,e_2,e_5)-\alpha(e_3)d\omega(e_1,e_2,e_3) - d\omega(e_1,Ae_4,e_3) \\
		&-\alpha(e_4)d\omega(e_1,e_2,e_4) + d\omega(e_2,Ae_3,e_4) =0\\
		dd^c\omega(e_2,X,e_5,e_6)=&-a_2d\omega(e_2,JX,e_5)+a_2d\omega(e_1,JX,e_6)+\nu(X)d\omega(e_1,e_2,e_5) \\
		&+\alpha(X)d\omega(e_1,e_2,e_6) + d\omega(e_1,AJX,e_6) -v_1d\omega(e_1,e_2,JX)\\
		&-d\omega(e_1,Jv,JX) + a d\omega(e_1,JX,e_6) \\
		=&-a_2\nu(JX)-a_2(J\alpha+\nu)(JX)-a_2\alpha(X)-(J\alpha+\nu)(AJX) \\
		&-a(J\alpha+\nu)(JX) \\
		=&(A+(a+2a_2){\rm Id}_{\mathfrak{k}_3})^*(J\nu-\alpha)(X), \\[1ex]
		dd^c\omega(e_3,e_4,e_5,e_6)=&-cd\omega(e_2,e_5,e_6)+\nu(e_3)d\omega(e_2,e_3,e_5) + \alpha(e_3)d\omega(e_2,e_3,e_6) \\
		&-(J\alpha+\nu)(e_3)d\omega(e_1,e_3,e_6)-d\omega(e_3,Ae_4,e_6) + \nu(e_4)d\omega(e_2,e_4,e_5) \\
		&+\alpha(e_4)d\omega(e_2,e_4,e_6) - (J\alpha+\nu)(e_4)d\omega(e_1,e_4,e_6) - d\omega(Ae_3,e_4,e_6) \\
		&-v_1d\omega(e_2,e_3,e_4) +v_2 d\omega(e_1,e_3,e_4) - a d\omega(e_3,e_4,e_6) \\
		=&-cv_1+\nu(e_3)^2 +\alpha(e_3)^2+(J\alpha+\nu)(e_3)^2+\tfrac{1}{2}(a+a_2)^2+\nu(e_4)^2+\alpha(e_4)^2\\
		&+(J\alpha+\nu)(e_4)^2+\tfrac{1}{2}(a+a_2)^2-cv_1-a(a+a_2) \\
		=&\lVert \alpha \rVert^2 + \lVert \nu \rVert^2 + \lVert J\alpha + \nu \rVert^2 + a_2(a+a_2) -2c \, v_1. \qedhere
	\end{align*}
\endgroup
\end{proof}

\begin{proposition} \label{SKT_sub_prop}
	The $6$-dimensional strongly unimodular Hermitian Lie algebra $(\mu(\Xi),J,g)$, with $\Xi$ in \eqref{Xinew} is SKT if and only if either
	\begin{equation} \label{SKT_sub2}
		a_2=-a,\quad c \neq 0, \quad v_1=\tfrac{1}{c} \lVert \alpha \rVert^2,\quad v=-\tfrac{1}{c}\left(A^t + a\,{\rm Id}_{\mathfrak{k}_3}\right)(\alpha^\sharp), \quad \nu=-J\alpha,
	\end{equation}
or
	\begin{equation} \label{SKT_sub1}
		a_2=-a \neq 0,\quad c=0,\quad \alpha=\nu=0,
	\end{equation}
\end{proposition}

\begin{proof}
	We need to look at \eqref{dJdomg} and impose $dd^c\omega=0$. We start by looking at
	\[
	dd^c\omega(e_1,e_2,e_5,e_6)=a_2(a+a_2),\quad dd^c\omega(e_1,e_3,e_4,e_6)=c(a+a_2).
	\]
	We then have to discuss two cases separately:
	\begin{itemize}
		\item[(i)] $a_2=-a$,
		\item[(ii)] $a_2=c=0$, $a \neq 0$.
	\end{itemize}
	Let us start with case (i). For an arbitrary $X \in \mathfrak{k}_3$, we have
	\[
	dd^c\omega(e_2,X,e_5,e_6)=(A-a\,{\rm Id}_{\mathfrak{k}_3})^*(J\nu - \alpha)(X).
	\]
	By the fact that $\text{Spec}(A) =\{\pm iq\}$, we then have $\nu=-J\alpha$.
	Instead, one has
	\begin{align*}
		dd^c\omega(e_1,JX,e_5,e_6)&=(A+a\,{\rm Id}_{\mathfrak{k}_3})^*(\alpha)(X) + c v^{\flat}(X), \\
		dd^c\omega(e_3,e_4,e_5,e_6)&=2\left(\lVert \alpha \rVert^2 - c v_1\right).
	\end{align*}
	Now, if $c=0$, we only need to have $\alpha=0$ in order for $dd^c\omega$ to vanish, and so we get \eqref{SKT_sub1}. If $c \neq 0$, instead, one obtains the conditions regarding $v_1$ and $v$ in \eqref{SKT_sub2}.

	Moving on to case (ii), by \eqref{dJdomg} we have have
	\[
	dd^c\omega(e_3,e_4,e_5,e_6)=\lVert \alpha \rVert^2 + \lVert \nu \rVert^2 + \lVert J\alpha + \nu \rVert^2,
	\]
	whose vanishing forces $\alpha=\nu=0$. This yields a contradiction, since now $\eta=0$.
	
	To finish the proof, we note that, in both cases, conditions \eqref{Lie1new}--\eqref{Lie3new} are already satisfied.
\end{proof}

Now, one can check that a $2n$-dimensional strongly unimodular Hermitian Lie algebra $(\mu(\Xi),J,g)$ satisfying conditions \eqref{SKT_sub1} is SKT. 
More precisely, if we assume $\mathfrak{k}_3=\R\left<e_3,\ldots,e_{2n-2}\right>$ ($\mathfrak{k}_3=\{0\}$ in real dimension four) to be an abelian ideal, namely we have
\begin{equation} \label{assumptionSKT}
\alpha=\gamma=0,\quad \xi=0,
\end{equation}
then $(\mu(\Xi),J,g)$ is SKT if and only if \eqref{SKT_sub1} holds:

\begin{proposition} \label{sub_SKTprop}
	Assuming \eqref{assumptionSKT}, $(\mu(\Xi),J,g)$ is strongly unimodular and SKT if and only if, in addition to \eqref{sub_integrable}, 
	one has $a_1=0$, $a_2=-a \neq 0$, $A \in \mathfrak{u}(\mathfrak{k}_3,J\rvert_{\mathfrak{k}_3},g)$.
\end{proposition}
\begin{proof}
First of all, the strongly unimodular condition imposes $a_1=0$. One can use \eqref{sub_domega_eq} to compute
\begin{equation} \label{dJdomg_assumpt}
dd^c\omega=a_2(a+a_2)e^{1} \wedge e^{2} \wedge e^{2n-1} \wedge e^{2n}+(A^2+A^tA+aA)^*\omega \wedge e^{2n-1} \wedge e^{2n}.
\end{equation}
By the computations in \cite[Theorem 4.6]{AL}, it then follows that $[A,A^t]=0$, with the eigenvalues of $A$ having real part equal to $0$ or $-\tfrac{a}{2}$.

Now, the condition $a_2(a+a_2)=0$ imposed by the vanishing of \eqref{dJdomg_assumpt} and the requirement $a_2\,e^{2} \wedge e^{2n-1}=\eta \neq 0$, implies $a_2=-a \neq 0$.
The traceless condition for the restriction of $B$ to $\mathfrak{k}_1$ imposes that all the eigenvalues of $A$ are purely imaginary, namely $A \in \mathfrak{u}(\mathfrak{k}_3,J\rvert_{\mathfrak{k}_3},g)$.

To end, we note that conditions \eqref{Lie1}--\eqref{Lie4} are all satisfied, concluding the proof.
\end{proof}

\begin{theorem} \label{class_SKT_sub}
Let $\mathfrak{g}$ be a $6$-dimensional strongly unimodular almost nilpotent Lie algebra with nilradical $\mathfrak{n}$ satisfying $\dim \mathfrak{n}^1=1$. Then, $\mathfrak{g}$ admits SKT structures $(J,g)$ satisfying $J \mathfrak{n}^1 \subset \mathfrak{n}$ if and only if it is isomorphic to one of the following:
\begin{itemize}[leftmargin=0.6cm]
	\item[]$\frs{4.6} \oplus \R^2=(f^{23},f^{26},-f^{36},0,0,0)$, \smallskip
	\item[]$\frs{4.7} \oplus \R^2=(f^{23},f^{36},-f^{26},0,0,0)$, \smallskip
	\item[]$\frs{6.25}=(f^{23},f^{36},-f^{26},0,f^{46},0)$, \smallskip
	\item[]$\frs{6.51}^{a,0}=(f^{23},af^{26},-af^{36},f^{56},-f^{46},0)$, $a>0$, \smallskip
	\item[]$\frs{6.158}=(f^{24}+f^{35},0,f^{36},0,-f^{56})$, \smallskip
	\item[]$\frs{6.164}^a=(f^{24}+f^{35},af^{26},f^{56},-af^{46},-f^{36},0)$, $a >0$.
\end{itemize}
Explicit examples of such SKT structures on these Lie algebras are provided by the complex structure indicated in Theorem \ref{class_cpx_sub}, together with the metric for which the basis $\{f_1,\ldots,f_6\}$ is orthonormal.
\end{theorem}
\begin{proof}
Looking at Theorem \ref{class_cpx_sub}, we need to prove that $\frs{5.16}$ does not admit SKT structures $(J,g)$ satisfying $J \mathfrak{n}^1 \subset \mathfrak{n}$. Proceeding similarly to the proof of Theorem \ref{class_cpx_sub}, we consider a model SKT Lie algebra $(\mu(\Xi),J,g)$ and, by contradiction, we assume it to be isomorphic to $\frs{5.16}$.

First, we assume \eqref{SKT_sub1}, but ${\rm Spec}(A_{\mathfrak{g}})=\{0,\pm i\}$ would force $a=0$, a contradiction.
For the same reason, if we assume \eqref{SKT_sub2}, we have to set $a=0$. But now, we notice that the two changes of basis
\begin{alignat*}{2}
(f_1,f_2,f_3,f_4,f_5,f_6)&=\left(ce_1,\tfrac{1}{q}\alpha(e_4)e_1+e_3,-\tfrac{1}{q}\alpha(e_3)e_1+e_4,-\tfrac{1}{c} \alpha^{\sharp} + e_5,e_2, \tfrac{1}{q}e_6\right),&\quad &\text{ if $v_2=0$}, \\
(f_1,f_2,f_3,f_4,f_5,f_6)&=\left(ce_1,\tfrac{1}{q}\alpha(e_4)e_1+e_3,-\tfrac{1}{q}\alpha(e_3)e_1+e_4,-\tfrac{1}{c} \alpha^{\sharp} + e_5,v_2\,e_2, \tfrac{1}{q}e_6\right),&\quad &\text{ if $v_2\neq 0$,}
\end{alignat*}
provide explicit isomorphisms of $\mu(\Xi)$ with $\frs{4.7} \oplus \R^2$ and $\frs{6.25}$, respectively, a contradiction.
\end{proof}

\begin{remark}
By \cite{Mac}, the Lie algebras of the previous theorem do not admit symplectic structures, hence they cannot be K\"ahler. 
\end{remark}

Using Proposition \ref{sub_SKTprop}, we can obtain a classification result in every even real dimension by restricting to SKT structures in which $\mathfrak{k}_3$ is an abelian ideal.
\begin{theorem} \label{sub_class}
Let $\mathfrak{g}$ be a $2n$-dimensional strongly unimodular almost nilpotent Lie algebra with nilradical $\mathfrak{n}$ satisfying $\dim \mathfrak{n}^1 =1$. Then $\mathfrak{g}$ admits SKT structures $(J,g)$ satisfying $J \mathfrak{n}^1\subset \mathfrak{n}$ and $\mathfrak{k}_3$ being an abelian ideal if and only if it admits a basis $\{f_1,\ldots,f_{2n}\}$ satisfying
\begin{equation} \label{SKT_eq}
	\begin{gathered}
		\begin{alignedat}{2}
		df^1&=f^{2} \wedge f^{3},&\quad df^2&=-f^{2} \wedge f^{2n}, \\
		df^3&=f^{3} \wedge f^{2n},  &\quad df^{2n}&=0, \\
        df^{2l+2}&=b_l f^{2l+3} \wedge f^{2n},&\qquad df^{2l+3}&=-b_l f^{2l+2} \wedge f^{2n},\quad 1 \leq l \leq n-2,
		\end{alignedat}
	\end{gathered}
\end{equation}
for some $b_l \in \R$, $1 \leq l \leq n-2$. 

In particular, $\mathfrak{g}$ is $3$-step solvable and it is decomposable if and only if $b_l=0$ for some $l$, $1 \leq l \leq n-2$.
Moreover, $\mathfrak{g}$ does not admit K\"ahler structures.
Explicit examples of SKT structures  are provided by
\begin{equation} \label{SKT_sub_ex}
Jf_1=-f_2,\quad Jf_3=f_{2n},\quad Jf_{2l+2}=f_{2l+3}, \quad 1 \leq l \leq n-2,\quad \quad 
g= \sum_{l=1}^{2n} (f^l)^2.
\end{equation}
\end{theorem}
\begin{proof}
We can work with model SKT Lie algebras $(\mu(a,v_1,v_2,v,A),J,g)$. With respect to the splitting $\mathfrak{g}=\R \left <e_1  \right > \oplus \R \left < e_2 \right >  \oplus \mathfrak{k}_3 \oplus \R \left <  e_{2n-1} \right > \oplus \R \left <e_{2n} \right >$, by Proposition \ref{sub_SKTprop}, we have
\[
\eta=-ae^{2} \wedge e^{2n-1}
\]
and
\begin{equation} \label{Bcases}
B=\left( \begin{array}{c|c|c|c}
	0 & 0 &0 & v_1 \\ \hline
	0& -a & 0 & v_2  \\ \hline
	 0 &0  & A & v \\ \hline
	0&0  &0  & a
\end{array} \right)
\end{equation}
for some $a \in \R-\{0\}$, $v_1,v_2 \in \R$, $v \in \mathfrak{k}_3$, $A \in \mathfrak{u}(\mathfrak{k}_3,J\rvert_{\mathfrak{k}_3},g)$, so that, up to diagonalization, we can assume $A$ to be of the form
\[
A=\text{diag}(C_{b_1},\ldots,C_{b_{n-2}}),
\]
if $n \geq 3$, while we recall that $\mathfrak{k}_3=\{0\}$ for $n=2$, so that we do not have the algebraic datum $A$.

It is now a matter of modifying the adapted basis to obtain a basis satisfying the requirements of the statement. In both cases in \eqref{Bcases}, we can notice that $a$ is an eigenvalue of $B$ of multiplicity one, so that one can find a vector $\tilde{X} \in \mathfrak{k}_1$ such that $\tilde{v}\coloneqq v_1e_1+v_2e_2+v \in \mathfrak{k}_1$ is given by
\[
\tilde{v}=B\tilde{X} - a \tilde{X}.
\]
Then, one can consider the new basis for $\mathfrak{g}$ given by:
\[
(f_1,\ldots,f_{2n})=\left( -ae_1,e_2,e_{2n-1}-\tilde{X},e_3,\ldots,e_{2n-2},\tfrac{1}{a}e_{2n}\right).
\]
One can easily see that this basis satisfies the requirements of the statement.

Now, we shall prove that none of these Lie algebras admit symplectic structures.
We can write the generic $2$-form $\Omega$ on $\mathfrak{g}$ as
\begin{align*}
	\Omega=&q_1f^{1} \wedge f^{2}+q_2f^{1} \wedge f^{3}+q_3f^{1} \wedge f^{2n}+q_4f^{2} \wedge f^{3}+q_5f^{2} \wedge f^{2n}+q_6f^{3} \wedge f^{2n} \\
	       &+f^1 \wedge \alpha_1 + f^2 \wedge \alpha_2 + f^3 \wedge \alpha_3 + f^{2n} \wedge \alpha_{2n} + \kappa,
\end{align*}
with $q_l \in \R$, $l=1,\ldots,6$, $\alpha_l \in \mathfrak{a}^*$, $l=1,2,3,2n$, $\kappa \in \Lambda^2 \mathfrak{a}^*$, with $\mathfrak{a} \coloneqq \R\left<f_4,\ldots,f_{2n-1}\right>$ ($\mathfrak{a}=\{0\}$ when $n=2$).
Then one can compute, in particular,
\begin{align*}
    d\Omega(f_1,f_2,f_{2n})=q_1, \quad 
	d\Omega(f_1,f_3,f_{2n})=-q_2, \quad
	d\Omega(f_2,f_3,f_{2n})=q_3
\end{align*}
and
\[
d\Omega(f_2,f_3,X)=\alpha_1(X), \quad X \in \mathfrak{a}.
\]
We conclude that
\[
q_1=q_2=q_3=0, \quad \alpha_1=0.
\]
Then, we have $\iota_{f_1} \Omega=0$, hence $\Omega$ is degenerate, concluding the proof.
\end{proof}

\begin{example}
Building on the results of Theorem \ref{sub_class}, we now exhibit a new family of $2n$-dimensional non-K\"ahler compact almost nilpotent solvmanifolds, $n \geq 2$, admitting SKT structures.

Consider the $2n$-dimensional simply connected strongly unimodular solvable Lie group $S_{2n}$ with indecomposable Lie algebra $\mathfrak{s}_{2n}$ endowed with a fixed basis $\{f_1,\ldots,f_{2n}\}$ with structure equations \eqref{SKT_eq} with  $b_l=\frac{2\pi}{\ln(2+\sqrt{3})}$ for all $l=1,\ldots,n-2$. The Lie groups $S_{2n}$ admit the left-invariant SKT structure \eqref{SKT_sub_ex}.

We shall now show that $S_{2n}$ admits a lattice.
By \cite{Bock}, a simply connected almost nilpotent Lie group with Lie algebra $\mathfrak{g}=\mathfrak{n} \rtimes_D \R$ admits compact quotients by lattices if there exist a rational basis of $\mathfrak{n}$ and some $t_0 \neq 0$ such that the matrix associated with $\text{exp}(t_0D) \in \text{GL}(\mathfrak{n})$ has integer entries.
In our case, we can view $\mathfrak{s}_{2n} \cong \mathfrak{n} \rtimes_D \R f_{2n}$, with $\mathfrak{n}=\R \left<f_1,\ldots,f_{2n-1} \right>$ and $D=\text{diag}(0,-1,1,C_{c},\ldots,C_{c})$, in the fixed basis.
	
	Now, $S_4$ admits compact quotients by lattices, given by Inoue surfaces of type $S_+$ (see \cite{Has}): applying the previous criterion, we can consider $t_0=\ln(2+\sqrt{3})$ and the new basis for $\mathfrak{n}$ given by
	\begin{equation} \label{newbasisn}
		f_1^\prime=-\frac{\sqrt{3}}{6} f_1,\quad f_2^\prime=\frac{\sqrt{3}}{6}(-f_2+f_3),\quad f_3^\prime=\frac{1}{2}(f_2+f_3)+\frac{\sqrt{3}}{3}(-f_2+f_3).
	\end{equation}
	The only non-zero bracket for $\mathfrak{n}$ is, as before, $[f_2^\prime,f_3^\prime]=-f_1^\prime$, and, in the new basis, one has
	\[
	\text{exp}(t_0D)=\begin{pmatrix} 1 & 0 & 0 \\ 0 & 0 & -1 \\ 0 & 1 & 4 \end{pmatrix}.
	\]
	
	To determine the existence of co-compact lattices in $S_{2n}$, for $n \geq 4$, we can observe that in the new basis for $\mathfrak{n}$ provided by \eqref{newbasisn} and $f_l^\prime=f_l$, $l=4,\ldots,2n-1$, the only non-zero bracket is again $[f_2^\prime,f_3^\prime]=-f_1^\prime$, and, setting $t_0=\ln (2+\sqrt{3})$, we have
	\[
	\text{exp}(t_0D)=\begin{pmatrix} 1 & 0 & 0 & 0 \\ 0 & 0 & -1 & 0 \\ 0 & 1 & 4 & 0 \\ 0 & 0 & 0 & I_{2(n-2)} \end{pmatrix},
	\]
	with $I_l$ denoting the $l \times l$ identity matrix.

The compact solvmanifolds we have constructed cannot admit K\"ahler structures: this follows from the results in \cite{Has2}, by which a compact solvmanifold $G/ \Gamma$ admitting K\"ahler structure must be such that $G$ is of type I, namely the eigenvalues of its adjoint maps $\text{ad}_X$, $X \in \mathfrak{g}$, must be purely imaginary. This is not the case for the Lie groups $S_{2n}$ we have considered.
\end{example}

\subsection{Case (3)} In the remaining case, namely the one where the SKT structure $(J,g)$ is such that $J \mathfrak{n}^1$ has trivial intersection with both $\mathfrak{n}$ and $\nperp$, we obtain the following result:
\begin{proposition}
Let $\mathfrak{g}$ be a $2n$-dimensional strongly unimodular almost nilpotent Lie algebra with nilradical $\mathfrak{n}$ satisfying $\dim \mathfrak{n}^1=1$. Assume $\mathfrak{g}$ admits an SKT structure $(J,g)$ satisfying $\hat\theta=0$, namely $J \mathfrak{n}^1 = \mathfrak{n}^{\perp_g}$. Then,
\begin{itemize}
\item[{\normalfont (i)}] for every $\theta\in (0,\frac{\pi}{2})$, $\mathfrak{g}$ also admits SKT structures $(J_{\theta},g_{\theta})$ satisfying $\hat\theta=\theta$,
\item[{\normalfont (ii)}] $\mathfrak{g}$ admits SKT structures satisfying $\hat\theta=\frac{\pi}{2}$ (namely, $J \mathfrak{n}^1\subset \mathfrak{n}$) if and only if it decomposes as the Lie algebra direct sum of an abelian Lie algebra and an almost nilpotent Lie algebra.
\end{itemize}
\end{proposition}
\begin{proof}
By Theorem \ref{class_SKT_perp}, $\mathfrak{g}$ satisfies the hypotheses of the statement if and only if it admits a basis $\{f_1,\ldots,f_{2n}\}$ with structure equations \eqref{streq_SKT}. 
With respect to this frame, we can fix the complex structure $J$ defined by $Jf_1=f_{2n}$, $Jf_{2l}=f_{2l+1}$, $l=1,\ldots,n-1$. For every $\theta \in [0,\frac{\pi}{2})$, denoting $f^k \odot f^l \coloneqq \tfrac{1}{2} (f^k \otimes f^l + f^l \otimes f^k)$, $k,l=1,\ldots,2n$, we can then consider the $J$-Hermitian metric
\[
g_\theta=\sum_{l=1}^{2n} (f^l)^2 +2\sin(\theta)(f^1 \odot f^{2n-1} - f^{2n-2} \odot f^{2n}),
\]
which is SKT, having (non-exact) closed torsion form
\[
d^c\omega_{\theta}=-f^{1} \wedge f^{2} \wedge f^{3}-b_{n-1}\sin(\theta)f^{1} \wedge f^{2n-2} \wedge f^{2n}-\sin(\theta)f^{2} \wedge f^{3} \wedge f^{2n-1},
\]
with $\omega_\theta \coloneqq g_{\theta}(J\cdot,\cdot)$.
Generators of $\mathfrak{n}^1$ and $\mathfrak{n}^{\perp_{g_\theta}}$ are $f_1$ and $f_{2n}+\sin(\theta)f_{2n-2}$, respectively, and one can compute	
\[
\frac{g_{\theta}(J f_1,f_{2n}+\sin(\theta)f_{2n-2})}{\lVert J f_1 \rVert_{g_\theta} \lVert f_{2n}+\sin(\theta)f_{2n-2} \rVert_{g_\theta}}=\cos(\theta),
\]	
showing that $(J,g_{\theta})$ satisfies $\hat \theta = \theta$. 	
This concludes the proof of the first part of the claim. Observe that the same argument does not hold for $\theta=\frac{\pi}{2}$, since $g_{\frac{\pi}{2}}$ would be degenerate.

We now focus on the second part of the claim. Using again the basis $\{f_l \}$ of $\mathfrak{g}$ with structure equations \eqref{streq_SKT} and recalling Remark \ref{vanishing_b}, we wish to prove that $\mathfrak{g}$ admits SKT structures satisfying $\hat\theta=\frac{\pi}{2}$ if and only if $b_l=0$ for some $l \in \{2,\ldots,n-1\}$, with $n \geq 3$.

For the \lq \lq if'' part, up to reordering the basis, we can assume $b_2=0$ and consider the SKT structure $(J^\prime,g^\prime)$ defined by
\begin{align*}
	J^\prime=&(f^1 \otimes f_4 - f^4 \otimes f_1)+(f^2 \otimes f_3 - f^3 \otimes f_2) + (f^5 \otimes f_{2n} - f^{2n} \otimes f_5)\\ 
	         &+ \sum_{l=3}^{n-1} (f^{2l} \otimes f_{2l+1} - f^{2l+1} \otimes f_{2l}), \\
	g^\prime=&\sum_{l=1}^{2n} (f^l)^2,
\end{align*}
with non-exact torsion form $H^\prime=-f^{123}$ and satisfying $\hat \theta =0$, since $J^\prime(\mathfrak{n}^1)=\R\left < f_4  \right >\subset \mathfrak{n}$.

We now look at the \lq \lq only if'' part. Our objective is to show that, when $b_2,\ldots,b_{n-1}$ are non-zero, the algebras of Theorem \ref{class_SKT_perp} do not admit SKT structures with $\hat \theta=0$. By contradiction, assume $b_2,\ldots,b_{n-1}$ non-zero and let $(J^\prime,g^\prime)$ be an SKT structure satisfying $\hat\theta=0$. Following Subsection \ref{sec_Hermitian}, consider an adapted unitary basis $\{e_l \}$. In this basis, the matrix $B$ associated with $\text{ad}_{e_{2n}}\rvert_{\mathfrak{n}}$ is of the form \eqref{B_sub}, while $\eta\coloneqq de^1 \rvert_{\Lambda^2 \mathfrak{n}}$ is given by \eqref{eta_sub}. Imposing \eqref{sub_integrable} to guarantee the integrability of $J^\prime$, we observe that $B$ is in block-upper-triangular form, so that
\[
\operatorname{Spec} B = \{0,a,a_2\} \cup \operatorname{Spec} A.
\]
We then need to impose $a=a_1=a_2=0$ to agree with the type I condition for $\mathfrak{g}$ (recall Theorem \ref{class_SKT_perp}). One can then rule out the case $b_1 \neq 0$, which would force the multiplicity of the eigenvalue $0$ of $B$ to be equal to $1$. In the remaining case characterized by $b_1=0$, the Lie algebra $\mathfrak{g}$ decomposes as $\mathfrak{h}_3 \oplus \mathfrak{l}$, with $\mathfrak{l}$ a centerless indecomposable almost abelian Lie algebra of type I. In particular, $\mathfrak{g}$ is $2$-step solvable and actually satisfies the stronger condition $[\mathfrak{n},\mathfrak{g}^1]=\{0\}$. This case also forces the multiplicity of the eigenvalue $0$ of $B$ to be exactly equal to $3$, so that $A$ is necessarily non-degenerate. Now, so far we have
\[
\eta= \xi + \nu \wedge e^{2n-1},\quad \xi \in \Lambda^2 \mathfrak{k}_3^*,\,\nu \in \mathfrak{k}_3^*.
\]
Let $X,Y \in \mathfrak{k}_3$ be arbitrary. Using that $A$ is an isomorphism, we deduce that there exists a unique $Z \in \mathfrak{k}_3$ so that $A(Z)=Y$. Then, one can compute
\[
[X,[e_{2n},Z]]=-\xi(X,Y)e_1, \quad [e_{2n-1},[e_{2n},Z]]=\nu(Y) e_1.
\]
Since $[\mathfrak{n},\mathfrak{g}^1]=\{0\}$, we have to impose $\xi=0$, $\nu=0$, yielding $\mathfrak{n}^1=\{0\}$, a contradiction.
\end{proof}

\section{Pluriclosed flow} \label{sec_plflow}
The pluriclosed flow on a Hermitian manifold $(M, J, g_0)$ of complex dimension $n$  is given by
\[
\tfrac{d}{dt}{\omega} (t) =-(\rho^B_{\omega(t)})^{1,1},\quad \omega(0)=\omega_0,
\]
where $\omega_0$ is the fundamental form of $(J, g_0)$ and  $(\rho^B_{\omega(t)})^{1,1}$ denotes the $(1,1)$-part of the \emph{Bismut-Ricci form} associated with $\omega(t)$, having local expression
\begin{equation} \label{bismutricci}
\rho^B_\omega(X,Y) = -\frac{1}{2} \sum_{l=1}^{2n} g(R^B(X,Y)e_l,Je_l), \quad X,Y \in \Gamma(TM)
\end{equation}
in a local orthonormal frame $\{e_1,\ldots,e_{2n}\}$. Here,  $R^B(X,Y)=[\nabla^B_X,\nabla^B_Y]-\nabla^B_{[X,Y]}$, $X,Y \in \Gamma(TM)$  the curvature tensor of the Bismut connection $\nabla^B$, which is the unique Hermitian connection having totally skew-symmetric torsion.

Given a split generalized K\"ahler structure $(J_+,J_-,g)$, the pluriclosed flow starting from the SKT structure $(J_+,g)$ produces a family of SKT metrics with respect to both $J_+$ and $J_-$, preserving the generalized K\"ahler condition $d^c_+ \omega_+ + d^c_- \omega_- =0$, so that one may say that the given split generalized K\"ahler structure evolves by $(J_+,J_-,g(t))$. This interpretation of the pluriclosed flow takes the name of \emph{generalized K\"ahler-Ricci flow}.

When one works on Lie groups, left-invariant initial conditions yield left-invariant solutions, so that the pluriclosed flow and the generalized K\"ahler flow reduce to systems of \textsc{ode}s on the associated Lie algebra.

We recall  that an SKT structure $(J,g)$ on a real Lie algebra  $\mathfrak{g}$  is a \emph{pluriclosed soliton} if the pluriclosed flow starting from $(J,g)$ evolves simply by rescalings and time-dependent biholomorphisms, namely $g(t)=c(t)\varphi_t^*g$, with $c(t) \in \R$ and $\varphi_t$ biholomorphisms. More precisely, we say that $(J,g)$ is a \emph{shrinking}, \emph{expanding} or \emph{steady} soliton on $\mathfrak{g}$  if $c= c(t)$ is respectively decreasing, increasing or constantly equal to $1$. Particular instances of pluriclosed solitons are provided by \emph{static} SKT structures $(J,\omega_0)$ (see \cite{ST}), satisfying
\begin{equation} \label{static}
(\rho^B_{\omega_0})^{1,1}=\lambda \omega_0,\,\lambda \in \R,
\end{equation}
whose corresponding solutions to the pluriclosed flow,
\[
\omega(t)=(1-\lambda t)\omega_0,\quad t \in \begin{cases} 
(\tfrac{1}{\lambda},\infty),&\lambda < 0,\\
\R, &\lambda=0, \\
(-\infty,\tfrac{1}{\lambda}),&\lambda > 0,
\end{cases}	
\]
evolve only by rescaling.

Analogously, we say that a split generalized K\"ahler structure $(J_+,J_-,g)$ on $\mathfrak{g}$ is a \emph{soliton} for the generalized K\"ahler flow if $(J_+,J_-,g(t))=(J_+,J_-,c(t)\varphi_t^*g)$.

We now briefly review the \emph{bracket flow} technique applied to the case of the pluriclosed flow, as treated in \cite{AL, EFV1, Lau}, to which we refer the reader for further details.

Let $(\mu,J,g)$ be a $2n$-dimensional Hermitian Lie algebra, with underlying vector space $\R^{2n}$ and Lie bracket $\mu \in \Lambda^2(\R^{2n})^* \otimes \R^{2n}$. As described earlier, $g$ is the standard inner product on $\R^{2n}$ while $J$ is defined by $Je_1=e_{2n}$, $Je_{2l}=e_{2l+1}$, $l=1,\ldots,n-1$, when dealing with case (1), and by  $Je_{2l-1}=e_{2l}$, $l=1,\ldots,n$, when dealing with case (2), so that the standard basis of $\R^{2n}$ is adapted to the splitting $J\nperp \oplus \mathfrak{k}_1 \oplus \nperp$ in case (1) and the splitting $\mathfrak{n}^1 \oplus J \mathfrak{n}^1 \oplus \mathfrak{k}_3 \oplus \nperp$ in case (2). To keep track of the associated algebraic data, the Lie bracket $\mu$ will be denoted by $\mu(a,0,A,\eta)$ (or $\mu(A,c)$ in the strongly unimodular case, by Remark \ref{datum_c}) in case (1) and by $\mu(\Xi)$ in case (2).

The Lie group $\text{GL}(2n,J)$ of automorphisms of $\R^{2n}$ preserving $J$ acts transitively on the set of $J$-Hermitian metrics via pullback, so that the pluriclosed flow starting from a Hermitian structure $(\mu_0,J,g)$ yields a family $(\mu_0,J,h(t)^*g)$, for some $h(t) \in \text{GL}(2n,J)$. 

One then observes that
\[
h(t) \colon (\mu_0,J,h(t)^*g) \to (h(t) \cdot \mu_0,J,g)
\]
is an isomorphism of Hermitian structures, namely $h(t)$ is a Lie algebra isomorphism which is both orthogonal and biholomorphic. Here we denoted
\[
h \cdot \mu = (h^{-1})^*\mu=h\mu(h^{-1} \cdot, h^{-1} \cdot).
\]
Let $\mu(t)=h(t) \cdot \mu_0$. Then,  up to time-dependent biholomorphisms, the pluriclosed flow starting from  $(\mu_0,J,g)$ can be interpreted as a flow $\mu(t)$ on $\Lambda^2(\R^{2n})^* \otimes \R^{2n}$, such that $\mu(t) \in \text{GL}(2n,J) \cdot \mu_0$ for all $t$. For every Lie bracket $\mu$, denote by $\rho_{\mu}^B$ the restriction to  $\Lambda^{2}T_0^*\R^{2n} \cong \Lambda^2(\R^{2n})^*$ of $\rho^B_{\omega_\mu}$, where $\omega_\mu$ is the left-invariant extension of $\omega=g(J\cdot,\cdot)$ on $\R^{2n}$ according to the Lie group operation corresponding to the bracket $\mu$.

The  evolution of
$\mu(t)$ is given by the so-called bracket flow
\begin{equation} \label{brflow_pl}
	\tfrac{d}{dt} \mu= - \pi(P_\mu)\mu,\quad \mu(0)=\mu_0,
\end{equation}
where
\[
P_\mu = \tfrac{1}{2} \omega^{-1} (\rho_\mu^B)^{1,1} \in \mathfrak{gl}_{2n},
\]
and 
\[
(\pi(A)\mu)(X,Y)=A\mu(X,Y) - \mu(AX,Y) - \mu(X,AY), \quad X,Y \in \R^{2n}
\]
for every $A \in \mathfrak{gl}_{2n}$, $\mu \in \Lambda^2(\R^{2n})^* \otimes \R^{2n}$. 
Applying a \emph{gauge} to the bracket flow \eqref{brflow_pl}, namely considering a flow of the form
\begin{equation} \label{gauged_brflow_pl}
	\tfrac{d}{dt} \bar{\mu} =\pi(P_{\bar{\mu}} - U_{\bar{\mu}})\bar{\mu}, \quad \bar{\mu}(0)=\mu_0,
\end{equation}
for some smooth map 
\[
U \colon \Lambda^2(\R^{2n})^* \otimes \R^{2n} \to \mathfrak{u}(\R^{2n},J,g),
\] then, by \cite[Theorem 2.2]{AL}, for any $\mu_0 \in \Lambda^2(\R^{2n})^* \otimes \R^{2n}$, there exist $k(t) \in \text{U}(\R^{2n},J,g)$ such that $\bar{\mu}(t) = k(t) \cdot \mu(t)=k(t)h(t) \cdot \mu_0$ for all $t$, where $\mu(t)$ and $\bar{\mu}(t)$ respectively denote the solutions to \eqref{brflow_pl} and \eqref{gauged_brflow_pl}.

This implies that, given an SKT Lie algebra $(\mu_0,J,g)$, assuming there exists a gauged bracket flow such that $\mu_0$ evolves only by rescaling, $\bar{\mu}(t)=c(t) \mu_0$, $c(t) \in \R$, then $(\mu_0,J,g)$ is a pluriclosed soliton on. The converse holds as well.

\subsection{Case (1)} We analyze the pluriclosed flow on strongly unimodular SKT Lie algebras $(\mu(A,c),J,g)$ satisfying $J \mathfrak{n}^1=\mathfrak{n}^\perp$, in terms of the algebraic data $(A,c)$ (recall Remark \ref{datum_c}). Recall that, in this case, we consider the complex structure on $\R^{2n}$ defined by $Je_1=e_{2n}$, $Je_{2l}=e_{2l+1}$, $l=1,\ldots,n-1$.

To treat the Bismut connection $\nabla^B$ and the Chern connection $\nabla^C$ (which we shall need in the following sections) together, we recall that they constitute special points (corresponding to the parameters $\tau=-1$ and $\tau=1$, respectively) of the line $\{\nabla^\tau\}_{\tau \in \R}$ of canonical Hermitian connections introduced by Gauduchon in \cite{Gau}. Explicitly, $\nabla^\tau$ is defined by
\begin{equation} \label{nablatau}
g\left(\nabla_{X}^{\tau} Y, Z\right)=g\left(\nabla_{X}^g Y, Z\right)+\frac{1-\tau}{4} T(X, Y, Z)+\frac{1+\tau}{4} C(X, Y, Z), \quad X,Y,Z \in \Gamma(TM),
\end{equation}
where $\nabla^g$ is  the Levi-Civita connection associated with $g$ and  $T$ and $C$ are given respectively by
\[
T (X, Y,Z) =-Jd\omega (X, Y, Z) =d\omega(J XJ Y,JZ), \;\; C (X, Y, Z)=-d\omega(J X, Y,Z),
\]
where $X$, $Y$ and $Z$ are vector fields. The Ricci-form $\rho^\tau$ of $\nabla^\tau$ is then defined analogously to \eqref{bismutricci}. The same goes for $\rho^\tau_\mu$, with $\mu \in \Lambda^2 (\R^{2n})^* \otimes \R^{2n}$ a bracket operation. When dealing with the Bismut and Chern connections, we shall use the notation $\rho^B$ and $\rho^C$ instead of $\rho^{-1}$ and $\rho^1$, respectively.

\begin{lemma} \label{Ricci_perp}
The Ricci form of $\nabla^\tau$ on the Hermitian Lie algebra $(\mu(a,0,A,\eta),J,g)$ is given by
\begin{equation} \label{ricciform_perp}
	\rho^\tau_{\mu(a,0,A,\eta)}=\left(-a - \tfrac{1}{2} \tau \operatorname{tr} A + \tfrac{1}{2} (\tau-1) \operatorname{tr} \eta\right)(ae^1 \wedge e^{2n} + \eta).
\end{equation}
In particular, $\rho^\tau_{\mu(a,0,A,\eta)}$ is of type $(1,1)$ with respect to $J$. 

In the strongly unimodular case, we have
\begin{equation} \label{ricciform_perp_su}
\rho^{\tau}_{\mu(0,0,A,\eta)}=\tfrac{1}{2} (\tau-1) (\operatorname{tr} \eta)\eta.
\end{equation}
\end{lemma}
\begin{proof}
Following \cite{Vez}, we know that $\rho^\tau_\mu=d \theta^{\tau}_\mu$, where $\theta^{\tau}_\mu \in (\R^{2n})^*$ is given by
\begin{equation} \label{thetatau}
\theta^\tau_\mu(X)=\frac{1}{2} \left( \operatorname{tr}(\text{ad}_X \circ J) - \tau \operatorname{tr} \text{ad}_{JX} + (\tau-1)g(\omega,dX^\flat)\right), \quad X \in \R^{2n}.
\end{equation}
Let us then compute an expression for $\theta^\tau_{\mu(a,0,A,\eta)}$. Let $X=e_1$, then
\begin{align*}
	\operatorname{tr}(\text{ad}_{e_1} \circ J)&=g([e_1,Je_1],e_1)=-a,\\
	\operatorname{tr} \text{ad}_{J e_1}&=\operatorname{tr} \text{ad}_{e_{2n}}=a+\operatorname{tr} A, \\
	g( \omega, de^1 ) &= a + g ( \omega, \eta )= a + \operatorname{tr} \eta,
\end{align*}
so that
\[
\theta^{\tau}_{\mu(a,0,A,\eta)}(e_1)= -a - \tfrac{1}{2} \tau \operatorname{tr} A + \tfrac{1}{2} (\tau-1) \operatorname{tr} \eta.
\]
Now, it is easy to see that $\theta^{\tau}_{\mu(a,0,A,\eta)}(X)=0$ for every $X \in \mathfrak{k}_1=\R\left< e_2,\ldots,e_{2n-1} \right>$. Instead, for $X=e_{2n}$, we have
\begin{align*}
\operatorname{tr}(\text{ad}_{e_{2n}} \circ J)&=\operatorname{tr} AJ\rvert_{\mathfrak{k}_1}\\
\operatorname{tr} \text{ad}_{J e_{2n}}&=-\operatorname{tr} \text{ad}_{e_1}=0, \\
g( \omega, de^{2n} ) &= 0,
\end{align*}
so that
\[
\theta^{\tau}_{\mu(a,0,A,\eta)}(e_{2n})=\tfrac{1}{2} \operatorname{tr} AJ\rvert_{\mathfrak{k}_1}.
\]
In conclusion,
\[
\theta^{\tau}_{\mu(a,0,A,\eta)}=\left(-a - \tfrac{1}{2} \tau \operatorname{tr} A + \tfrac{1}{2} (\tau-1) \operatorname{tr} \eta\right) e^1 + \tfrac{1}{2} (\operatorname{tr} AJ\rvert_{\mathfrak{k}_1}) e^{2n}.
\]
The claim then easily follows by differentiating.
\end{proof}

\begin{corollary} \label{BismutRicci_perp}
	The Bismut-Ricci form of the strongly unimodular almost nilpotent SKT Lie algebra $(\mu(A,c),J,g)$ is given by
	\begin{equation} \label{rhoB_perp}
		\rho^B_{\mu(A,c)}=-c^2 \, e^{23}.
	\end{equation}
In particular, $\rho^B_{\mu(A,c)}$ is of type $(1,1)$ with respect to $J$ and the static SKT condition is never satisfied.

The endomorphism $P_{\mu(A,c)}$ is of the form
\begin{equation} \label{Pmu_perp}
P_{\mu(A,c)}=-\tfrac{1}{2} c^2 \left( e^2 \otimes e_2 + e^3 \otimes e_3 \right).
\end{equation}
\end{corollary}
\begin{proof}
To obtain \eqref{rhoB_perp}, we just need to rewrite \eqref{ricciform_perp_su}, with $\eta= c e^{23}$ and $\tau=-1$. Then, a straightforward computation yields \eqref{Pmu_perp}.
\end{proof}

We observe that the endomorphism $P_{\mu(A,c)}$ preserves the decomposition $\R \left <e_1 \right >\oplus \mathfrak{k}_1 \oplus \R \left < e_{2n} \right >$, from which we conclude that the bracket flow starting from a bracket of the form $\mu(A_0,c_0)$ will remain of the form $\mu(A(t),c(t))$. More precisely, we obtain:

\begin{theorem}
Let $(\mu(A_0,c_0),J,g)$ be a strongly unimodular  $2n$-dimensional SKT almost nilpotent Lie algebra. Then the solution to the bracket flow starting from the bracket $\mu(A_0,c_0)$ is of the form $\mu(t)=\mu(A_0,c(t))$, with
\begin{equation} \label{c(t)}
c(t)=\frac{c_0}{\sqrt{1+2c_0^2t}},  \quad  t \in \left(-\frac{1}{2c_0^2},\infty \right).
\end{equation}
In particular, the limit  $(\mu_\infty,J,g)$ is a flat K\"ahler almost abelian Lie algebra.
\end{theorem}
\begin{proof}
Using the expression \eqref{Pmu_perp} for the endomorphism $P_{\mu(A,c)}$, one can compute
\[
\tfrac{d}{dt}(\text{ad}_{\mu(t)} e_{2n})=\text{ad}_{\tfrac{d}{dt} \mu(t)} e_{2n}=-\text{ad}_{\pi(P_{\mu(t)}) \mu(t)} e_{2n}=-\left[ P_{\mu(t)}, \text{ad}_{\mu(t)} e_{2n} \right] + \text{ad}_{\mu(t)} (P_{\mu(t)} e_{2n})=0,
\]
showing that the algebraic datum $A$ remains constant. Instead
\[
(\tfrac{d}{dt} c(t)) e_1= - \tfrac{d}{dt} \mu(t)(e_2,e_3)=-\mu(t)(P_{\mu(t)} e_2,e_3)-\mu(t)(e_2,P_{\mu(t)} e_3)=c(t)^2 \mu(t)(e_2,e_3)=-c(t)^3 e_1,
\]
that is,
\[
\tfrac{d}{dt} c(t) = -c(t)^3,
\]
yielding \eqref{c(t)}. In particular, $\lim_{t \to \infty} c(t)=0$, so that the limit bracket $\mu_\infty$ is almost abelian. The resulting Hermitian Lie algebra is K\"ahler and Ricci-flat (hence flat, by \cite{AK}) since, along the flow, one has
\[
d_{\mu(t)}\omega=c(t)\, e^2 \wedge e^3 \wedge e^{2n},\quad \rho^B_{\mu(t)}=-c(t)^2 \, e^2 \wedge e^3,
\]
both of which vanish at the limit.
\end{proof}

Using the general results in \cite{Lau0}, we therefore obtain the following result.
Here and in the rest of the paper, we shall say that a sequence of Hermitian Lie groups $(G,J,g_t)$ converges in the Cheeger-Gromov sense to a Hermitian Lie group $(G_{\infty},J_{\infty},g_{\infty})$ if and only if there exist a subsequence $\{t_l\}_{l \in \mathbb{N}}$, a collection $\{U_l\}_{k \in \mathbb{N}}$ of open sets exhausting $G_{\infty}$ and containing its identity element and maps $\varphi_l \colon U_l \to G$ which preserve the respective identity elements, which are biholomorphisms (between $J_{\infty}$ and $J$) on their images and such that $\varphi_l^*g_{t_l}$ converges to $g_{\infty}$ as $l \to \infty$ in $C^{\infty}$ topology, uniformly over compact subsets.

\begin{corollary}
Let $(J,g_0)$ be a left-invariant SKT structure on a $2n$-dimensional  strongly unimodular almost nilpotent Lie group $G$ with nilradical $\mathfrak{n} \subset \mathfrak{g}$ having one-dimensional commutator. Assume that the SKT structure satisfies $J \mathfrak{n}^1 =\nperp$. Then the pluriclosed flow starting from $(J,g_0)$ is defined for all positive times and converges in the Cheeger-Gromov sense to a flat  K\"ahler almost abelian Lie group.
\end{corollary}

\subsection{Case (2)} Assuming $J \mathfrak{n}^1 \subset \mathfrak{n}$, we analyze the behaviour of the pluriclosed flow on $6$-dimensional strongly unimodular SKT Lie algebras $(\mu(\Xi),J,g)$ with algebraic data \eqref{Xi} determining \eqref{B_sub} and \eqref{eta_sub}.

In this case, recall that we are considering the complex structure on $\R^{2n}$ given by $J e_{2l-1}=e_{2l}$, $l=1,\ldots,n$.

According to Proposition \ref{SKT_sub_prop}, we must split the discussion into two parts: we shall start with the one determined by \eqref{SKT_sub2}.

\begin{lemma}
The Bismut-Ricci form of the $6$-dimensional SKT Lie algebra $(\mu(\Xi),J,g)$ satisfying \eqref{SKT_sub2} has the following expression
\begin{equation}
\begin{aligned}
\rho^{B}_{\mu(\Xi)}=&	\tfrac{1}{c} a (c^2 + \lVert \alpha \rVert^2) e^{25}  + av_2 e^{26}- (c^2+\lVert \alpha \rVert^2) e^{34}+ \tfrac{1}{c} (c^2 + \lVert \alpha \rVert^2)J\alpha \wedge e^5 \\ 
&+ \tfrac{1}{c} \left( aA - (c^2+\lVert \alpha \rVert^2+q^2){\rm Id}_{\mathfrak{k}_3} \right)^* \alpha \wedge e^6 \\
&- \tfrac{1}{c^2} \left( \lVert \alpha \rVert^2 (c^2 + \lVert \alpha \rVert^2+q^2+a^2) + c^2v_2^2 \right)e^{56}.
\end{aligned}
\end{equation}
Hence, the $(1,1)$-part of $\rho^{-1}_{\mu(\Xi)}$ is given by
\begin{equation}
\begin{aligned}
\left( \rho^{B}_{\mu(\Xi)} \right)^{1,1}=&\tfrac{1}{2} av_2e^{15}- \tfrac{1}{2c}a(c^2+\lVert \alpha \rVert^2)e^{16} +\tfrac{1}{2c} a (c^2 + \lVert \alpha \rVert^2) e^{25}  + \tfrac{1}{2}av_2 e^{26} - (c^2+\lVert \alpha \rVert^2) e^{34} \\
&- \tfrac{1}{2c} \left( aA - (2c^2 + 2\lVert \alpha \rVert^2 + q^2 ){\rm Id}_{\mathfrak{k}_3} \right)^*(J\alpha)	\wedge e^5\\
&+\tfrac{1}{2c} \left( aA - (2c^2 + 2\lVert \alpha \rVert^2 + q^2 ){\rm Id}_{\mathfrak{k}_3} \right)^*\alpha \wedge e^6 \\
&- \tfrac{1}{c^2} \left( \lVert \alpha \rVert^2 (c^2 + \lVert \alpha \rVert^2+a^2+q^2) + c^2v_2^2 \right)e^{56}.
\end{aligned}
\end{equation}
\end{lemma}
\begin{proof}
As we did in the proof of Lemma \ref{Ricci_perp}, we can compute an expression for the $1$-form $\theta_{\mu(\Xi)}^{-1}$ in \eqref{thetatau}, exploiting that the second summand vanishes, since $\mu(\Xi)$ is unimodular. Letting $X \in \mathfrak{k}_3$ be arbitrary, we have
\begin{align*}
	\theta^{-1}_{\mu(\Xi)}(e_1)=&\tfrac{1}{2} \operatorname{tr}(\text{ad}_{e_1} \circ J)-g(\omega,de^1)
	                           =-c- \tfrac{1}{c} \lVert \alpha \rVert^2, \\
	\theta^{-1}_{\mu(\Xi)}(e_2)=&\tfrac{1}{2} \operatorname{tr}(\text{ad}_{e_2} \circ J)-g(\omega,de^2)
	                           =-v_2, \\
	\theta^{-1}_{\mu(\Xi)}(X)=&\tfrac{1}{2} \operatorname{tr}(\text{ad}_{X} \circ J)-g(\omega,dX^{\flat})
	                           =\tfrac{1}{c}(A+a\,\text{Id}_{\mathfrak{k}_3})^*\alpha(X), \\                   
	\theta^{-1}_{\mu(\Xi)}(e_5)=&\tfrac{1}{2} \operatorname{tr}(\text{ad}_{e_5} \circ J)-g(\omega,de^5)
	=-2a, \\
	\theta^{-1}_{\mu(\Xi)}(e_6)=&\tfrac{1}{2} \operatorname{tr}(\text{ad}_{e_6} \circ J)-g(\omega,de^6)
	=q,      
\end{align*}
yielding
\begin{equation}
	\theta^{-1}_{\mu(\Xi)}=-\tfrac{1}{c}(c^2+\lVert \alpha \rVert^2)e^1 - v_2e^2 + \tfrac{1}{c}(A+a\,\text{Id}_{\mathfrak{k}_3})^*\alpha - 2a e_5 + qe_6.
\end{equation}
The claim then follows by differentiation, exploiting that $A^2=-q^2\text{Id}_{\mathfrak{k}_2}$.
\end{proof}

\begin{lemma}
The symmetric endomorphism $P_{\mu(\Xi)}$ associated with the $6$-dimensional almost nilpotent SKT Lie algebra $(\mu(\Xi),J,g)$ satisfying \eqref{SKT_sub2} is given by
\[
P_{\mu(\Xi)}=-\frac{1}{4}\left( \begin{array}{c|c|c|c|c}
	0 & 0 & \phantom{\hspace{1.5em}}0\phantom{\hspace{1.5em}} & \frac{a}{c} (c^2 + \lVert \alpha \rVert^2) & a v_2 \\ \hline
	0 & 0 & 0 & -a v_2 & \frac{a}{c} (c^2 + \lVert \alpha \rVert^2)  \\ \hline
	\vphantom{\begin{pmatrix} 0 \\ 0 \\ 0 \end{pmatrix}}0 & 0 & 2(c^2 + \lVert \alpha \rVert^2){\rm Id}_{\mathfrak{k}_3} & Jz & -z \\ \hline
	\frac{a}{c} (c^2 + \lVert \alpha \rVert^2) & a v_2  & (Jz)^t  & 2r & 0 \\ \hline
	-a v_2 & \frac{a}{c} (c^2 + \lVert \alpha \rVert^2) & -z^t & 0 & 2r
\end{array} \right),
\]
with
\begin{equation}
\begin{aligned}
r(\Xi)=&\tfrac{1}{c^2}\left( \lVert \alpha \rVert^2 (c^2 + \lVert \alpha \rVert^2+a^2+q^2) + c^2(v_2^2+2a^2) \right),\\
z(\Xi)=&\tfrac{2}{c}\left((aA-(c^2+\lVert \alpha \rVert^2 + q^2){\rm Id}_{\mathfrak{k}_3})^*\alpha\right)^{\sharp}.
\end{aligned}
\end{equation}
\end{lemma}

Now, the endomorphism $P_{\mu(\Xi)}$ behaves badly with respect to the decomposition $\R \left < e_1  \right>\oplus \R \left < e_2  \right > \oplus \mathfrak{k}_3 \oplus \R  \left < e_{5} \right > \oplus \R  \left < e_{6} \right >$: for example, $P_{\mu(\Xi)}$ does not preserve the nilradical $\mathfrak{n}= \R \left<e_1,\ldots,e_{5}\right>$. This implies that, in order to translate the bracket flow into a flow for the associated algebraic data, one needs to introduce a gauge correction, provided by some skew-Hermitian endomorphism $U_{\mu(\Xi)} \in \mathfrak{u}(\R^{6},J,g)$. In our case, one can pick, for example,
\begin{equation}
	U_{\mu(\Xi)}=-\frac{1}{4}\left( \begin{array}{c|c|c|c|c}
		0 & 0 & \phantom{\hspace{1.5em}}0\phantom{\hspace{1.5em}} & -\frac{a}{c} (c^2 + \lVert \alpha \rVert^2) & -a v_2 \\ \hline
		0 & 0 & 0 & a v_2 & -\frac{a}{c} (c^2 + \lVert \alpha \rVert^2)  \\ \hline
		\vphantom{\begin{pmatrix} 0 \\ 0 \\ 0 \end{pmatrix}}0 & 0 & 0 & -Jz & z \\ \hline
		\frac{a}{c} (c^2 + \lVert \alpha \rVert^2) & a v_2  & (Jz)^t  & 0 & 0 \\ \hline
		-a v_2 & \frac{a}{c} (c^2 + \lVert \alpha \rVert^2) & -z^t & 0 & 0
	\end{array} \right),
\end{equation}
yielding the gauge-corrected endomorphism
\begin{equation} \label{Pgauge_sub1}
P_{\mu(\Xi)}-U_{\mu(\Xi)}=-\frac{1}{2}\left( \begin{array}{c|c|c|c|c}
	0 & 0 & \phantom{\hspace{1.5em}}0\phantom{\hspace{1.5em}} & \frac{a}{c} (c^2 + \lVert \alpha \rVert^2) & a v_2 \\ \hline
	0 & 0 & 0 & -a v_2 & \frac{a}{c} (c^2 + \lVert \alpha \rVert^2)  \\ \hline
	\vphantom{\begin{pmatrix} 0 \\ 0 \\ 0 \end{pmatrix}}0 & 0 & (c^2 + \lVert \alpha \rVert^2){\rm Id}_{\mathfrak{k}_3} & Jz & -z \\ \hline
	0 & 0  & 0  & r & 0 \\ \hline
	0 & 0 & 0 & 0 & r
\end{array} \right),
\end{equation}

\begin{proposition}
Let $(\mu(\Xi_0),J,g)$ be a $6$-dimensional strongly unimodular SKT almost nilpotent Lie algebra satisfying \eqref{SKT_sub2}. Then, the gauged bracket flow \eqref{gauged_brflow_pl} starting from the bracket $\mu(\Xi_0)$ is equivalent to the \textsc{ode} system
\begin{equation} \label{pluriclosed_ode_sub1}
\begin{cases}
\tfrac{d}{dt} a =-\tfrac{1}{2}ra, &a(0)=a_0,\\
\tfrac{d}{dt} q =-\tfrac{1}{2}rq, &q(0)=q_0,\\
\tfrac{d}{dt} v_2 =-(r+a^2)v_2,&v_2(0)=(v_2)_0,\\
\tfrac{d}{dt} c=-(c^2+\lVert \alpha \rVert^2)c,&c(0)=c_0,\\
\tfrac{d}{dt} \alpha=\left(-\tfrac{1}{2}r - \tfrac{1}{2c^2}(3c^4+3\lVert \alpha \rVert^2 + q^2) \right)\alpha + \tfrac{1}{2} ac^2qJ\alpha,&\alpha(0)=\alpha_0.	
\end{cases}
\end{equation}
\end{proposition}
\begin{proof}
One computes
\[
\tfrac{d}{dt}(\text{ad}_\mu e_{6})=\text{ad}_{\tfrac{d}{dt} \mu} e_{6}=-\text{ad}_{\pi(P_\mu - U_\mu) \mu} e_{6}=-\left[ P_\mu - U_\mu, \text{ad}_\mu e_{6} \right] + \text{ad}_{\mu} ((P_\mu - U_\mu)e_{6}).
\]
The evolution equations for $a$, $q$, $v_2$ and $\alpha$ now follow from an explicit computation using \eqref{Pgauge_sub1}, the matrix $B$ associated with $\text{ad}_{e_{6}}$ in \eqref{B_sub1} and the fact that $A=-qJ\rvert_{\mathfrak{k}_3}$.

As for $c$, one easily computes
\[
\tfrac{d}{dt}c\,e_1=-\tfrac{d}{dt} \mu(e_3,e_4)=-\mu((P_{\mu}-U_{\mu})e_3,e_4)-\mu(e_3,(P_{\mu}-U_{\mu})e_4)=-(c^2+\lVert\alpha\rVert^2)ce_1,
\]
which finishes the proof.
\end{proof}

\begin{proposition}
	Let $(\mu(\Xi_0),J,g)$ be a $6$-dimensional strongly unimodular SKT almost nilpotent Lie algebra satisfying \eqref{SKT_sub2}. Then, the pluriclosed flow starting from such SKT structure is defined for all positive times.
\end{proposition}
\begin{proof}
We can look at the corresponding gauged bracket flow in \eqref{pluriclosed_ode_sub1}. Then, one can easily compute that $\tfrac{d}{dt} a^2$, $\tfrac{d}{dt} q^2$, $\tfrac{d}{dt} v_2^2$, $\tfrac{d}{dt} c^2$, $\tfrac{d}{dt} \lVert\alpha\rVert^2$ are all negative. Hence, the solution to \eqref{pluriclosed_ode_sub1} is contained in a compact set for all positive times and is thus defined for all such times.
\end{proof}

\begin{example}
Consider the $6$-dimensional strongly unimodular SKT almost nilpotent Lie algebra $(\mu(\Xi_0),J,g)$ whose only nonzero algebraic data are $q_0$ and $c_0$: this corresponds to having
\[
B= q_0 (e^4 \otimes e_3 - e^3 \otimes e_4),\quad \eta=c_0 \,e^{34}.
\]
Notice that such a Lie algebra is isomorphic to $\frs{4.7} \oplus \R^2$. The, the gauged bracket flow \eqref{gauged_brflow_pl} reduces to an \textsc{ode} system for $(q(t),c(t))$. Since we have $r(\Xi_0)=0$, we obtain
\[
\begin{cases}
\tfrac{d}{dt} q= 0, &q(0)=q_0,\\
\tfrac{d}{dt} c=-c^3, &c(0)=c_0,	
\end{cases}
\]
having explicit solution
\[
q(t)=q_0, \quad c(t)=\frac{c_0}{\sqrt{1+2c_0^2t}}, \quad t \in \left( -\tfrac{1}{2c_0^2},\infty\right).
\]
As $c(t) \to 0$ as $t \to \infty$, we have that the limit SKT Lie algebra, having $q_0$ as its only nonzero piece of algebraic data, is K\"ahler, flat and almost abelian.
\end{example}

Now that we have discussed SKT structures satisfying \eqref{SKT_sub2}, we can study the ones satisfying \eqref{SKT_sub1}. As we have remarked in Section \ref{sectionSKTstruct}, this case can actually be generalized to every even dimension, by assuming \eqref{assumptionSKT}, which is why the following discussion is not limited to dimension six.

\begin{lemma} \label{Ricciform_sub}
The Ricci form of $\nabla^{\tau}$ on the $2n$-dimensional  Hermitian Lie algebra $(\mu(\Xi),J,g)$ satisfying \eqref{assumptionSKT} is given by
\[
\begin{aligned}
	\rho^{\tau}_{\mu(\Xi)}=&\tfrac{1}{2} (\tau-1) a_1v_1 e^1 \wedge e^{2n} + \tfrac{1}{2} (\tau-1)a_2 v_2 e^2 \wedge e^{2n} +\tfrac{1}{2}(\tau-1)(a_2-a_1)v_1 e^2 \wedge e^{2n-1} \\
 &+ \tfrac{1}{2}(\tau-1)(A^t v)^\flat \wedge e^{2n} + \tfrac{1}{2} \big( (\tau-1)(v_1^2 + v_2^2 + \lVert v \rVert^2) \\
 &- a(2a+ (\tau+1)a_1 + (\tau-1)a_2 + \tau \operatorname{tr} A) \big) e^{2n-1} \wedge e^{2n}.
\end{aligned}
\]
\end{lemma}
\begin{proof}
As we did in the proof of Lemma \ref{Ricci_perp}, we can compute an expression for the $1$-form $\theta^{\tau}_{\mu(\Xi)}$ in \eqref{thetatau}.
We can work with respect to the decomposition $\R^{2n}=\R \left < e_1 \right >  \oplus \R  \left < e_2 \right >  \oplus \mathfrak{k}_3 \oplus \R  \left < e_{2n-1}  \right >\oplus \R  \left < e_{2n} \right >$: then, for $l=1,2$,
\begin{align*}
	\operatorname{tr}(\text{ad}_{e_l} \circ J)&=0,\\
	\operatorname{tr} \text{ad}_{Je_l}&=0,\\
	g(\omega,de^l)&=v_1,
\end{align*}
so that
\[
\theta^{\tau}_{\mu(\Xi)}(e_l)=\tfrac{1}{2}(\tau-1)v_l,\quad l=1,2.
\]
Let $X \in \mathfrak{k}_3$ be generic. Then,
\begin{align*}
	\operatorname{tr}(\text{ad}_{X} \circ J)&=0,\\
	\operatorname{tr} \text{ad}_{JX}&=0,\\
	g(\omega,dX^\flat)&=g(v,X),
\end{align*}
yielding
\[
\theta^{\tau}_{\mu(\Xi)}(X)=\tfrac{1}{2}(\tau-1)v^\flat,\quad X \in \mathfrak{k}_3.
\]
Instead,
\begin{align*}
	\operatorname{tr}(\text{ad}_{e_{2n-1}} \circ J)&=g([e_{2n-1},e_2],e_1)+g([e_{2n-1},e_{2n}],e_{2n-1})=-a+a_2-a_1,\\
	\operatorname{tr} \text{ad}_{Je_{2n-1}}&=\operatorname{tr} \text{ad}_{e_{2n}}=a+a_1+a_2+\operatorname{tr}A,\\
	g(\omega,de^{2n-1})&=a,
\end{align*}
so that
\[
\theta^{\tau}_{\mu(\Xi)}(e_{2n-1})=-\tfrac{1}{2}(2a+(\tau+1)a_1+(\tau-1)a_2+\tau\,\operatorname{tr} A),
\]
while
\begin{align*}
	\operatorname{tr}(\text{ad}_{e_{2n}} \circ J)&=\operatorname{tr} AJ\rvert_{\mathfrak{k}_3},\\
	\operatorname{tr} \text{ad}_{Je_{2n}}&=0,\\
	g(\omega,de^{2n})&=0,
\end{align*}
yielding
\[
\theta^{\tau}_{\mu(\Xi)}(e_{2n})=\tfrac{1}{2} \operatorname{tr} AJ\rvert_{\mathfrak{k}_3}.
\]
We then obtain
\begin{equation} \label{thetatau_muXi}
\begin{aligned}
\theta^{\tau}_{\mu(\Xi)}=&\tfrac{1}{2}(\tau-1)(v_1 e^1 + v_2 e^2 + v^\flat) - \tfrac{1}{2} (2a+(\tau+1)a_1+(\tau-1)a_2 + \tau \operatorname{tr} A) e^{2n-1} \\&+ \tfrac{1}{2} \operatorname{tr} AJ\rvert_{\mathfrak{k}_3} \, e^{2n}.
\end{aligned}
\end{equation}
The claim follows by using that
\begin{align*}
	de^1&=a_1e^{1} \wedge e^{2n}+v_1e^{2n-1} \wedge e^{2n}+(a_2-a_1)e^2 \wedge e^{2n-1},\\
	de^2&=a_2e^{2} \wedge e^{2n}+v_2e^{2n-1} \wedge e^{2n},\\
	d\sigma&=\left(A^*\sigma\right) \wedge e^{2n} + \sigma(v)e^{2n-1} \wedge e^{2n}, \quad \sigma \in \mathfrak{k}_3^*, \\
	de^{2n-1}&=ae^{2n-1} \wedge e^{2n},\\
	de^{2n}&=0. \qedhere
\end{align*}
\end{proof}

We can then specialize our discussion for the Bismut-Ricci form ($\tau=-1$). Before that, we show a result concerning the Chern-Ricci form ($\tau=1$). This result holds even without the assumption \eqref{assumptionSKT}, since this only intervenes when studying the term  corresponding to  $g(\omega,dX^\flat)$ in $\theta^{\tau}$, which is not present for $\tau=1$, and when differentiating $de^1, \ldots,de^{2n-2}$, which is not a problem since, by \eqref{thetatau_muXi}, we have $\theta^1_{\mu(\Xi)} \in \R\left<e^{2n-1},e^{2n}\right>$.

\begin{remark}
By the previous result, in particular we have that
the Chern-Ricci form of the  $2n$-dimensional Hermitian Lie algebra $(\mu(\Xi),J,g)$ is 
\[
\rho^C_{\mu(\Xi)}=-a\left( a+a_1+\tfrac{1}{2}\operatorname{tr} A \right) e^{2n-1} \wedge e^{2n},
\]
while, under the assumption \eqref{assumptionSKT}, its Bismut-Ricci form is given by
\begin{equation} \label{rhoB_sub}
\begin{aligned}
	\rho^B_{\mu(\Xi)}=&-a_1v_1 e^1 \wedge e^{2n} - a_2v_2e^2 \wedge e^{2n} - (a_2-a_1)v_1e^{2}\wedge e^{2n-1}-(A^tv)^\flat \wedge e^{2n} \\
	       &- \left( a\left(a-a_2-\tfrac{1}{2}\operatorname{tr}A\right) + v_1^2 + v_2^2 + \lVert v \rVert^2 \right) e^{2n-1} \wedge e^{2n}.
\end{aligned}
\end{equation}
\end{remark}

The $(1,1)$-part of $\rho^B_{\mu(\Xi)}$, namely $(\rho^B_{\mu(\Xi)})^{1,1}=\frac{1}{2}(\rho^B_{\mu(\Xi)}+J\rho^B_{\mu(\Xi)})$, is thus given by
\begin{equation} \label{rhoB11_sub}
\begin{aligned}
	(\rho^B_{\mu(\Xi)})^{1,1}=&-\tfrac{1}{2}a_2v_2e^{1} \wedge e^{2n-1} + \tfrac{1}{2}(a_2-2a_1)v_1e^{1} \wedge e^{2n} - \tfrac{1}{2}(a_2-2a_1)v_1e^{2} \wedge e^{2n-1} \\
                   &- \tfrac{1}{2} a_2 v_2 e^{2} \wedge e^{2n} +\tfrac{1}{2}(JA^tv)^\flat \wedge e^{2n-1} - \tfrac{1}{2}(A^tv)^\flat \wedge e^{2n}  \\
	               &- \left( a\left(a-a_2-\tfrac{1}{2}\operatorname{tr}A\right) + v_1^2 + v_2^2 + \lVert v \rVert^2 \right) e^{2n-1} \wedge e^{2n}.
\end{aligned}
\end{equation}

We then obtain the following matrix expression for the endomorphism $P_{\mu(\Xi)}$:
\begin{lemma}
The symmetric endomorphism $P_{\mu(\Xi)}$ associated with the $2n$-dimensional almost nilpotent SKT Lie algebra $(\mu(\Xi),J,g)$ satisfying \eqref{assumptionSKT} is given by
\[
P_{\mu(\Xi)}=\left( \begin{array}{c|c|c|c|c}
	0 & 0 & \phantom{\hspace{1.5em}}0\phantom{\hspace{1.5em}} & -\frac{1}{4}av_1 & -\frac{1}{4}a v_2 \\ \hline
	0 & 0 & 0 & \frac{1}{4}a v_2 & -\frac{1}{4}av_1  \\ \hline
	\vphantom{\begin{pmatrix} 0 \\ 0 \\ 0 \end{pmatrix}}0 & 0 & 0 & z & Jz \\ \hline
	-\frac{1}{4}av_1 & \frac{1}{4}a v_2  & z^t  & r & 0 \\ \hline
	-\frac{1}{4}a v_2 & -\frac{1}{4}av_1 & (Jz)^t & 0 & r
\end{array} \right),
\]
where
\begin{equation} \label{constants}
	\begin{gathered}
	r(a,v_1,v_2,v,A)=-\tfrac{1}{2}a^2\left(2+\tfrac{k}{2}\right) - \tfrac{1}{2} \left( v_1^2 + v_2^2 + \lVert v \rVert^2 \right) \in \R_{<0}, \qquad k=\tfrac{1}{2}\operatorname{rk}(A+A^t),\\
	 z= - \tfrac{1}{4} A^tv \in \mathfrak{k}_3.
	\end{gathered}
\end{equation}
\end{lemma}

Now, the endomorphism $P_{\mu(\Xi)}$ behaves badly with respect to the decomposition $\R \left < e_1 \right >  \oplus \R  \left < e_2 \right >  \oplus \mathfrak{k}_3 \oplus \R \left < e_{2n-1} \right > \oplus \R  \left < e_{2n} \right > $: for example, $P_{\mu(\Xi)}$ does not preserve the nilradical $\mathfrak{n}= \R\left<e_1,\ldots,e_{2n-1}\right>$. This implies that, in order to translate the bracket flow into a flow for the associated algebraic data, one needs to introduce a gauge correction, provided by some skew-Hermitian endomorphism $U_{\mu(\Xi)} \in \mathfrak{u}(\R^{2n},J,g)$. In our case, one can pick, for example,
\begin{equation} \label{Umu}
	U_{\mu(\Xi)}=\left( \begin{array}{c|c|c|c|c}
		0 & 0 & \phantom{\hspace{1.5em}}0\phantom{\hspace{1.5em}} & \frac{1}{4}av_1 & \frac{1}{4}a v_2 \\ \hline
	0 & 0 & 0 & -\frac{1}{4}a v_2 & \frac{1}{4}av_1  \\ \hline
		\vphantom{\begin{pmatrix} 0 \\ 0 \\ 0 \end{pmatrix}}0 & 0 & \frac{a}{4}(A-A^t) & -z & -Jz \\ \hline
			-\frac{1}{4}av_1 & \frac{1}{4}a v_2  & z^t  & 0 & 0 \\ \hline
		-\frac{1}{4}a v_2 & -\frac{1}{4}av_1 & (Jz)^t & 0 & 0
	\end{array} \right),
\end{equation}
yielding
\begin{equation} \label{Pmu-Umu}
	P_{\mu(\Xi)}-U_{\mu(\Xi)}=\left( \begin{array}{c|c|c|c|c}
			0 & 0 & \phantom{\hspace{1.5em}}0\phantom{\hspace{1.5em}} & -\frac{1}{2}av_1 & -\frac{1}{2}a v_2 \\ \hline
		0 & 0 & 0 & \frac{1}{2}a v_2 & -\frac{1}{2}av_1  \\ \hline
		\vphantom{\begin{pmatrix} 0 \\ 0 \\ 0 \end{pmatrix}}0 & 0 & \frac{a}{4}(A^t-A) & 2z & 2Jz \\ \hline
		0 & 0 & 0 & r & 0 \\ \hline
		0 & 0 & 0 & 0 & r
	\end{array} \right),
\end{equation}

\begin{theorem}
Let
$(\mu(\Xi_0),J,g)$
be a $2n$-dimensional  SKT almost nilpotent Lie algebra satisfying \eqref{assumptionSKT}. Then the gauged bracket flow \eqref{gauged_brflow_pl} starting from the bracket
$\mu(\Xi_0)$
is equivalent to the \textsc{ode} system
\begin{equation} \label{ode_sub}
\begin{cases}
	\tfrac{d}{dt} a =ra, &a(0)=a_0,\\
	\tfrac{d}{dt} v_1 = 2rv_1, & v_1(0)=(v_1)_0,\\
	\tfrac{d}{dt} v_2= -a^2v_2 + 2rv_2, & v_2(0)=(v_2)_0, \\
	\tfrac{d}{dt} v = rv + Sv - \tfrac{1}{2}(v_1^2+v_2^2+\lVert v \rVert^2)v, & v(0)=v_0,\\
	\tfrac{d}{dt} A = rA & A(0)=A_0.
\end{cases}
\end{equation}
Here $S=S(a,\eps,A)$ is the symmetric endomorphism of $\mathfrak{k}_3$ defined by
\[
S=-\tfrac{1}{2} AA^t + \tfrac{a}{4}(A+A^t)-\tfrac{1}{2}a^2\left(2+\tfrac{k}{2}\right) \operatorname{Id}_{\mathfrak{k}_3},
\]
and $r$ and $k$ were defined in \eqref{constants}.
\end{theorem}
\begin{proof}
One computes
\[
\tfrac{d}{dt}(\text{ad}_\mu e_{2n})=\text{ad}_{\tfrac{d}{dt} \mu} e_{2n}=-\text{ad}_{\pi(P_\mu - U_\mu) \mu} e_{2n}=-\left[ P_\mu - U_\mu, \text{ad}_\mu e_{2n} \right] + \text{ad}_{\mu} ((P_\mu - U_\mu)e_{2n}),
\]
which yields
\footnotesize
\[
\frac{d}{dt} \left( \begin{array}{c|c|c|c}
	0 & 0  & \phantom{\hspace{0.6em}} 0 \phantom{\hspace{0.6em}} & v_1 \\ \hline
	 0 & -a & 0 & v_2  \\ \hline
	\vphantom{\begin{pmatrix} 0 \\ 0 \end{pmatrix}}0 & 0 & A & v \\ \hline
	0 & 0 & 0 & a
\end{array} \right) = 
\left( \begin{array}{c|c|c|c}
	0 & 0 & \phantom{\hspace{0.6em}} 0 \phantom{\hspace{0.6em}} & -\tfrac{1}{2}a^2 v_1 + 2rv_1 \\ \hline
	0 & -ra & 0  & -a^2v_2 + 2rv_2  \\ \hline
	\vphantom{\begin{pmatrix} 0 \\ 0 \end{pmatrix}} 0 & 0 & \tfrac{a}{4} [A,A^t]+rA & 2(Az-az+rv)+ \tfrac{a}{4}(A-A^t)v\\ \hline
	0 & 0 & 0 & ra
\end{array} \right).
\]
\normalsize
Now, since the initial data satisfies $[A_0,A_0^t]=0$, by the uniqueness of the solution we can assume $[A(t),A(t)^t]=0$ for all $t$. The claim then follows, the only non-trivial part being rewriting the evolution equation for $v$, substituting the expressions for $z$ and $r$.
\end{proof}

\begin{theorem}
Let $(J,g_0)$ be a left-invariant SKT structure on a $2n$-dimensional  almost nilpotent Lie group $G$ with nilradical $\mathfrak{n}$ of the Lie algebra $\mathfrak{g}$  of $G$ having one-dimensional commutator. Assume that the SKT structure satisfies $J\mathfrak{n}^1\subset \mathfrak{n}$ and that $\mathfrak{k}_3$ is an abelian ideal. Then the pluriclosed flow starting from $(J,g_0)$ is defined for all positive times and, after a suitable normalization, converges in the Cheeger-Gromov sense to an expanding pluriclosed soliton.
\end{theorem}
\begin{proof}
Translating the problem into the study of the corresponding bracket flow, we normalize \eqref{ode_sub} so that $a$ and $A$ remain constant. This is obtained by subtracting $r\,\text{Id}$ from the right hand-side, obtaining the new \textsc{ode}
\begin{equation} \label{norm_brflow}
\begin{cases}
	\tfrac{d}{dt} a =0, &a(0)=a_0,\\
	\tfrac{d}{dt} v_1 = rv_1, & v_1(0)=(v_1)_0,\\
	\tfrac{d}{dt} v_2= -a^2v_2 + rv_2, & v_2(0)=(v_2)_0, \\
	\tfrac{d}{dt} v = Sv - \tfrac{1}{2}(v_1^2+v_2^2+\lVert v \rVert^2)v, & v(0)=v_0,\\
	\tfrac{d}{dt} A = 0, & A(0)=A_0.
\end{cases}
\end{equation}
We claim that the solution to \eqref{norm_brflow} is defined for all positive times, with $v_1$, $v_2$ and $v$ converging to zero as $t \to \infty$.
Indeed, along \eqref{norm_brflow},
\begin{equation} \label{decr}
\begin{aligned}
\tfrac{d}{dt} v_1^2&=2rv_1^2 \leq - v_1^4,\\
\tfrac{d}{dt} v_2^2&=2a_0^2v_2^2 + 2rv_2 \leq - v_2^4,\\
\tfrac{d}{dt} \lVert v\rVert^2 &=2g(Sv,v)-(v_1^2+v_2^2+\lVert v \rVert^2)\lVert v \rVert^2 \leq -\lVert v \rVert^4,
\end{aligned}
\end{equation}
where we exploited that $a,A$ remain constant, $r \leq -\tfrac{1}{2}(v_1^2 + v_2^2 + \lVert v \rVert^2)$ and
\[
g(Sv,v) \leq -\tfrac{1}{2}a^2\left(2+\tfrac{k}{2}\right) \lVert v \rVert^2 \leq 0.
\]
Our claim then easily follows by comparison with the \textsc{ode} $\tfrac{d}{dt} x = -x^2$.

We now denote with
$\Xi(t)$ the solution to the non-normalized gauged bracket flow \eqref{ode_sub}.
Similarly to \eqref{decr}, it is easy to check that $a^2$, $v_1^2$, $v_2^2$, $\lVert v \rVert^2$ and $\lVert A \rVert^2$ are all decreasing along \eqref{ode_sub}, so that the whole solution remains inside a compact set and is thus defined for all positive times.

We also note that the limit bracket is $0$:
given that the constant $r$ in \eqref{constants} converges to the negative constant
$-\tfrac{1}{2} a_0^2(2+\tfrac{k}{2})$
along the normalized flow \eqref{norm_brflow}, it follows that there exists a positive number $r_+$ such that, along the non-normalized flow, for all $t$,
\[
\frac{r}{a^2+\lVert A \rVert^2} \leq - r_+ < 0,
\]
so that one obtains
\[
\tfrac{d}{dt}(a^2+\lVert A \rVert^2)=2r(a^2+\lVert A \rVert^2) \leq -2r_+ (a^2+\lVert A \rVert^2)^2,
\]
yielding $a \to 0$, $A \to 0$ as $t \to \infty$. Similarly to \eqref{decr}, it also easy to see directly that also $v_1$, $v_2$ and $v$ go to zero as $t \to \infty$.

Returning once again to the normalized bracket flow \eqref{norm_brflow}, we note that its limit bracket $\mu(\Xi_{\infty})$ satisfies $(v_1)_\infty=(v_2)_\infty=0$, $v_\infty=0$ and is thus a stationary solution of \eqref{norm_brflow}. It is then an (expanding) pluriclosed algebraic soliton, by \cite[Proposition 2.4]{AL}.
\end{proof}

\section{Balanced Hermitian structures} \label{balancedstructsection}

We now turn our attention to the study of another important generalization of the K\"ahler condition, namely the balanced condition. Balanced Hermitian metrics on a complex manifold $(M,J)$  of complex dimension $n$ are equivalently characterized by the coclosure of the associated fundamental form $\omega$, by the closure of $\omega^{n-1}$ or  by the vanishing of the Lee form $\theta \coloneqq Jd^* \omega$, which, we recall, is the unique $1$-form such that $d \omega^{n-1}= \theta \wedge \omega^{n-1}$. We recall the following 

\begin{lemma} {\normalfont (\cite{FP2})} \label{Leeform_generic} 
	The Lee form of a Hermitian Lie algebra $(\mathfrak{g},J,g)$ satisfies
	\begin{equation} \label{Leeform_formula_generic}
		\theta(X)=-\operatorname{tr} \text{\normalfont ad}_X + \frac{1}{2} g\left(\sum_{l=1}^{2n} [e_l,Je_l],JX\right),\quad X \in \mathfrak{g},
	\end{equation}	
	where $\{e_1,\ldots,e_{2n}\}$ is an orthonormal basis of $(\mathfrak{g},g)$.
\end{lemma}

We will  now study separately  the existence of balanced  Hermitian structures  $(J,g)$  in cases (1) and (2).

\subsection{Case (1)}
We  can prove  the following 
\begin{proposition} \label{bal_perp}
	The Lee form of the  $2n$-dimensional almost nilpotent Hermitian Lie algebra $(\mu(a,0,A,\eta),J,g)$ is given by
	\[
	\theta=(\operatorname{tr} \eta - \operatorname{tr} A) e^{2n},
	\]
	Therefore, $(\mu(a,0,A,\eta),J,g)$ is balanced if and only if
$\operatorname{tr} A= \operatorname{tr} \eta.$	

\end{proposition}
\begin{proof}
	Using \eqref{Leeform_formula_generic}, one can see that $\theta(e_1)=\theta(X)=0$ for all $X \in \mathfrak{k}_1$. To compute $\theta(e_{2n})$, we first notice that $\operatorname{tr} \text{ad}_{e_{2n}}=a+\operatorname{tr} A$, while
	\[
		g([e_1,Je_1]+[e_{2n},Je_{2n}],Je_{2n})=2g(-ae_1,-e_1)=2a,
	\]
	and
	\[
		g\left( \sum_{l=2}^{2n-1} [e_l,Je_l],Je_{2n} \right)= \sum_{l=2}^{2n-1} \eta(e_l,Je_l) = 2 \sum_{l=1}^{n-1} \eta(e_{2l},e_{2l+1}) = 2 \operatorname{tr} \eta,
	\]
	where we used that $\omega=e^{1} \wedge e^{2n} + \sum_{l=1}^{n-1} e^{2l} \wedge e^{2l+1}$. The claim then follows.
\end{proof}

\begin{proposition}
Let $\mathfrak{g}$ be a $2n$-dimensional  almost nilpotent Lie algebra with nilradical $\mathfrak{n}$ satisfying $\dim \mathfrak{n}^1 =1$. Let $(J,g)$ be a Hermitian structure on $\mathfrak{g}$ such that $J([\mathfrak{n},\mathfrak{n}])=\nperp$. Then $(J,g)$ is locally conformally balanced, i.e., $\theta$ is closed.
\end{proposition}

In real dimension six we get the following result

\begin{theorem} \label{class_bal_perp}
Let $\mathfrak{g}$ be a $6$-dimensional strongly unimodular almost nilpotent Lie algebra with nilradical $\mathfrak{n}$ satisfying $\dim \mathfrak{n}^1 =1$. Then, $\mathfrak{g}$ admits balanced  structures $(J,g)$ with $J \mathfrak{n}^1 = \nperp$ if and only if it is isomorphic to one of the following:
\begin{itemize}[leftmargin=0.6cm]
	\item[]$\mathfrak{s}_{6.159}=(f^{24}+f^{35},0,-f^{56},0,f^{36},0)$, \smallskip
	\item[]$\mathfrak{s}_{6.162}^{1}=(f^{24}+f^{35},f^{26},f^{36},-f^{46},-f^{56},0)$, \smallskip
	\item[]$\mathfrak{s}_{6.165}^a=(f^{24}+f^{35},af^{26}+f^{36},-f^{26}+af^{36},-af^{46}+f^{56},-f^{46}-af^{56},0)$, $a>0$, \smallskip
	\item[]$\mathfrak{s}_{6.166}^a=(f^{24}+f^{35},-f^{46},-af^{56},f^{26},af^{36},0)$, $0<|a|\leq 1$, \smallskip
	\item[]$\mathfrak{s}_{6.167}=(f^{24}+f^{35},-f^{36},-f^{26},f^{26}+f^{56},f^{36}-f^{46},0)$.
\end{itemize}
On these Lie algebras, examples of such balanced structures are provided by the complex structures indicated in Theorem \ref{class_cpx_perp}, together with the metric for which the basis $\{f_1,\ldots,f_6\}$ is orthonormal.
\end{theorem}
\begin{proof}
We start by noticing that, by Proposition \ref{bal_perp}, a strongly unimodular Hermitian Lie algebra $(\mu(0,0,A,\eta),J,g)$ is balanced if and only if $\operatorname{tr} \eta=0$.

Now, since $\eta$ is a $(1,1)$-form, we can assume it to be of the form
\[
\eta=b_1 e^{23} + b_2 e^{45}, \quad b_l \in \R,\,l=1,2,\quad |b_1| \geq |b_2|, \quad b_1 \neq 0,
\]
so that one has $\operatorname{tr} \eta = b_1+b_2$. Its vanishing then forces $b_2=-b_1$, so that both $b_1$ and $b_2$ must be nonzero. It follows that $\eta^2 \neq 0$, so that the nilradical $\mathfrak{n}$ must be isomorphic to $\mathfrak{h}_5$.

Now, on each of the Lie algebras from Theorem \ref{class_cpx_perp} having nilradical isomorphic to $\mathfrak{h}_5$, it is possible to find an explicit example of balanced Hermitian structure $(J,g)$ satisfying $J \mathfrak{n}^1 = \nperp$. This concludes the proof.
\end{proof}

\begin{remark}
Balanced structures on the Lie algebras $\frs{6.162}^{1}$ and $\frs{6.166}^1$ were first determined in \cite[Theorem 4.5]{FOU}. In particular, these two Lie algebras are the only $6$-dimensional unimodular (non almost abelian) almost nilpotent Lie algebras admitting balanced structures $(J,g)$ and a non-zero holomorphic $(3,0)$-form.
\end{remark}

\subsection{Case (2)} We can now proceed by focusing on balanced Hermitian metrics with respect to complex structures satisfying $J \mathfrak{n}^1 \subset \mathfrak{n}$.
\begin{proposition}
The Lee form of the almost nilpotent Hermitian Lie algebra $(\mu(\Xi),J,g)$ is given by
\[
\theta= -v_2 e^1 + (v_1+\operatorname{tr} \xi) e^2 + (Jv)^\flat - (a_1+a_2+\operatorname{tr} A) e^{2n}.
\]
Then,  $(\mu(\Xi),J,g)$ is balanced if and only if
\begin{equation} \label{balanced_sub}
v_1=-\operatorname{tr} \xi,\quad v_2=0,\quad v=0,\quad \operatorname{tr} A=-a_1-a_2.
\end{equation}

\end{proposition}
\begin{proof}
We can apply \eqref{Leeform_formula_generic}:
\begin{align*}
	\theta(e_1)&=-g([e_{2n},e_{2n-1}],Je_1)=-v_2,\\
	\theta(e_2)&=-g([e_{2n},e_{2n-1}],Je_2)+\frac{1}{2}g\left( \sum_{l=3}^{2n-2} [e_l,Je_l],Je_2\right)=v_1+\operatorname{tr} \xi,\\
	\theta(X)&=-g([e_{2n},e_{2n-1}],JX)=-g(v,JX)=(Jv)^\flat (X),\quad X \in \mathfrak{k}_3,\\
	\theta(e_{2n-1})&=0,\\
	\theta(e_{2n})&=-\operatorname{tr} \text{ad}_{e_{2n}} - g([e_{2n},e_{2n-1}],Je_{2n})=-(a_1+a_2+\operatorname{tr} A),
\end{align*}
and the claim follows.
\end{proof}

In real dimension six we have the following.

\begin{theorem} \label{class_bal_sub}
Let $\mathfrak{g}$ be a $6$-dimensional strongly unimodular almost nilpotent Lie algebra with nilradical $\mathfrak{n}$ satisfying $\dim \mathfrak{n}^1 =1$. Then $\mathfrak{g}$ admits balanced Hermitian structures $(J,g)$ satisfying $J \mathfrak{n}^1 \subset \mathfrak{n}$, but no K\"ahler structures, if and only if it is isomorphic to
\[
\frs{5.16} \oplus \R=(f^{23}+f^{46},f^{36},-f^{26},0,0,0).
\]
In particular, it is  a decomposable, $3$-step solvable, of type I and  its nilradical  $\mathfrak{n}$  is isomorphic to $\mathfrak{h}_3 \oplus \R^2$.

An example of such balanced structure on this Lie algebra is provided by the complex structure indicated in Theorem \ref{class_cpx_sub}, together with the metric for which the basis $\{f_1,\ldots,f_6\}$ is orthonormal.
\end{theorem}
\begin{proof}
Choose a balanced structure on $\mathfrak{g}$ satisfying the requirements of the statement and consider an equivalent model Hermitian Lie algebra $(\mu(\Xi),J,g)$. After imposing \eqref{balanced_sub}, we note that
\[
\operatorname{tr} \text{ad}_{e_6}=\operatorname{tr} B= a_2+\operatorname{tr} A + a = a,
\]
so that the unimodularity of $\mathfrak{g}$ imposes $a=0$. Now, equation \eqref{Lie3} reads
\[
a_2^2=0,
\]
so that we also have to set $a_2=0$. Now, \eqref{Lie4} can be written as
\begin{equation} \label{a1_eigen}
A^*(\gamma-J\alpha)=0.
\end{equation}
This implies that either $A$ is degenerate or $\gamma-J\alpha=0$ but, since $\dim \mathfrak{k}_3=2$ and $[A,J\rvert_{\mathfrak{k}_3}]=0$, the former option is equivalent to $A=0$, which would imply that $B$ is a nilpotent matrix, so that $\mathfrak{g}$ is nilpotent, so we discard this case.

Now, we note that all the previous discussion implies that the matrix $B=\text{ad}_{e_{6}}\rvert_{\mathfrak{n}}$ and the $2$-form $\eta=de^1\rvert_{\Lambda^2 \mathfrak{n}}$ are of the form
\begin{equation} \label{B_eta_bal_sub}
B=\begin{pmatrix} 
0 & 0 & p  & q & -c \\
0 & 0 & -q & p & 0 \\
0 & 0 & 0 & b & 0 \\
0 & 0 & -b & 0 & 0 \\
0 & 0 & 0 & 0 & 0
\end{pmatrix}, \quad \eta=c\,e^{34},\qquad b,c \in \R-\{0\},\; \alpha=pe^3+qe^4 \in \mathfrak{k}_3^*,
\end{equation}
so that one has
\begin{equation} \label{streq_bal_sub}
\begin{gathered}
de^1=pe^{36}+qe^{46}-ce^{56}+ce^{34}, \qquad de^2=-qe^{36}+pe^{46}, \\
de^3=-be^{46}, \qquad de^4=-be^{36}, \qquad de^5=de^6=0.
\end{gathered}
\end{equation}
An explicit isomorphism with the Lie algebra $\frs{5.16} \oplus \R$ can then be obtained by setting
\begin{alignat*}{3}
f_1&=ce_1,&\quad f_2&=\tfrac{q}{b} e_1+\tfrac{p}{b}e_2+e_3,&\quad f_3&=-\tfrac{p}{b}e_1 + \tfrac{q}{b} e_2 + e_4, \\
f_4&=-be_5, &\quad f_5&=e_2, &\quad f_6&=\tfrac{1}{b} e_6,
\end{alignat*}
concluding the proof.
\end{proof}

To end this section, we recall that it was conjectured in \cite{FV1} that a compact complex manifold $(M,J)$ cannot admit both SKT and balanced metrics unless it admits K\"ahler metrics as well.
Examining our classification results, we obtain the following related result.

\begin{corollary}
Let $\Gamma \backslash G$ be a $6$-dimensional almost nilpotent solvmanifold, with the nilradical $\mathfrak{n}$ of $G$ having one-dimensional commutator. Then, $\Gamma \backslash G$ cannot admit both  invariant SKT structures $(J,g)$ satisfying either $J\mathfrak{n}^1 = \nperp$ or $J\mathfrak{n}^1 \subset \mathfrak{n}$ and invariant  balanced structures $(J^\prime,g^\prime)$ satisfying either $J^\prime \mathfrak{n}^1 = \mathfrak{n}^{\perp_{g^\prime}}$ or $J^\prime \mathfrak{n}^1 \subset \mathfrak{n}$.
\end{corollary}

\section{Balanced flow} \label{balancedflowsection}
In \cite{BV},  a parabolic flow for balanced Hermitian metrics on a complex manifold was introduced. The evolution equation for the $(n-1, n-1)$-form $\varphi(t)$  is
\begin{equation} \label{flowBV}
	\tfrac{\partial}{\partial t}\varphi(t)=i \partial \overline\partial *_t \left(\rho_{\omega(t)}^C \wedge *_t \varphi(t)\right) +\Delta_{BC} \varphi(t),\quad \varphi(0)=\varphi_0=*_0 \omega_0,
\end{equation}
or equivalently
\begin{equation} \label{flowBV2}
	\tfrac{\partial}{\partial t}\omega(t)=(n-2)!\,\iota_{\omega(t)^{n-2}}\left(i\partial \overline{\partial} *_t \left(\rho_{\omega(t)}^C \wedge \omega(t) \right) \right) + \frac{1}{n-1} \, \iota_{\omega(t)^{n-2}} \Delta_{BC} \omega(t)^{n-1},\quad \omega(0)=\omega_0.
\end{equation}
Here $*_t$ denotes the Hodge star operator associated with $\omega(t)$ and $\Delta_{BC}$ is the modified Bott-Chern Laplacian
\[
\Delta_{B C}=\partial \overline{\partial \partial}^{*} \partial^{*}+\bar{\partial}^{*} \partial^{*} \partial \bar{\partial}+\bar{\partial}^{*} \partial \partial^{*} \bar{\partial}+\partial^{*} \overline{\partial \partial}^{*} \partial+\bar{\partial}^{*} \bar{\partial}+\partial^{*} \partial
\]
associated with $\omega(t)$. The flow \eqref{flowBV} differs from the one introduced in \cite{BV} by a factor in the second summand in the right-hand side, as suggested by the authors (see \cite{FP1}). We shall refer to this flow as \emph{balanced flow} and we shall denote the right-hand side of \eqref{flowBV2} by $q(\omega(t))$ for simplicity.  As shown in \cite{BV}, \eqref{flowBV} preserves the balanced condition of the initial data: in particular, its solution $\varphi(t)$ satisfies $[\varphi(t)]=[\varphi_0] \in H_{BC}^{2n-2}={\ker d}/{\operatorname{im} \partial \overline\partial}$ for all times.

When the initial data is K\"ahler, the balanced flow reduces to the \emph{Calabi flow} (see \cite{Cal}),
\begin{equation} \label{calabiflow}
	\tfrac{\partial}{\partial t} \omega(t)=i\partial \overline\partial s_{\omega(t)},\quad \omega(0)=\omega_0,
\end{equation}
where $s_{\omega(t)}$ denotes the scalar curvature of $\omega(t)$, 

On a connected Lie group, analogously to the situation in Section \ref{sec_plflow}, we have that left-invariant initial data yield left-invariant solutions, so that, once again, the flow \eqref{flowBV} reduces to an \textsc{ode} on the associated Lie algebra.

Moreover, since the scalar curvature of a left-invariant metric on a Lie group is constant, it is immediate to see that left-invariant K\"ahler metrics on Lie groups constitute stationary solutions of the Calabi flow and, consequently, of the balanced flow.

To study the balanced flow on the balanced Hermitian Lie algebras described in the previous section, once again we apply the bracket flow technique.

In particular, as already described in \cite{FP1}, the bracket flow equation corresponding to the balanced flow is given by
\begin{equation} \label{brflow_bal}
	\tfrac{d}{dt}\mu= - \pi(Q_\mu)\mu,\quad \mu(0)=\mu_0.
\end{equation}
Denoting by $q(\omega(t))$ the right-hand side of \eqref{flowBV2}, for every bracket $\mu \in \Lambda^2(\R^{2n})^* \otimes \R^{2n}$, we denote by $\omega_{\mu}$ the left-invariant extension of $\omega=g(J \cdot, \cdot)$ on $\R^{2n}$ according to the Lie group operation corresponding to the bracket $\mu$. We can then denote by $q_{\mu}$ the restriction to $\Lambda^2 T_0^* \R^{2n} \cong \Lambda^2 (\R^{2n})^*$ of $q(\omega_\mu)$ so that we can finally define
\begin{equation} \label{Qmu}
	Q_\mu = - \tfrac{1}{2} \omega^{-1} q_{\mu} \in \mathfrak{gl}_{2n}.
\end{equation}

Again, we shall also consider gauge-corrected bracket flows of the form
\begin{equation} \label{brflow_gauge_bal}
	\tfrac{d}{dt}\mu= - \pi(Q_\mu-U_{\mu})\mu,\quad \mu(0)=\mu_0,
\end{equation}
where $U \colon \Lambda^2(\R^{2n})^* \otimes \R^{2n} \to \mathfrak{u}(\R^{2n},J,g)$ is some smooth map.

\subsection{Case (1)} 
We can start by determining an expression for the endomorphism $Q_{\mu}$ in case (1).

\begin{lemma}
The endomorphism $Q_{\mu(a,0,A,\eta)}$ associated with the $2n$-dimensional balanced Hermitian Lie algebra $(\mu(a,0,A,\eta),J,g)$ is of the form
\[
Q_{\mu(a,0,A,\eta)}=\begin{pmatrix} p & 0 & 0 \\ 0 & P & 0 \\ 0 & 0 & p \end{pmatrix},
\]
with $p=p(a,A,\eta)$ a real number and $P=P(a,A,\eta)$ a symmetric endomorphism of $\mathfrak{k}_1$ commuting with $J\rvert_{\mathfrak{k}_1}$. The entries of $Q_{\mu(a,0,A,\eta)}$ are all homogeneous fourth-order polynomials in $a$, the coefficients of $\eta$ and the entries of $A$.
\end{lemma}
\begin{proof}
	Recalling the definition of $Q_\mu$, it will be enough to show that both $\Delta_{BC} \omega^{n-1}$ and $i \partial \overline\partial * (\rho^C_{\omega} \wedge \omega)$ lie in $e^1 \wedge e^{2n} \wedge \Lambda^{2n-4} \mathfrak{k}_1^* \oplus \Lambda^{2n-2} \mathfrak{k}_1^*$. This can be easily achieved by recalling that $\omega,\rho^C_{\omega},de^1 \in \R e^1 \wedge e^{2n} \oplus \Lambda^2 \mathfrak{k}_1^*$, $d\mathfrak{k}_1^* \subset \mathfrak{k}_1^* \wedge e^{2n}$, $de^{2n}=0$ and the fact that $*(\Lambda^4 \mathfrak{k}_1^*) \subset e^1 \wedge e^{2n} \wedge \Lambda^{2n-6} \mathfrak{k}_1^*$, $*(e^1 \wedge e^{2n} \wedge \Lambda^2 \mathfrak{k}_1^*) \subset \Lambda^{2n-4} \mathfrak{k}_1^*$ and $\rho^C_{\omega} \in \R e^1 \wedge e^{2n} \oplus \Lambda^2 \mathfrak{k}_1^*$.
\end{proof}

In the $6$-dimensional strongly unimodular case, an explicit computation shows
\begin{equation} \label{pPbal}
\begin{aligned}
p(A,\eta)&=-\frac{m}{32} \lVert A^+ \rVert^2 + \frac{1}{16} \lVert \eta \rVert^4,\\
P(A,\eta)&=\frac{m}{32}[A,A^t]- \frac{1}{16} \lVert \eta \rVert^4\,\text{Id}_{\mathfrak{k}_1},
\end{aligned}
\end{equation}
with
\[
m=m(A)=4\lVert A \rVert^2 - (\operatorname{tr}(JA))^2.
\]

\begin{theorem} \label{ode_balflow_perp}
The bracket flow \eqref{brflow_bal} starting from the  $6$-dimensional strongly unimodular balanced Hermitian Lie algebra $(\mu(a_0,0,A_0,\eta_0),J,g)$ is equivalent to the \textsc{ode} system
\[
\begin{cases}
\tfrac{d}{dt} a= pa, & a(0)=a_0, \\
\tfrac{d}{dt} A = [A,P] + pA, & A(0)=A_0, \\
\tfrac{d}{dt} \eta = P^*\eta - p\eta, & \eta(0)=\eta_0.
\end{cases}
\]
\end{theorem}
\begin{proof}
Using the definition of the bracket flow \eqref{brflow_bal}, and as we did previously for the pluriclosed flow, we have
\[
\tfrac{d}{dt}(\text{ad}_\mu e_{2n})=\text{ad}_{\tfrac{d}{dt} \mu} e_{2n}=-\text{ad}_{\pi(Q_\mu) \mu} e_{2n}=-\left[ Q_\mu, \text{ad}_\mu e_{2n} \right] + \text{ad}_{\mu} (Q_\mu e_{2n}),
\]
which, after a quick computation, can be reinterpreted as
\[
\frac{d}{dt} \begin{pmatrix} a & 0 \\ 0 & A \end{pmatrix} = \begin{pmatrix} pa & 0 \\ 0 & [A,P] + pA \end{pmatrix},
\]
yielding the first two equations. To obtain the third one, let $X,Y \in \mathfrak{k}_1$ be arbitrary and compute
\begin{align*}
\left(\tfrac{d}{dt} \eta (X,Y)\right) e_1=& -\tfrac{d}{dt} \mu(X,Y) = Q_\mu (\mu(X,Y)) - \mu(Q_\mu X,Y)-\mu(X,Q_\mu Y) \\
=& (- p \eta(X,Y) + \eta(PX,Y) + \eta(X,PY))e_1.
\end{align*}
The claim then follows.
\end{proof}

\begin{example}
Let us consider the simply connected  strongly unimodular  $6$-dimensional Lie  group with Lie algebra $\mathfrak{g}$ having structure equations
\[
df^1=q_0f^{23}-q_0f^{45},\quad df^2=r_0f^{36},\quad df^3=-r_0f^{26},\quad df^4=s_0f^{56},\quad df^5=-s_0f^{46},\quad df^6=0,
\]
where $q_0,r_0 \in \R-\{0\}$, $s_0 \in \R$. We note that the Lie algebra $\mathfrak{g}$ is isomorphic to $\frs{6.159}$ if $s_0=0$, while it is isomorphic to $\frs{6.166}^a$ if $s_0 \neq 0$, with $a=\frac{s_0}{r_0}$ or $a=\frac{r_0}{s_0}$, depending on whether one has $|s_0| \leq |r_0|$ or the converse, respectively. Let us consider the balanced Hermitian structure $(J,g_0)$ on $\mathfrak{g}$ defined by
\[
Jf_1=f_6,\quad Jf_2=f_3,\quad Jf_4=f_5
\]
and
\[
g_0=\sum_{l=1}^6 (f^l)^2,
\]
so that the corresponding fundamental form is $\omega_0=f^{16}+f^{23}+f^{45}$.

Notice that $\mathfrak{g}$ is almost nilpotent, with nilradical isomorphic to $\mathfrak{h}_5$, and that $(J,g_0)$ satisfies  $J \mathfrak{n}^1=\mathfrak{n}^{\perp_{g_0}}$.

To solve the balanced flow starting from  $\omega_0$ directly, let us assume the solution $\omega(t)$ to be 
\[
\omega(t)=u_1(t)f^{16}+u_2(t)f^{23}+u_2(t)f^{45},
\]
for some smooth positive functions $u_1(t)$ and $u_2(t)$, with 
$u_1(0)=u_2(0)=1$. One can easily see that the algebraic data $(A(t),\eta(t))$ associated with $(J,\omega(t))$ are of the form
\[
A(t)=\begin{pmatrix} 
0 & \sqrt{u_1(t)}r_0 & 0 & 0 \\
-\sqrt{u_1(t)}r_0 & 0 & 0 & 0 \\
0 & 0 & 0 & \sqrt{u_2(t)}s_0 \\
0 & 0 & -\sqrt{u_2(t)}s_0 & 0
 \end{pmatrix}, \quad \eta(t)=\tfrac{\sqrt{u_1(t)}}{u_2} (q_0e^{23}-q_0e^{45}).
\]

By Proposition \ref{bal_perp} and Lemma \ref{Ricci_perp} we then conclude that $(J,\omega(t))$ is balanced and Chern-Ricci flat, no matter the values of $q_0,r_0,s_0,u_1(t),u_2(t)$.

An explicit computation yields
\[
\tfrac{1}{2} \Delta_{BC} \omega(t)^2 = \frac{u_1(t)^2}{u_2(t)^2} q_0^4 f^{2345},
\]
so that one has
\[
\iota_{\omega(t)}\left(\tfrac{1}{2} \Delta_{BC} \omega(t)^2\right)=-\frac{1}{2} \frac{u_1(t)^3}{u_2(t)^4} q_0^4 f^{16} + \frac{1}{2}\frac{u_1(t)^2}{u_2(t)^3} q_0^4 (f^{23}+f^{45}),
\]
so that the balanced flow reduces to the \textsc{ode} system
\[
\begin{dcases}
 \tfrac{d}{dt} u_1= -\frac{1}{2} \tfrac{u_1^3}{u_2^4} q_0^4, &u_1(0)=1,\\
\tfrac{d}{dt} u_2= -\frac{1}{2} \tfrac{u_1^2}{u_2^3} q_0^4, &u_2(0)=1,
\end{dcases}
\]
with explicit solution
\[
u_1(t)=(3q_0^4t+1)^{-\frac{1}{6}},\quad u_2(t)=(3q_0^4t+1)^{\frac{1}{6}}, \quad t \in (- \tfrac{1}{3q_0^4},\infty).
\]
This solution is thus immortal.

Looking at Theorem \ref{ode_balflow_perp}, we can assume the solution $\mu(t)=\mu(A(t),\eta(t))$ to \eqref{brflow_bal} to be of the form
\[
A(t)=\begin{pmatrix}
 0 & r(t) & 0 & 0 \\
-r(t) & 0 & 0 & 0 \\
0 & 0 & 0 & s(t) \\
0 & 0 & -s(t) & 0 \end{pmatrix}, \quad \eta(t)=q(t)f^{23} - q(t)f^{45},
\]
yielding, by \eqref{pPbal},
\[
p(t)=\tfrac{1}{4} q(t)^4, \quad P(t)=-\tfrac{1}{4} q(t)^4 \text{Id}_{\mathfrak{k}_1}.
\]
The bracket flow then reduces to the \textsc{ode} system
\[
\begin{cases}
\tfrac{d}{dt} q= -\tfrac{3}{4} q^5, & q(0)=q_0,\\
\tfrac{d}{dt} r= \tfrac{1}{4} q^4 r, & r(0)=r_0, \\
\tfrac{d}{dt} s= \tfrac{1}{4} q^4 s, & s(0)=s_0, \\
\end{cases}
\]
which has solution
\[
q(t)=q_0(3q_0^4 t +1)^{-\tfrac{1}{4}}, \quad r(t)=r_0(3q_0^4 t +1)^{\tfrac{1}{12}}, \quad s(t)=s_0(3q_0^4 t +1)^{\tfrac{1}{12}}, \quad t \in (-\tfrac{1}{3q_0^4},\infty).
\]
We observe that $q(t) \to 0$, while $r(t),s(t) \to \infty$ as $t \to \infty$. The norm-normalized solution $\tilde{\mu}(t) \coloneqq \frac{\mu(t)}{\lVert \mu(t) \rVert}$ then has limit $\tilde{\mu}_{\infty}$, determined by the following parameters:
\[
\tilde{q}_{\infty}=0,\quad \tilde{r}_{\infty}=r_0 \frac{\sqrt{q_0^2+r_0^2+s_0^2}}{\sqrt{r_0^2+s_0^2}}, \quad \tilde{s}_{\infty}=s_0 \frac{\sqrt{q_0^2+r_0^2+s_0^2}}{\sqrt{r_0^2+s_0^2}}
\]
It follows that $(\tilde{\mu}_{\infty},J,g_0)$ is a flat K\"ahler almost abelian Lie algebra.

Therefore, by \cite{Lau0}, we can conclude that the balanced flow starting from $(G,J,g_0)$ converges in the Cheeger-Gromov sense, after a suitable normalization (multiplication by $(1+t)^{-\frac{1}{6}}$, as $\lVert \mu(t) \rVert^2$ increases as $t^{\frac{1}{6}}$, as $t \to \infty$), to a flat K\"ahler almost abelian Lie group.
\end{example}

\subsection{Case (2)} By the proof of Theorem \ref{class_bal_sub}, the algebraic data associated with the $6$-dimensional strongly unimodular almost nilpotent balanced Hermitian Lie algebra $(\mu(\Xi),J,g)$ reduce to the quadruple $(b,c,p,q) \in \R^4$ determining \eqref{B_eta_bal_sub}. For this reason, from now on we denote such a Hermitian Lie algebra by $(\mu(b,c,p,q),J,g)$.

\begin{lemma}
The endomorphism $Q_{\mu(b,c,p,q)}$ associated with the  $6$-dimensional strongly unimodular balanced Hermitian Lie algebra $(\mu(b,c,p,q),J,g)$ is of the form
\begin{equation}
	Q_{\mu(b,c,p,q)}=
	\begin{pmatrix}
	 h & 0 & \tfrac{1}{2}bqk & -\tfrac{1}{2}bpk & 0 & 0 \\
	 0 & h & \tfrac{1}{2}bpk & \tfrac{1}{2}bqk  & 0 & 0 \\
	 \tfrac{1}{2}bqk & \tfrac{1}{2}bpk & -h  & 0    & 0 & 0 \\
	 -\tfrac{1}{2}bpk & \tfrac{1}{2}bqk & 0   & -h   & 0 & 0 \\
	 0 & 0 & 0   & 0    & -h & 0 \\
	 0 & 0 & 0   & 0    & 0 & -h 	
	\end{pmatrix},
\end{equation}
where
\begin{equation} \label{hk_balanced}
h(b,c,p,q)=\tfrac{1}{4}\big((c^2+p^2+q^2)^2 + \tfrac{1}{2}b^2(p^2+q^2) \big), \quad k(b,c,p,q)=\tfrac{1}{2}\left(c^2+p^2+q^2+\tfrac{1}{2}b^2\right).
\end{equation}
\end{lemma}
\begin{proof}
We first notice that the results concerning the Chern-Ricci form in Lemma \ref{Ricciform_sub} hold even without the assumptions \eqref{assumptionSKT}, from which we conclude that the Hermitian Lie algebra $(\mu(b,c,p,q),J,g)$ is Chern-Ricci flat. Hence, the first summand of $q_{\mu(b,c,p,q)}$ vanishes, and now an explicit computation using that $\omega=e^{12}+e^{34}+e^{56}$ and the structure equations \eqref{streq_bal_sub} yields
\begin{align*}
q_{\mu(b,c,p,q)}&=\tfrac{1}{2} \iota_{\omega} \Delta_{BC} \omega^2= \tfrac{1}{4} \iota_{\omega}(dJd * dJd \omega) \\
                &=2 h (-e^{12}+e^{34}+e^{56}) - bpk (e^{13}+e^{24}) + bqk (-e^{14}+e^{23}),
\end{align*}
with $h$ and $k$ from \eqref{hk_balanced}. Now, a straightforward computation shows that we have $q_{\mu(b,c,p,q)}=-\tfrac{1}{2} \omega(Q_{\mu(b,c,p,q)}\cdot,\cdot)$, with $Q_{\mu(b,c,p,q)}$ being the symmetric endomorphism of $\R^6$ of the statement, concluding the proof.
\end{proof}

In order for the bracket flow \eqref{brflow_bal} to preserve Lie brackets of the form $\mu(b,c,p,q)$, we need to introduce a gauge correction, provided for instance by the skew-Hermitian endomorphism
\[
U_{\mu(b,c,p,q)}=\begin{pmatrix} 
0 & 0 & -\tfrac{1}{2}bqk & \tfrac{1}{2}bpk & 0 & 0 \\
0 & 0 & -\tfrac{1}{2}bpk & -\tfrac{1}{2}bqk  & 0 & 0 \\
\tfrac{1}{2}bqk & \tfrac{1}{2}bpk & 0  & 0    & 0 & 0 \\
-\tfrac{1}{2}bpk & \tfrac{1}{2}bqk & 0   & 0   & 0 & 0 \\
0 & 0 & 0   & 0    & 0 & 0 \\
0 & 0 & 0   & 0    & 0 & 0	
 \end{pmatrix} \in \mathfrak{u}(\R^{2n},J,g),
\]
yielding the gauge-corrected endomorphism
\[
Q_{\mu(b,c,p,q)}-U_{\mu(b,c,p,q)}=
\begin{pmatrix}
	h & 0 & bqk & -bpk & 0 & 0 \\
	0 & h & bpk & bqk  & 0 & 0 \\
	0 & 0 & -h  & 0    & 0 & 0 \\
	0 & 0 & 0   & -h   & 0 & 0 \\
	0 & 0 & 0   & 0    & -h & 0 \\
	0 & 0 & 0   & 0    & 0 & -h 	
\end{pmatrix}.
\]

\begin{proposition}
The gauged bracket flow \eqref{brflow_gauge_bal} starting from a $6$-dimensional strongly-unimodular almost nilpotent Lie bracket $\mu(b_0,c_0,p_0,q_0)$ is equivalent to the following \textsc{ode}:
\begin{equation}
	\begin{cases}
		\tfrac{d}{dt} b = -hb, &b(0)=b_0, \\
		\tfrac{d}{dt} c = -3hc, &c(0)=c_0, \\
		\tfrac{d}{dt} p = -(kb^2+3h)p, &p(0)=p_0, \\
		\tfrac{d}{dt} q = -(kb^2+3h)q, &q(0)=q_0,
	\end{cases}
\end{equation}	
where $h$ and $k$ were defined in \eqref{hk_balanced}.
\end{proposition}
\begin{proof}
As we did previously, we can compute
\[
\tfrac{d}{dt}(\text{ad}_{\mu(t)} e_{6})=\text{ad}_{\tfrac{d}{dt} \mu(t)} e_{6}=-\text{ad}_{\pi(Q_{\mu(t)}-U_{\mu(t)}) \mu(t)} e_{6}=-\left[ Q_{\mu(t)}-U_{\mu(t)}, \text{ad}_{\mu(t)} e_{6} \right] + \text{ad}_{\mu(t)} ((Q_{\mu(t)}-U_{\mu(t)}) e_{6}),
\]
obtaining
\[
\frac{d}{dt} \begin{pmatrix} 
	0 & 0 & p  & q & -c \\
	0 & 0 & -q & p & 0 \\
	0 & 0 & 0 & b & 0 \\
	0 & 0 & -b & 0 & 0 \\
	0 & 0 & 0 & 0 & 0
\end{pmatrix} = 
\begin{pmatrix} 
	0 & 0 & -(kb^2+3h)p  & -(kb^2+3h)q & 3hc \\
	0 & 0 & (kb^2+3h)q & -(kb^2+3h)p & 0 \\
	0 & 0 & 0 & -hb & 0 \\
	0 & 0 & hb & 0 & 0 \\
	0 & 0 & 0 & 0 & 0
\end{pmatrix}.
\]
The claim now follows.
\end{proof}

\begin{theorem}
Let $(J,\omega_0)$ be a left-invariant balanced Hermitian structure on a $6$-dimensional almost nilpotent Lie group $G$ such that the nilradical $\mathfrak{n}$ of its Lie algebra has one-dimensional commutator. Assume that	$J \mathfrak{n}^1 \subset \mathfrak{n}$. Then, the solution of the balanced flow starting from $(J,\omega_0)$ is defined for all positive times and converges in the Cheeger-Gromov sense to a flat K\"ahler almost abelian Lie group.
\end{theorem}
\begin{proof}
A quick analysis of the \textsc{ode} system shows that the norm of each term $b,c,p,q$ is non-increasing along the flow, thanks to the fact that both $h$ and $k$ are non-negative: in particular, the solution $(b(t),c(t),p(t),q(t))$ remains inside a compact set of $\R^4$ for $t>0$ and is thus defined for all positive times, converging to some limit $(b_{\infty},c_{\infty},p_{\infty},q_{\infty})$, since the norm of every term is non-increasing.

Moreover, looking at the exact expression for $h$ in \eqref{hk_balanced} one can compute
\[
\tfrac{d}{dt} (c^2+p^2+q^2) \leq -6h(c^2+p^2+q^2) \leq -\tfrac{3}{2}(c^2+p^2+q^2)^3.
\]
Then, a comparison with the \textsc{ode} $\tfrac{d}{dt} x = -\tfrac{3}{2} x^3$ shows that the terms $c$, $p$ and $q$ converge to zero as $t \to \infty$. 

The limit Hermitian Lie algebra is therefore of the form $(\mu(b_{\infty},0,0,0),J,g)$, which is a flat K\"ahler almost abelian Lie algebra, in particular isomorphic either to $\R^6$ or to
\[
\frs{3.3}^0 \oplus \R^3=(f^{26},-f^{16},0,0,0,0),
\]
in the notations of \cite{SW}.

The result regarding Cheeger-Gromov convergence now follows from \cite{Lau0}.
\end{proof}

\section{Generalized K\"ahler  structures}  \label{sectionGKsection}

Having studied the SKT condition, we now study the existence of generalized K\"ahler structures.  In particular, we  show some non-existence results in the strongly unimodular $6$-dimensional case. For the non-split case, this will involve the study of holomorphic Poisson structures.

We recall that holomorphic Poisson structures are $(2,0)$-vector fields lying in the kernel of the Cauchy-Riemann operator (see \cite{Gau})
\[
\overline\partial_X(Y \wedge Z) \coloneqq \overline\partial_XY \wedge Z + Y \wedge \overline\partial_XZ=[X,Y]^{1,0} \wedge Z + Y \wedge [X,Z]^{1,0},
\]
where $X$ is a $(0,1)$-vector field and $Y,Z$ are $(1,0)$-vector fields, and in the isotropic cone of the Schouten bracket
\[
[X_0 \wedge X_1,Y_0 \wedge Y_1]=\sum_{j,k=0}^1 (-1)^{j+k}[X_j,Y_k] \wedge X_{j+1}\wedge Y_{k+1},
\]
where $X_l,Y_l$, $l=0,1$, are vector fields and indices are meant mod $2$.

Following the previous section, we separate our discussion for case (1)  and case (2).  To do this, we shall use the characterization results of Proposition \ref{perp_SKTprop} and Proposition \ref{sub_SKTprop}, respectively.

\subsection{Case (1)}
Recall that, in Theorem \ref{class_SKT_perp}, we proved that a $6$-dimensional strongly unimodular almost nilpotent Lie algebra with nilradical having one-dimensional commutator admits SKT structures with $\hat\theta=0$ if and only if it is isomorphic to either $\mathfrak{h}_3 \oplus \mathfrak{s}_{3.3}^0$, $\mathfrak{s}_{4.7} \oplus \R^2$ or $\mathfrak{g}_{6.52}^{0,b}$, $s \neq 0$.

If one adds the requirement that $(\mathfrak{g},J)$ admits non-zero holomorphic Poisson structures, in Theorem \ref{class_SKT_perp}, then one obtains the Lie algebras $\mathfrak{h}_3 \oplus \mathfrak{s}_{3.3}^0$ and $\mathfrak{s}_{4.7} \oplus \R^2$ only.
To see this, following the proof of Theorem \ref{class_SKT_perp}, consider an adapted unitary basis $\{e_l \}$, $Je_1=e_6$, $Je_2=e_3$, $Je_4=e_5$, having structure equations
\begin{equation} \label{streq6}
de^1=ce^{23},\quad de^2=b_1 e^{36},\quad de^3=-b_1 e^{26},\quad de^4=b_2 e^{56},\quad de^5=-b_2e^{46},\quad de^6=0,
\end{equation}
for some $c,b_1,b_2 \in \R$, $c \neq 0$, $b_1^2+b_2^2 \neq 0$. A basis of $(1,0)$-vectors is provided by
\[
Z_1=e_1-ie_6,\quad Z_2=e_2-ie_3,\quad Z_3=e_4-ie_5.
\]
so that, in the basis $\{Z_1,Z_2,Z_3,\overline{Z}_1,\overline{Z}_2,\overline{Z}_3\}$ for $\mathfrak{g} \otimes \C$, the non-zero brackets, up to exchanging the terms, are given by
\begin{equation} \label{brackcpx}
\begin{gathered} 
[Z_1,Z_l]=-b_{l-1}Z_l,\quad [Z_1,\overline{Z}_l]=b_{l-1} \overline{Z}_l,\quad [\overline{Z}_1,Z_l]=b_{l-1}Z_l, \\ [\overline{Z}_1,\overline{Z}_l]=-b_{l-1}\overline{Z}_l, \quad [Z_2,\overline{Z}_2]=-ic(Z_1+\overline{Z}_1),
\end{gathered}
\end{equation}
for $l=2,3$. A basis for $\mathfrak{g}^{2,0}$ is given by $\{Z_{12},Z_{13},Z_{23}\}$, where $Z_{kl\ldots}$ is a short-hand for $Z_k \wedge Z_l \wedge \ldots$. Then, using \eqref{brackcpx}, one obtains $\overline\partial_{\overline{Z}_2}Z_{23}=ic Z_{13} \neq 0$, while $\overline\partial_{\overline{Z}_2}Z_{1l}=0$, $l=2,3$, so that holomorphic Poisson structures may only be found inside $V \coloneqq \mathbb{C} \left<Z_{12},Z_{13}\right>$. The operator $\overline\partial_{\overline{Z}_3}$ vanishes identically, so that we only need to focus on $\overline\partial_{\overline{Z}_1}$, which preserves $V$, and the Schouten bracket $[\cdot,\cdot]$, seen respectively as an endomorphism and a complex-valued symmetric bilinear form on $V$, after identifying $\mathfrak{g}^{3,0}$ with $\mathbb{C}$ by sending $Z_{123}$ into $1$. In matrix form, these two operators are represented by
\[
\overline\partial_{\overline{Z}_1}\rvert_V=\begin{pmatrix} b_1 & 0 \\ 0 & b_2 \end{pmatrix},\quad [\cdot,\cdot]\rvert_{V\times V}=\begin{pmatrix} 0 & b_1-b_2 \\ b_1-b_2 & 0 \end{pmatrix}.
\]
In order for $\overline\partial_{\overline{Z}_1}\rvert_V$ to have kernel, we need to impose the vanishing of either $b_1$ or $b_2$. In the former case, $\mathfrak{g} \cong \mathfrak{h}_3 \oplus \mathfrak{s}_{3.3}$ and $Z_{12}$ is a holomorphic Poisson structure, while, in the latter case, $\mathfrak{g} \cong \mathfrak{s}_{4.7} \oplus \R^2$ and a holomorphic Poisson structure is given by $Z_{13}$.

\begin{proposition}
Let $\mathfrak{g}$ be a  $6$-dimensional  (non-K\"ahler) strongly unimodular almost nilpotent Lie algebra with nilradical $\mathfrak{n}$ satisfying $\dim \mathfrak{n}^1=1$. Then $\mathfrak{g}$ does not admit  generalized K\"ahler structures $(J_+,J_-,g)$ such that $J_+(\mathfrak{n}^1) = \mathfrak{n}^\perp$.
\end{proposition}
\begin{proof}
Proceeding by contradiction, let us assume $J_+(\mathfrak{n}^1) = \mathfrak{n}^\perp$. Then there exists an adapted unitary basis $\{e_l \}$ with structure equations \eqref{streq6}, with $c \neq 0$, $b_1^2+b_2^2 \neq 0$, and $J_+e_1=e_6$, $J_+e_2=e_3$, $J_+e_4=e_5$. The generic skew-symmetric $J_-$ is of the form $J_-=\sum_{k<l} J_{kl}(e^l \otimes e_k - e^k \otimes e_l)$, with $J_{kl} \in \R$. We denote its associated Nijenhuis tensor by $N_-$, regarding it as a $(0,3)$-tensor with the aid of the metric $g$.

Let us focus on the split case, first.  We have to set $J_{12}=-J_{36}$, $J_{13}=J_{26}$, $J_{14}=-J_{56}$, $J_{15}=J_{46}$, $J_{24}=J_{35}$, $J_{25}=-J_{34}$ to obtain $[J_+,J_-]=0$. One can compute
\[
N_-(e_1,e_6,e_1)=c(J_{26}^2+J_{36}^2),\quad N_-(e_4,e_5,e_1)=c(J_{34}^2+J_{35}^2),
\]
which imposes $J_{26}=J_{36}=J_{34}=J_{35}=0$. Now, denoting $H_{\pm}=d^c_{\pm}\omega_{\pm}$, with $\omega_{\pm}=g(J_{\pm}\cdot, \cdot)$, one computes
\[
(H_++H_-)(e_1,e_2,e_3)=-c\left(1+J_{23}^2(J_{16}^2+J_{46}^2+J_{56}^2)\right),
\]
which cannot vanish, leading to a contradiction.

We now proceed with the non-split case, which implies that $J_+$ admits non-zero holomorphic Poisson structures. As we have seen, this implies either $b_1=0$ or $b_2=0$.

If $b_1=0$, holomorphic Poisson structures are multiples of $(e_1-ie_6)\wedge (e_2 -ie_3)$, which forces
\[
[J_+,J_-]g^{-1} \in \R\left< e_{12},e_{13},e_{26},e_{36} \right>.
\]
To obtain this, necessary conditions are given by $J_{15}=J_{46}$, $J_{14}=-J_{56}$, $J_{25}=-J_{34}$, $J_{24}=J_{35}$. Having imposed these, we compute
\begin{alignat*}{2}
&N_-(e_1,e_2,e_3)=cJ_{13}^2,\quad& &N_-(e_1,e_3,e_2)=-cJ_{12}^2, \\ &N_-(e_4,e_5,e_1)=c(J_{34}^2+J_{35}^2),\quad& &N_-(e_4,e_5,e_6)=b_2(J_{46}^2+J_{56}^2),
\end{alignat*}
forcing $J_{12}=J_{13}=J_{34}=J_{35}=J_{46}=J_{56}=0$. Then,
\[
(H_++H_-)(e_1,e_2,e_3)=-c(1+J_{16}^2J_{23}^2) \neq 0,
\]
a contradiction. An analogous discussion holds for the case $b_2=0$.
\end{proof}

\subsection{Case (2)}
First, we shall focus on the case where the subspace $\mathfrak{k}_3$ determined by $(J_+,g)$ is not an abelian ideal and show the following non-existence result:
\begin{proposition}
Let $\mathfrak{g}$ be a $6$-dimensional   (non-K\"ahler) strongly unimodular almost nilpotent Lie algebra with nilradical $\mathfrak{n}$ satisfying $\dim \mathfrak{n}^1=1$. Then $\mathfrak{g}$ does not admit generalized K\"ahler structures $(J_+,J_-,g)$ such that $J_+(\mathfrak{n}^1) \subset \mathfrak{n}$ and $[\mathfrak{k}_3,\mathfrak{k}_3]\neq \{0\}$.
\end{proposition}
\begin{proof}
By contradiction, assume $J_+(\mathfrak{n}^1) \subset \mathfrak{n}$, with $[\mathfrak{k}_3,\mathfrak{k}_3]\neq \{0\}$, so that we may consider a basis $\{e_l \}$ adapted to $(J_+,g)$ and the associated algebraic data $\Xi$ in \eqref{Xinew} satisfying \eqref{Lie1new}--\eqref{Lie3new} and \eqref{SKT_sub2}.
We can then consider the generic skew-symmetric $J_-=\sum_{k<l} J_{kl} (e^l \otimes e_k - e^k \otimes e_l)$, denoting its associated Nijenhuis tensor, regarded as a $(0,3)$-tensor thanks to the metric $g$, by $N_-$. 

Let us start by assuming the generalized K\"ahler structure to be split.
The split condition $[J_+,J_-]=0$ imposes 
\[
J_{23}=-J_{14}, \quad J_{24}=J_{13}, \quad J_{25}=-J_{16},\quad J_{26}=J_{15}, \quad J_{45}=-J_{36},\quad J_{46}=J_{35}.
\]
In what follows, we denote $H_\pm=d^c_{\pm} \omega_{\pm}$, with $\omega_{\pm}=g(J_{\pm} \cdot, \cdot)$.

Now, we start by computing $N_-(e_1,e_2,e_5)=2aJ_{16}^2$, so that either $J_{16}=0$ or $a=0$ (with $J_{16} \neq 0$). If $J_{16}=0$, one computes $N_-(e_5,e_6,e_6)=-q(J_{35}^2+J_{36}^2)$; if we assume $J_{35}=J_{36}=0$, then we have
\[
N_-(e_3,e_4,e_1)=c(J_{34}^2-1),\quad \lVert J_-e_3 \rVert^2=J_{13}^2+J_{14}^2+J_{34}^2,
\]
which force $J_{13}=J_{14}=0$, $J_{34}=\pm 1$, so that
\[
(H_++H_-)(e_1,e_3,e_4)=-c(1+J_{12}^2+J_{15}^2)
\]
cannot vanish. Instead, if we assume $q =0$ (with $J_{35}^2 +J_{36}^2 \neq 0$), we compute
\[
N_-(e_2,e_3,e_6)=aJ_{15}J_{36},\quad N_-(e_2,e_4,e_6)=aJ_{15}J_{35},
\]
whose vanishing forces $J_{15}=0$, since $a=0$ would make $\mathfrak{g}$ nilpotent. To obtain a contradiction, we only need to compute
\[
N_-(e_5,e_6,e_5)=a(J_{56}^2-1),\quad \lVert J_-e_5 \rVert^2=J_{35}^2+J_{36}^2 + J_{56}^2.
\]

Now, we can assume $a=0$, $J_{16} \neq 0$. In order for $\mathfrak{g}$ to be non-nilpotent, we also have to require $q \neq 0$. First, we have $N_-(e_1,e_2,e_2)=v_2(J_{15}^2+J_{16}^2)$, which forces $v_2=0$. We then look at
\[
N_-(e_3,e_4,e_5)=-2J_{15}(\alpha(e_3)J_{35}-\alpha(e_4) J_{36}).
\]
A first possibility is given by $J_{15}=0$: in this case, we have
\[
N_-(e_1,e_3,e_5)=qJ_{16}J_{36},\quad N_-(e_1,e_4,e_5)=qJ_{16}J_{35},
\]
which force $J_{35}=J_{36}=0$. Then,
\[
N_-(e_3,e_4,e_1)=c(J_{34}^2-1),\quad \lVert J_- e_3 \rVert^2=J_{13}^2+J_{14}^2+J_{34}^2
\]
imply $J_{13}=J_{14}=0$, $J_{34}^2=1$. Finally, computing
\[
N_-(e_1,e_2,e_1)=\tfrac{1}{c} J_{16}^2 \lVert \alpha \rVert^2
\]
yields $\alpha=0$, so that
\[
(H_++H_-)(e_1,e_3,e_4)=-c(1+J_{12}^2+J_{16}^2)
\]
is never zero. A second possibility is given by $J_{15} \neq 0$, $\alpha(e_3)=0$ and either $\alpha(e_4)=0$ or $J_{36}=0$ (with $\alpha(e_4) \neq 0$): in the former case, one has $N_-(e_5,e_6,e_6)=-q(J_{35}^2+J_{36}^2)$, forcing $J_{35}=J_{36}=0$. Now, one has $N_-(e_1,e_2,e_1)=c(J_{13}^2+J_{14}^2)$, yielding $J_{13}=J_{14}=0$. The contradiction follows from 
\[
(H_++H_-)(e_1,e_3,e_4)=-c(1+J_{34}^2(J_{12}^2+J_{15}^2+J_{16}^2)).
\]
The third and final possibility is given by $J_{15} \neq 0$, $\alpha(e_3) \neq 0$, $J_{35}=\frac{\alpha(e_4)}{\alpha(e_3)} J_{36}$. In this case, we have $N_-(e_3,e_4,e_4)=-\frac{1}{c \alpha(e_3)} q J_{36}^2 \lVert \alpha \rVert^2$, from which $J_{36}=0$ follows. A contradiction can be obtained by considering
\[
N_-(e_5,e_6,e_1)=\tfrac{1}{c} \lVert \alpha \rVert^2(J_{56}^2 -1), \quad \lVert J_-e_5 \rVert^2=J_{15}^2+J_{16}^2+J_{56}^2.
\]
This concludes the split case.

Now, to work on the non-split case, we have to determine under what conditions the starting SKT structure $(J_+,g)$ on $\mathfrak{g}$ admits holomorphic Poisson structures. With respect to $J_+$, a basis for $\mathfrak{g}^{1,0}$ is given by
\begin{equation} \label{10vectors}
Z_1=e_1-ie_2,\quad Z_2=e_3-ie_4,\quad Z_3=e_5-ie_6.
\end{equation}
With respect to the basis $\{Z_1,Z_2,Z_3,\overline{Z}_1,\overline{Z}_2,\overline{Z}_3\}$ for $\mathfrak{g} \otimes \C$, the non-zero brackets, up to exchanging the terms and conjugating, are given by
\begin{gather*}
[Z_1,Z_3]=-iaZ_1,\quad [Z_2,Z_3]=qZ_2,\quad [Z_1,\overline{Z}_3]=-ia\overline{Z}_1,\\ [Z_2,\overline{Z}_2]=-ic(Z_1+\overline{Z}_1), \quad [Z_2,\overline{Z}_3]=-i(\alpha(e_3)-i\alpha(e_4))(Z_1+\overline{Z}_1)-qZ_2,\\
\begin{aligned}
[Z_3,\overline{Z}_3]=&-\tfrac{1}{c}(i\lVert \alpha \rVert^2-cv_2)Z_1 + \tfrac{i}{c}(a+iq)(\alpha(e_3)+i\alpha(e_4))Z_2-iaZ_3 \\
                     &-\tfrac{1}{c}(i\lVert \alpha \rVert^2 + cv_2) \overline{Z}_1 + \tfrac{i}{c}(a-iq)(\alpha(e_3)-i\alpha(e_4))\overline{Z}_2 - ia \overline{Z}_3.
\end{aligned}
\end{gather*}
Now, an explicit computation using the definition shows that the endomorphism $\overline{\partial}_{\overline{Z}_l}$ of $\mathfrak{g}^{2,0}$, $l=1,2,3$, can be represented in matrix form, with respect to the basis $\{Z_{12},Z_{13},Z_{23}\}$, as
\footnotesize
\begin{gather*}
\overline{\partial}_{\overline{Z}_1}=\begin{pmatrix}
0 & 0 & ia \\ 0 & 0 & 0 \\ 0 & 0 & 0
\end{pmatrix}, \qquad
\overline{\partial}_{\overline{Z}_2}=\begin{pmatrix}
0 & 0 & i(\alpha(e_3)+i\alpha(e_4)) \\ 0 & 0 & -ic \\ 0 & 0 & 0
\end{pmatrix}, \\
\overline{\partial}_{\overline{Z}_3}=\begin{pmatrix}
-q & \tfrac{i}{c}(a+iq)(\alpha(e_3)+i\alpha(e_4)) & \tfrac{1}{c}(i\lVert \alpha \rVert^2 - cv_2) \\ 0 & -ia & -i(\alpha(e_3)-i\alpha(e_4)) \\ 0 & 0 & -i(a-iq)
\end{pmatrix}.
\end{gather*}
\normalsize
One can easily see that the intersection of the kernels of the three endomorphisms is non-trivial if and only if either $a=0$ or $q=0$ (with $a^2+q^2 \neq 0$ to have a non-nilpotent Lie algebra): such intersection is spanned by $(\alpha(e_3)+i\alpha(e_4))Z_{12}-Z_{13}$, for $a=0$, or by $Z_{12}$, for $q=0$. Such holomorphic $(2,0)$-vectors are then automatically Poisson, as shown by an explicit computation.

Let us consider the case $a=0$. We simplify the results of the previous paragraph, by just using that holomorphic Poisson structures lie in $\mathbb{C} \left<Z_{12},Z_{13}\right>$. This means that
\[
[J_+,J_-]g^{-1} \in \R\left<e_{13},e_{14},e_{15},e_{16},e_{23},e_{24},e_{25},e_{26}\right>.
\]
To achieve this, we only need to set
\[
J_{45}=-J_{36},\quad J_{46}=J_{35}.
\]
We can compute $N_-(e_3,e_4,e_5)=-2J_{15}(\alpha(e_3)J_{35}-\alpha(e_4) J_{36})$, so that, similarly to our earlier discussion, we have three possibilities to discuss. If $J_{15}=0$, we focus on
\[
N_-(e_1,e_3,e_5)=qJ_{16}J_{36},\quad qJ_{16}J_{36}.
\]
Now, either $J_{35}=J_{36}=0$, in which case
\begin{equation} \label{N341}
N_-(e_3,e_4,e_1)=c(J_{34}^2-1),\quad \lVert J_-e_3 \rVert^2=J_{13}^2+J_{23}^2+J_{34}^2,\quad \lVert J_-e_4 \rVert^2=J_{14}^2+J_{24}^2+J_{34}^2
\end{equation}
imply $J_{13}=J_{23}=J_{14}=J_{24}=0$, $J_{34}^2=1$, so that
\[
(H_++H_-)(e_1,e_3,e_4)=-c(1+J_{12}^2+J_{16}^2)
\]
can never vanish,
or $J_{16}=0$, with $J_{35}^2+J_{36}^2 \neq 0$: in this case, a contradiction is given by
\begin{equation} \label{N566}
N_-(e_5,e_6,e_6)=-q(J_{35}^2+J_{36}^2).
\end{equation}
Instead, if $\alpha(e_3)=0$, with $J_{15} \neq 0$, either $J_{36}=0$, in which case $N_-(e_5,e_6,e_6)=-qJ_{35}^2$ implies $J_{35}=0$, and now \eqref{N341} forces $J_{13}=J_{23}=J_{14}=J_{24}=0$, $J_{34}^2=1$, so that a contradiction follows from
\begin{equation} \label{H134}
(H_++H_-)(e_1,e_3,e_4)=-c(1+J_{12}^2+J_{15}^2+J_{16}^2),
\end{equation}
or $\alpha(e_4)=0$ (with $J_{36} \neq 0$): here, we have \eqref{N566}, which can never vanish. Lastly, if $\alpha(e_3) \neq 0$ and $J_{35}=\tfrac{\alpha(e_4)}{\alpha(e_3)} J_{36}$, we have 
\[
N_-(e_3,e_4,e_4)=-\tfrac{1}{c\alpha(e_3)}q\lVert \alpha \rVert^2 q,
\]
implying $J_{36}=0$, so that, again, \eqref{N341} forces $J_{13}=J_{23}=J_{14}=J_{24}=0$, $J_{34}^2=1$ and \eqref{H134} provides a contradiction.

We can then move on to the case $q=0$ (with $a \neq 0$). Here, holomorphic Poisson structures are spanned by $Z_{12}$, so that we have
\[
[J_+,J_-]g^{-1} \in  \R \left< e_{13},e_{14},e_{23},e_{24} \right>,
\]
meaning that we have to set
\[
J_{25}=-J_{16},\quad J_{26}=J_{15}, \quad J_{45}=-J_{36},\quad J_{46}=J_{35}.
\]
Actually, we also have $J_{16}=0$, since $N_-(e_1,e_2,e_5)=2aJ_{16}^2$. Looking at
\[
N_-(e_2,e_3,e_6)=aJ_{15}J_{36},\quad N_-(e_2,e_4,e_6)=aJ_{15}J_{35},
\]
we have either $J_{35}=J_{36}=0$, yielding \eqref{N341}, implying $J_{13}=J_{23}=J_{14}=J_{24}=0$, $J_{34}^2=1$ and consequently
\[
(H_+,H_-)(e_1,e_3,e_4)=-c(1+J_{12}^2+J_{15}^2),
\]
or $J_{15}=0$ (with $J_{35}^2 + J_{36}^2 \neq 0$), which immediately yields a contradiction, since $N_-(e_3,e_4,e_5)=a(J_{35}^2+J_{36}^2)$. This concludes the proof of the proposition.
\end{proof}

We now move one to the case where $\mathfrak{k}_3$ is an abelian ideal.

\begin{proposition}
Let $\mathfrak{g}$ be a non-K\"ahler $6$-dimensional strongly unimodular almost nilpotent Lie algebra with nilradical $\mathfrak{n}$ satisfying $\dim \mathfrak{n}^1=1$. Then $\mathfrak{g}$ does not admit generalized K\"ahler structures $(J_+,J_-,g)$ such that $(J_+,g)$ satisfies $J_+(\mathfrak{n}^1) \subset \mathfrak{n}$ and $\mathfrak{k}_3$ being an abelian ideal.
\end{proposition}
\begin{proof}
We begin by analyzing the split case. By contradiction, assume $J_+ \mathfrak{n}^1 \subset \mathfrak{n}$ and $\mathfrak{k}_3$ to be an abelian ideal. Then, by Proposition \ref{sub_SKTprop}, there exists an adapted unitary basis $\{e_l\}$, $J_+e_{2j-1}=e_{2j}$, $j=1,2,3$, having structure equations
\begin{alignat*}{3}
de^1&=-ae^{25}+v_1e^{56},\quad &
de^2&=-ae^{26}+v_2e^{56},\quad &
de^3&=b e^{46} + v_3 e^{56}, \\
de^4&=-b e^{36} + v_4 e^{56}, \quad &
de^5&=ae^{56}, \quad &
de^6&=0,
\end{alignat*}
for some $a \in \R-\{0\}$, $b,v_l \in \R$, $l=1,2,3,4$. Consider the generic skew-symmetric $J_-=\sum_{k<l}J_{kl}(e^l \otimes e_k - e^k \otimes e_l)$. To have $[J_+,J_-]=0$, we need to impose
\[
J_{23}=-J_{14},\quad J_{24}=J_{13},\quad J_{25}=-J_{16},\quad J_{26}=J_{15},\quad J_{45}=-J_{36},\quad J_{46}=J_{35}.
\]
Denoting by $N_-$ the Nijenhuis tensor associated with $J_-$, regarded as a $(0,3)$-tensor with the aid of the metric $g$, we can compute
\[
N_-(e_3,e_4,e_5)=a(J_{35}^2+J_{36}^2),
\]
implying $J_{35}=J_{36}=0$. Then,
\[
N_-(e_1,e_2,e_5)=2aJ_{16}^2,\quad N_-(e_3,e_6,e_3)=-aJ_{14}^2,\quad N_-(e_4,e_6,e_4)=-aJ_{13}^2,
\]
by which we have $J_{13}=J_{14}=J_{16}=0$. Denoting $H_\pm=d^c_\pm \omega_\pm$, we then have
\[
(H_++H_-)(X,e_5,e_6)=-g(v,X)(1+J_{34}^2J_{56}^2),\quad X \in \mathfrak{k}_3,
\]
so that we have to set $v_3=v_4=0$. We now have
\[
N_-(e_1,e_2,e_1)=v_1 J_{15}^2,
\]
so that we can discuss the two cases $v_1=0$ and $J_{15}=0$ (with $v_1 \neq 0$): in the latter case we would have
\[
(H_++H_-)(e_1,e_5,e_6)=-v_1(1+J_{12}^2J_{56}^2) \neq 0,
\]
so that we have to stick with $v_1=0$. We compute
\[
N_-(e_1,e_5,e_5)=aJ_{15}(J_{12}+J_{56}).
\]
Again, we discuss the two cases $J_{15}=0$ and $J_{56}=-J_{12}$ (with $J_{15} \neq 0$): in the former case,
\[
(H_++H_-)(e_2,e_5,e_6)=-v_2(J_{12}^2J_{56}^2+1)
\]
implies $v_2=0$, but then, the following two expressions cannot vanish simultaneously:
\[
N_-(e_1,e_5,e_2)=aJ_{12}(J_{12}-J_{56}),\quad (H_++H_-)(e_1,e_2,e_5)=-a(1+J_{12}^3J_{56}).
\]
Instead, in the case $J_{56}=-J_{12}$ (with $J_{15} \neq 0$),
\[
N_-(e_2,e_5,e_6)=-J_{15}(2aJ_{12}-v_2J_{15})
\]
forces $v_2=\frac{2aJ_{12}}{J_{15}}$. But then
\[
(H_++H_-)(e_1,e_2,e_5)=-a(1+(J_{12}^2+J_{15}^2)^2)
\]
cannot vanish.

Let us now work in the non-split case, which means we first have to study holomorphic Poisson structures with respect to $J_+$. A basis of $(1,0)$-vectors with respect to $J_+$ is thus provided by
\[
Z_1=e_1-ie_2,\qquad Z_2=e_3-ie_4,\qquad Z_3=e_5-ie_6.
\]
With respect to the basis $\{Z_1,Z_2,Z_3,\overline{Z}_1,\overline{Z}_2,\overline{Z}_3\}$ for $\mathfrak{g} \otimes \C$, the non-zero brackets, up to exchanging the terms, are given by
\begin{gather*}
\begin{align*}
&[Z_1,Z_3]=-iaZ_1, \quad \; & 
&[Z_2,Z_3]=bZ_2, \quad \; &
&[Z_1,\overline{Z}_3]=- i a Z_1, \\
&[Z_2,\overline{Z}_3]=-bZ_2, \quad \; &
&[Z_3,\overline{Z}_1]=-iaZ_1, \quad \; &
&[Z_3,\overline{Z}_2]=b\overline{Z}_2,
\end{align*} \\
\begin{align*}
&[Z_3,\overline{Z}_3]=-i \left( (v_1+iv_2)Z_1+(v_3+iv_4)Z_2+aZ_3+(v_1-iv_2)\overline{Z}_1+(v_3-iv_4)\overline{Z}_2 + a\overline{Z}_3 \right),
\end{align*} \\
[\overline{Z}_1, \overline{Z}_3]=ia\overline{Z}_1, \qquad [\overline{Z}_2,\overline{Z}_3]=b\overline{Z}_2,
\end{gather*}
By a direct computation, we then get $\overline\partial_{\overline{Z}_2}Z_{jk}=0$ for all $j,k$, while
\[
\overline\partial_{\overline{Z}_1}Z_{12}=\overline\partial_{\overline{Z}_1}Z_{13}=0,\quad \overline\partial_{\overline{Z}_1}Z_{23}=-iaZ_{12},
\]
so that potential holomorphic Poisson structures must lie in $\mathbb{C}\left<Z_{12},Z_{13}\right>$. Then one computes
\[
\overline\partial_{\overline{Z}_3}Z_{12}=b Z_{12},\quad \overline\partial_{\overline{Z}_3}Z_{13}=i(v_1+iv_4)Z_{12}+iaZ_{13}.
\]
It follows that holomorphic $(2,0)$-vectors exist and that they must be multiples of $Z_{12}=(e_1-ie_2)\wedge(e_3-ie_4)$. These are also Poisson, since $[Z_{12},Z_{12}]=-2[Z_1,Z_2]\wedge Z_1 \wedge Z_2=0$.

We can then consider the generic skew-symmetric $J_-=\sum_{k<l}J_{kl}(e^l \otimes e_k - e^k \otimes e_l)$. By the above discussion, we must have
\[
[J_+,J_-]g^{-1} \in \R\left< e_{13},e_{14},e_{23},e_{24} \right>,
\]
which forces $J_{25}=-J_{16}$, $J_{26}=J_{15}$, $J_{45}=-J_{36}$ and $J_{46}=J_{35}$.
One can then compute
\[
N_-(e_3,e_4,e_5)=a(J_{35}^2+J_{36}^2),
\]
which implies $J_{35}=J_{36}=0$.

One has
\[
N_-(e_3,e_6,e_3)=-aJ_{23}^2,\quad N_-(e_4,e_6,e_4)=-aJ_{24}^2,
\]
so that $J_{23}=J_{24}=0$. Then
\[
\lVert J_- e_1 \rVert^2 - \lVert J_- e_2 \rVert^2=J_{13}^2+J_{14}^2,
\]
whose vanishing implies $J_{13}=J_{14}=0$, since $(J_-)^2=-\text{Id}$. But now one has $[J_+,J_-]=0$ and the sought generalized K\"ahler structure must be split.
\end{proof}

In real dimension six, the only remaining case is thus represented by generalized K\"ahler structures $(J_+,J_-,g)$ whose underlying SKT structures are described by case (3). In higher dimension, too, nothing is known about the existence of generalized K\"ahler structures on almost nilpotent Lie algebras with nilradical having one-dimensional commutator. 

On this subject, we recall that, by \cite{FP1}, all generalized K\"ahler structures on $6$-dimensional almost abelian Lie algebras  are split. In dimension higher than six, this is no longer true, as the next theorem shows.

\begin{theorem} \label{gkexamples}
There exist non-K\"ahler compact solvmanifolds of complex dimension $n \geq 4$ admitting non-split generalized K\"ahler structures.
\end{theorem}
\begin{proof}
Consider the $2n$-dimensional unimodular almost abelian Lie groups $A_{2n}^{p,q}$, $n \geq 4$, $p,q \in \R-\{0\}$, whose Lie algebra $\mathfrak{a}_{2n}^{p,q}$ is endowed with a basis $\{f_1,\ldots,f_{2n}\}$ with structure equations
\begin{gather*}
df^1=f^1 \wedge f^{2n},\quad df^2=-\tfrac{1}{2}f^2 \wedge f^{2n} + p f^3 \wedge f^{2n}, \quad df^3=-p f^2 \wedge f^{2n} -\tfrac{1}{2} f^3 \wedge f^{2n}, \quad df^{2n}=0,\\
df^{2l+2}= q f^{2l+3} \wedge f^{2n}, \quad df^{2l+3}= - q f^{2l+2} \wedge f^{2n},\quad 1 \leq l \leq n-2,
\end{gather*}
On $A_{2n}^{p,q}$ we can consider the left-invariant generalized K\"ahler structure $(J_{\pm},g)$ provided by
\begin{gather*}
J_\pm f_1=f_{2n},\quad J_\pm f_2=\pm f_3, \\
J_+f_4=f_5,\quad J_+f_6=-f_7,\quad J_-f_4=-f_7,\quad J_-f_5=f_6, \\
J_\pm f_{2l+6}=f_{2l+7}, \; 1 \leq l \leq n-4, \\
g=\sum_{l=1}^{2n} (f^l)^2.
\end{gather*}
Its torsion $H_+=d^c_+\omega_+ = - d^c_- \omega_-$ is provided by $H_+=-f^{123}.$

The generalized K\"ahler structure is non-split, since
\[
[J_+,J_-]g^{-1}=-2(f_{46}+f_{57}) \neq 0.
\]
We claim that, for some values of $p$ and $q$, $A_{2n}^{p,q}$ admits cocompact lattices. By \cite{Bock}, an almost abelian Lie algebra $\R^{2n-1} \rtimes_D \R$, $D \in \mathfrak{gl}_{2n-1}$,  admits cocompact lattices if and only if there exists a basis for $\mathfrak{n}$ and some $t_0 \in \R-\{0\}$ such that the matrix associated with $\text{exp}(t_0 D)$ has integer entries.
In our case, in the basis $\{f_1,\ldots,f_{2n-1}\}$ for the nilradical of $\mathfrak{a}_{2n}^{p,q}$, we have
\[
D=\text{diag}(1,C_p - \tfrac{1}{2} \text{Id}_2,C_q,\ldots,C_q).
\]
By \cite{AO}, there exist $p,t^\prime \in \R-\{0\}$ such that the $3 \times 3$-matrix
$\text{exp}\big(t^\prime\text{diag}(1,C_p-\tfrac{1}{2}\text{Id}_2)\big)$
is similar to an integer matrix, say $E \in \text{GL}_3$. Considering one of these values for $p$, we set $t_0=t^\prime$ and $q=\frac{2\pi}{t^\prime}$. It then follows that $\text{exp}(t_0 D)$ is similar to the integer matrix
$\text{diag}(E,1,\ldots,1)$,
concluding the proof.
\end{proof}

\section{Appendix}

Borrowing the notation of  \cite{SW},  the next two tables list all   $6$-dimensional  strongly unimodular almost nilpotent Lie algebras with nilradical having one-dimensional commutator. For each of them, we also indicate whether they admit complex structures, dividing the two cases $J \mathfrak{n}^1 \not\subset \mathfrak{n}$ and  $J \mathfrak{n}^1 \subset \mathfrak{n}$. We also indicate whether they admit SKT structures or balanced structures with $\hat\theta=0,\frac{\pi}{2}$, following Theorems \ref{class_SKT_perp}, \ref{class_SKT_sub}, \ref{class_bal_perp} and \ref{class_bal_sub}.

\begin{table}[H]
\begin{center}
\addtolength{\leftskip} {-2cm}
\addtolength{\rightskip}{-2cm}
\scalebox{0.58}{
\begin{tabular}{|l|l|l|l|l|l|l|l|}
\hline \xrowht{10pt}
\multirow{2}{*}{\vspace{-5pt}Name} & \multirow{2}{*}{\vspace{-5pt}Structure equations} & \multicolumn{2}{c|}{Complex structures} & \multicolumn{2}{c|}{SKT structures} & \multicolumn{2}{c|}{Balanced structures}\\
\cline{3-8} \xrowht{10pt}
&& $J \mathfrak{n}^1\not\subset \mathfrak{n}$ & \multicolumn{1}{c|}{$J \mathfrak{n}^1 \subset \mathfrak{n}$}
& $J \mathfrak{n}^1= \nperp$ & \multicolumn{1}{c|}{$J \mathfrak{n}^1 \subset \mathfrak{n}$}
& $J \mathfrak{n}^1= \nperp$ & \multicolumn{1}{c|}{$J \mathfrak{n}^1 \subset \mathfrak{n}$}
 \\
\hline \hline \xrowht{20pt}
$\mathfrak{h}_3 \oplus \frs{3.1}^{-1}$ & $(f^{23},0,0,f^{46},-f^{56},0)$ & --  & -- & -- & -- & -- & --
   \\ \hline \xrowht{20pt}
$\mathfrak{h}_3 \oplus \frs{3.3}^{0}$ & $(f^{23},0,0,f^{56},-f^{46},0)$ & \cmark & -- & \cmark & -- & -- & --
   \\ \hline \xrowht{20pt}
$\frs{4.6} \oplus \R^2$ & $(f^{23},f^{26},-f^{36},0,0,0)$ & -- & \cmark & -- & \cmark & -- & --
   \\ \hline \xrowht{20pt}
$\frs{4.7} \oplus \R^2$ & $(f^{23},f^{36},-f^{26},0,0,0)$ & \cmark & \cmark & \cmark & \cmark & -- & --
   \\ \hline \xrowht{20pt}
$\frs{5.15} \oplus \R$ & $(f^{23}+f^{46},f^{26},-f^{36},0,0,0)$  & -- & -- & -- & -- & -- & --
   \\ \hline \xrowht{20pt}
$\frs{5.16} \oplus \R$ & $(f^{23}+f^{46},f^{36},-f^{26},0,0,0)$  & -- & \cmark & -- & -- & -- & \cmark
\\ \hline \xrowht{20pt}
$\frs{6.24}$ & $(f^{23},f^{26},-f^{36},f^{56},0,0)$ & -- & -- & -- & -- & -- & --
\\ \hline \xrowht{20pt}
$\frs{6.25}$ & $(f^{23},f^{36},-f^{26},0,f^{46},0)$ & -- & \cmark & -- & \cmark & -- & --
\\ \hline \xrowht{20pt}
$\frs{6.30}$ & $(f^{23}+f^{56},f^{26},-f^{36},0,f^{46},0)$ & -- & -- & -- & -- & --& --
\\ \hline \xrowht{20pt}
$\frs{6.31}$ & $(f^{23}+f^{56},f^{36},-f^{26},0,f^{46},0)$ & -- & -- & -- & -- & --& --
\\ \hline \xrowht{20pt}
$\frs{6.32}^{-1}$ & $(f^{23},f^{36},0,f^{46},-f^{56},0)$ & -- & -- & -- & -- & --& --
\\ \hline \xrowht{20pt}
$\frs{6.34}^0$ & $(f^{23},f^{36},0,f^{56},-f^{46},0)$ & -- & -- & -- & -- & -- & --
\\ \hline \xrowht{20pt}
$\frs{6.43}$ & $(f^{23},f^{26},-f^{36},f^{26}+f^{46},f^{36}-f^{56},0)$ & -- & -- & -- & -- & -- & --
\\ \hline \xrowht{20pt}
$\frs{6.44}$ & $(f^{23},f^{36},-f^{26},f^{26}+f^{56},f^{36}-f^{46},0)$ & \cmark & -- & -- & -- & -- & --
\\ \hline \xrowht{20pt}
$\frs{6.45}^{a,-1}$ & $(f^{23},af^{26},-af^{36},f^{46},-f^{56},0)$, $a>0$ & -- & -- & -- & -- & -- & --
\\ \hline \xrowht{20pt}
$\frs{6.46}^{a,-a}$ & $(f^{23},f^{36},-f^{26},af^{46},-af^{56},0)$, $a \neq 0$ & -- & -- & -- & -- & -- & --
\\ \hline \xrowht{20pt}
$\frs{6.47}^{-1}$ & $(f^{23},-f^{26},f^{36},f^{36}+f^{46},-f^{56},0)$ & -- & -- & -- & -- & -- & --
\\ \hline \xrowht{20pt}
$\frs{6.51}^{a,0}$ & $(f^{23},af^{26},-af^{36},f^{56},-f^{46},0)$, $a>0$ & -- & \cmark & -- & \cmark & -- & --
\\ \hline \xrowht{20pt}
$\frs{6.52}^{0,b}$ & $(f^{23},f^{36},-f^{26},bf^{56},-bf^{46},0)$, $b>0$ & \cmark & -- & \cmark & -- & -- & --
\\ \hline 
\end{tabular}}
\caption{$6$-dimensional strongly unimodular almost nilpotent Lie algebras with nilradical $\mathfrak{n}\cong \mathfrak{h}_3 \oplus \R^2$.} \label{table-h3}
\end{center}
\end{table}

\begin{table}[H]
\begin{center}
\addtolength{\leftskip} {-2cm}
\addtolength{\rightskip}{-2cm}
\scalebox{0.58}{
\begin{tabular}{|l|l|l|l|l|l|l|l|}
\hline \xrowht{10pt}
\multirow{2}{*}{\vspace{-5pt}Name} & \multirow{2}{*}{\vspace{-5pt}Structure equations} & \multicolumn{2}{c|}{Complex structures} & \multicolumn{2}{c|}{SKT structures} & \multicolumn{2}{c|}{Balanced structures}\\
\cline{3-8} \xrowht{10pt}
&& $J\mathfrak{n}^1\not\subset \mathfrak{n}$ & \multicolumn{1}{c|}{$J\mathfrak{n}^1 \subset \mathfrak{n}$}
& $J \mathfrak{n}^1= \nperp$ & \multicolumn{1}{c|}{$J \mathfrak{n}^1 \subset \mathfrak{n}$}
& $J \mathfrak{n}^1= \nperp$ & \multicolumn{1}{c|}{$J \mathfrak{n}^1 \subset \mathfrak{n}$} \\
\hline \hline \xrowht{20pt}
$\frs{6.158}$ & $(f^{24}+f^{35},0,f^{36},0,-f^{56},0)$ & -- & \cmark & -- & \cmark & -- & --
\\ \hline \xrowht{20pt}
$\frs{6.159}$ & $(f^{24}+f^{35},0,f^{56},0,-f^{36},0)$ & \cmark & -- & -- & -- & \cmark & --
\\ \hline \xrowht{20pt}
$\frs{6.160}$ & $(f^{24}+f^{35},f^{46},f^{36},0,-f^{56},0)$ & -- & -- & -- & --  & -- & --
\\ \hline \xrowht{20pt}
$\frs{6.161}^{\eps}$ & $(f^{24}+f^{35},\eps f^{46},f^{56},0,f^{36},0)$, $\varepsilon=\pm 1$ & -- & -- & -- & -- & -- & --
\\ \hline \xrowht{20pt}
$\frs{6.162}^a$ & $(f^{24}+f^{35},f^{26},af^{36},-f^{46},-af^{56},0)$, $0<a \leq 1$ & $a=1$ & -- & -- & -- & $a=1$ & --
\\ \hline \xrowht{20pt}
$\frs{6.163}$ & $(f^{24}+f^{35},f^{26},f^{26}+f^{36},-f^{46}-f^{56},-f^{56},0)$ & -- & -- & -- & -- & -- & --
\\ \hline \xrowht{20pt}
$\frs{6.164}^a$ & $(f^{24}+f^{35},af^{26},f^{56},-af^{46},-f^{36},0)$, $a>0$ & -- & \cmark & -- & \cmark & -- & --
\\ \hline \xrowht{20pt}
$\frs{6.165}^a$ & $(f^{24}+f^{35},af^{26}+f^{36},-f^{26}+af^{36},-af^{46}+f^{56},-f^{46}-af^{56},0)$, $a>0$ & \cmark & -- & -- & -- & \cmark & --
\\ \hline \xrowht{20pt}
$\frs{6.166}^a$ & $(f^{24}+f^{35},f^{46},af^{56},-f^{26},-af^{36},0)$, $0<|a| \leq 1$ & \cmark  & -- & -- & -- & \cmark & --
\\ \hline \xrowht{20pt}
$\frs{6.167}$ & $(f^{24}+f^{35},f^{36},-f^{26},f^{26}+f^{56},f^{36}-f^{46},0)$ & \cmark  & -- & -- & -- & \cmark & --
\\ \hline
\end{tabular}}
\caption{$6$-dimensional strongly unimodular almost nilpotent Lie algebras with nilradical $\mathfrak{n}\cong \mathfrak{h}_5$.} \label{table-h5}
\end{center}
\end{table}

\end{document}